\documentclass[12pt]{amsart}
\usepackage{txfonts}      
\usepackage{amssymb}
\usepackage{eucal}
\usepackage{amsmath}
\usepackage{amscd}
\usepackage{xcolor}
\usepackage{multicol}
\usepackage[all]{xy}           
\usepackage{graphicx}
\usepackage{color}
\usepackage{colordvi}
\usepackage{xspace}
\usepackage{tikz}
\usepackage{makecell}
\usepackage{appendix}
\usepackage{amsthm}
\usepackage[misc]{ifsym}
\usepackage{mathrsfs} 
\usepackage{xypic}
\usepackage{extarrows}

\usepackage{ifpdf}
\ifpdf
\usepackage[colorlinks,final,backref=page,hyperindex]{hyperref}
\else
\usepackage[colorlinks,final,backref=page,hyperindex,hypertex]{hyperref}
\fi

\usepackage[active]{srcltx} 

\begin{document}
\newtheorem{thm}{Theorem}[section]
\newtheorem{lem}[thm]{Lemma}
\newtheorem{cor}[thm]{Corollary}
\newtheorem{pro}[thm]{Proposition}
\theoremstyle{definition}
\newtheorem{defi}[thm]{Definition}
\newtheorem{ex}[thm]{Example}
\newtheorem{rmk}[thm]{Remark}
\newtheorem{pdef}[thm]{Proposition-Definition}
\newtheorem{condition}[thm]{Condition}
\renewcommand{\labelenumi}{{\rm(\alph{enumi})}}
\renewcommand{\theenumi}{\alph{enumi}}
\baselineskip=14pt

\newcommand {\emptycomment}[1]{} 

\newcommand{\nc}{\newcommand}
\newcommand{\delete}[1]{}

\nc{\todo}[1]{\tred{To do:} #1}

\nc{\tred}[1]{\textcolor{red}{#1}}
\nc{\tblue}[1]{\textcolor{blue}{#1}}
\nc{\tgreen}[1]{\textcolor{green}{#1}}
\nc{\tpurple}[1]{\textcolor{purple}{#1}}
\nc{\tgray}[1]{\textcolor{gray}{#1}}
\nc{\torg}[1]{\textcolor{orange}{#1}}
\nc{\tmag}[1]{\textcolor{magenta}}
\nc{\btred}[1]{\textcolor{red}{\bf #1}}
\nc{\btblue}[1]{\textcolor{blue}{\bf #1}}
\nc{\btgreen}[1]{\textcolor{green}{\bf #1}}
\nc{\btpurple}[1]{\textcolor{purple}{\bf #1}}

    \nc{\mlabel}[1]{\label{#1}}  
    \nc{\mcite}[1]{\cite{#1}}  
    \nc{\mref}[1]{\ref{#1}}  
    \nc{\meqref}[1]{\eqref{#1}}  
    \nc{\mbibitem}[1]{\bibitem{#1}} 

\delete{
    \nc{\mlabel}[1]{\label{#1}  
        { {\small\tgreen{\tt{{\ }(#1)}}}}}
    \nc{\mcite}[1]{\cite{#1}{\small{\tt{{\ }(#1)}}}}  
    \nc{\mref}[1]{\ref{#1}{\small{\tred{\tt{{\ }(#1)}}}}}  
    \nc{\meqref}[1]{\eqref{#1}{{\tt{{\ }(#1)}}}}  
    \nc{\mbibitem}[1]{\bibitem[\bf #1]{#1}} 
}

\nc{\cm}[1]{\textcolor{red}{Chengming:#1}}
\nc{\yy}[1]{\textcolor{blue}{Yanyong: #1}}
\nc{\zy}[1]{\textcolor{yellow}{Zhongyin: #1}}
\nc{\li}[1]{\textcolor{purple}{#1}}
\nc{\lir}[1]{\textcolor{purple}{Li:#1}}


\nc{\tforall}{\ \ \text{for all }}
\nc{\hatot}{\,\widehat{\otimes} \,}
\nc{\complete}{completed\xspace}
\nc{\wdhat}[1]{\widehat{#1}}

\nc{\ts}{\mathfrak{p}}
\nc{\mts}{c_{(i)}\ot d_{(j)}}

\nc{\NA}{{\bf NA}}
\nc{\LA}{{\bf Lie}}
\nc{\CLA}{{\bf CLA}}

\nc{\cybe}{CYBE\xspace}
\nc{\nybe}{NYBE\xspace}
\nc{\ccybe}{CCYBE\xspace}

\nc{\ndend}{pre-Novikov\xspace}
\nc{\calb}{\mathcal{B}}
\nc{\rk}{\mathrm{r}}
\newcommand{\g}{\mathfrak g}
\newcommand{\h}{\mathfrak h}
\newcommand{\pf}{\noindent{$Proof$.}\ }
\newcommand{\frkg}{\mathfrak g}
\newcommand{\frkh}{\mathfrak h}
\newcommand{\Id}{\rm{Id}}
\newcommand{\gl}{\mathfrak {gl}}
\newcommand{\ad}{\mathrm{ad}}
\newcommand{\add}{\frka\frkd}
\newcommand{\frka}{\mathfrak a}
\newcommand{\frkb}{\mathfrak b}
\newcommand{\frkc}{\mathfrak c}
\newcommand{\frkd}{\mathfrak d}
\newcommand {\comment}[1]{{\marginpar{*}\scriptsize\textbf{Comments:} #1}}


\nc{\disp}[1]{\displaystyle{#1}}
\nc{\bin}[2]{ (_{\stackrel{\scs{#1}}{\scs{#2}}})}  
\nc{\binc}[2]{ \left (\!\! \begin{array}{c} \scs{#1}\\
    \scs{#2} \end{array}\!\! \right )}  
\nc{\bincc}[2]{  \left ( {\scs{#1} \atop
    \vspace{-.5cm}\scs{#2}} \right )}  
\nc{\ot}{\otimes}
\nc{\sot}{{\scriptstyle{\ot}}}
\nc{\otm}{\overline{\ot}}
\nc{\ola}[1]{\stackrel{#1}{\la}}

\nc{\scs}[1]{\scriptstyle{#1}} \nc{\mrm}[1]{{\rm #1}}

\nc{\dirlim}{\displaystyle{\lim_{\longrightarrow}}\,}
\nc{\invlim}{\displaystyle{\lim_{\longleftarrow}}\,}

\nc{\bfk}{{\bf k}} \nc{\bfone}{{\bf 1}}
\nc{\rpr}{\circ}
\nc{\dpr}{{\tiny\diamond}}
\nc{\rprpm}{{\rpr}}

\nc{\mmbox}[1]{\mbox{\ #1\ }} \nc{\ann}{\mrm{ann}}
\nc{\Aut}{\mrm{Aut}} \nc{\can}{\mrm{can}}
\nc{\twoalg}{{two-sided algebra}\xspace}
\nc{\colim}{\mrm{colim}}
\nc{\Cont}{\mrm{Cont}} \nc{\rchar}{\mrm{char}}
\nc{\cok}{\mrm{coker}} \nc{\dtf}{{R-{\rm tf}}} \nc{\dtor}{{R-{\rm
tor}}}
\renewcommand{\det}{\mrm{det}}
\nc{\depth}{{\mrm d}}
\nc{\End}{\mrm{End}} \nc{\Ext}{\mrm{Ext}}
\nc{\Fil}{\mrm{Fil}} \nc{\Frob}{\mrm{Frob}} \nc{\Gal}{\mrm{Gal}}
\nc{\GL}{\mrm{GL}} \nc{\Hom}{\mrm{Hom}} \nc{\hsr}{\mrm{H}}
\nc{\hpol}{\mrm{HP}}  \nc{\id}{\mrm{id}} \nc{\im}{\mrm{im}}

\nc{\incl}{\mrm{incl}} \nc{\length}{\mrm{length}}
\nc{\LR}{\mrm{LR}} \nc{\mchar}{\rm char} \nc{\NC}{\mrm{NC}}
\nc{\mpart}{\mrm{part}} \nc{\pl}{\mrm{PL}}
\nc{\ql}{{\QQ_\ell}} \nc{\qp}{{\QQ_p}}
\nc{\rank}{\mrm{rank}} \nc{\rba}{\rm{RBA }} \nc{\rbas}{\rm{RBAs }}
\nc{\rbpl}{\mrm{RBPL}}
\nc{\rbw}{\rm{RBW }} \nc{\rbws}{\rm{RBWs }} \nc{\rcot}{\mrm{cot}}
\nc{\rest}{\rm{controlled}\xspace}
\nc{\rdef}{\mrm{def}} \nc{\rdiv}{{\rm div}} \nc{\rtf}{{\rm tf}}
\nc{\rtor}{{\rm tor}} \nc{\res}{\mrm{res}} \nc{\SL}{\mrm{SL}}
\nc{\Spec}{\mrm{Spec}} \nc{\tor}{\mrm{tor}} \nc{\Tr}{\mrm{Tr}}
\nc{\mtr}{\mrm{sk}}

\nc{\ab}{\mathbf{Ab}} \nc{\Alg}{\mathbf{Alg}}

\nc{\BA}{{\mathbb A}} \nc{\CC}{{\mathbb C}} \nc{\DD}{{\mathbb D}}
\nc{\EE}{{\mathbb E}} \nc{\FF}{{\mathbb F}} \nc{\GG}{{\mathbb G}}
\nc{\HH}{{\mathbb H}} \nc{\LL}{{\mathbb L}} \nc{\NN}{{\mathbb N}}
\nc{\QQ}{{\mathbb Q}} \nc{\RR}{{\mathbb R}} \nc{\BS}{{\mathbb{S}}} \nc{\TT}{{\mathbb T}}
\nc{\VV}{{\mathbb V}} \nc{\ZZ}{{\mathbb Z}}


\nc{\calao}{{\mathcal A}} \nc{\cala}{{\mathcal A}}
\nc{\calc}{{\mathcal C}} \nc{\cald}{{\mathcal D}}
\nc{\cale}{{\mathcal E}} \nc{\calf}{{\mathcal F}}
\nc{\calfr}{{{\mathcal F}^{\,r}}} \nc{\calfo}{{\mathcal F}^0}
\nc{\calfro}{{\mathcal F}^{\,r,0}} \nc{\oF}{\overline{F}}
\nc{\calg}{{\mathcal G}} \nc{\calh}{{\mathcal H}}
\nc{\cali}{{\mathcal I}} \nc{\calj}{{\mathcal J}}
\nc{\call}{{\mathcal L}} \nc{\calm}{{\mathcal M}}
\nc{\caln}{{\mathcal N}} \nc{\calo}{{\mathcal O}}
\nc{\calp}{{\mathcal P}} \nc{\calq}{{\mathcal Q}} \nc{\calr}{{\mathcal R}}
\nc{\calt}{{\mathcal T}} \nc{\caltr}{{\mathcal T}^{\,r}}
\nc{\calu}{{\mathcal U}} \nc{\calv}{{\mathcal V}}
\nc{\calw}{{\mathcal W}} \nc{\calx}{{\mathcal X}}
\nc{\CA}{\mathcal{A}}

\nc{\fraka}{{\mathfrak a}} \nc{\frakB}{{\mathfrak B}}
\nc{\frakb}{{\mathfrak b}} \nc{\frakd}{{\mathfrak d}}
\nc{\oD}{\overline{D}}
\nc{\frakF}{{\mathfrak F}} \nc{\frakg}{{\mathfrak g}}
\nc{\frakm}{{\mathfrak m}} \nc{\frakM}{{\mathfrak M}}
\nc{\frakMo}{{\mathfrak M}^0} \nc{\frakp}{{\mathfrak p}}
\nc{\frakS}{{\mathfrak S}} \nc{\frakSo}{{\mathfrak S}^0}
\nc{\fraks}{{\mathfrak s}} \nc{\os}{\overline{\fraks}}
\nc{\frakT}{{\mathfrak T}}
\nc{\oT}{\overline{T}}
\nc{\frakX}{{\mathfrak X}} \nc{\frakXo}{{\mathfrak X}^0}
\nc{\frakx}{{\mathbf x}}
\nc{\frakTx}{\frakT}      
\nc{\frakTa}{\frakT^a}        
\nc{\frakTxo}{\frakTx^0}   
\nc{\caltao}{\calt^{a,0}}   
\nc{\ox}{\overline{\frakx}} \nc{\fraky}{{\mathfrak y}}
\nc{\frakz}{{\mathfrak z}} \nc{\oX}{\overline{X}}

\font\cyr=wncyr10


\title[Leibniz conformal bialgebras]{Leibniz conformal bialgebras and the classical Leibniz conformal Yang-Baxter equation}

\author{Zhongyin Xu}
\address{Chern Institute of Mathematics \& LPMC, Nankai University,
Tianjin, 300071, China}
\email{zy\_xu@mail.nankai.edu.cn}

\author{Chengming Bai}
\address{Chern Institute of Mathematics \& LPMC, Nankai University, Tianjin 300071, PR China}
\email{baicm@nankai.edu.cn}

\author{Yanyong Hong}
\address{School of Mathematics, Hangzhou Normal University,
Hangzhou, 311121, China}
\email{yyhong@hznu.edu.cn}

\subjclass[2010]{17A32, 17A60, 17B38, 17B69}
\keywords{Leibniz conformal bialgebra, Leibniz-dendriform conformal algebra, classical Leibniz conformal Yang-Baxter equation, Novikov dialgebra, Novikov bi-dialgebra, Rota-Baxter operator}

\begin{abstract}

We introduce the notion of Leibniz conformal bialgebras,
presenting a bialgebra theory for Leibniz conformal algebras as
well as the conformal analogues of Leibniz bialgebras. They are
equivalently characterized in terms of matched pairs and conformal
Manin triples of Leibniz conformal algebras. In the coboundary
case, the classical Leibniz conformal Yang-Baxter equation is
introduced, whose symmetric solutions give Leibniz conformal
bialgebras. Moreover, such solutions are constructed from
$\mathcal{O}$-operators on Leibniz conformal algebras and
Leibniz-dendriform conformal algebras. On the other hand, the
notion of Novikov bi-dialgebras is introduced, which correspond to
a class of Leibniz conformal bialgebras,  lifting the
correspondence between Novikov dialgebras and a class of Leibniz
conformal algebras to the context of bialgebras. In addition, we
introduce the notion of classical duplicate Novikov Yang-Baxter
equation whose symmetric solutions produce Novikov bi-dialgebras
and thus Leibniz conformal bialgebras.

\end{abstract}

\maketitle
\tableofcontents

\section{Introduction}
The notion of Lie conformal algebras was introduced by Kac in \cite{K1,K2} as a formal language describing the algebraic properties of the operator product expansions of chiral fields in two-dimensional conformal field theory.  Moreover, Lie conformal algebras are closely related with infinite-dimensional Lie algebras satisfying the locality property \cite{K3}, local Poisson brackets in the theory of integrable evolution equations \cite{BDK} and so on. The structure theory \cite{DK}, representation theory \cite{CK,CK1} and cohomology theory \cite{BKV,DKV} of finite Lie conformal algebras have been studied and well developed.

Leibniz conformal algebras introduced in \cite{BKV} are the
conformal analogues of Leibniz algebras which were first
discovered by Bloh and named D-algebras \cite{B} and then
rediscovered by Loday \delete{and named it  Leibniz algebras}
\cite{Lo,LP}. Leibniz conformal algebras are non-commutative
generalizations of Lie conformal algebras and have close
connections to field algebras \cite{BK}. 
From the perspective of algebraic structures, a Leibniz conformal
algebra is a Leibniz pseudoalgebra \cite{W} which is a Leibniz
algebra in the pseudo-tensor category \cite{BDK1}. Thus Leibniz
conformal algebras can also be seen as generalizations of Leibniz
algebras. Recently, Leibniz conformal algebras have been widely
studied, for example, quadratic Leibniz conformal algebras and
their central extensions \cite{ZH}, the cohomology theory of
Leibniz conformal algebras \cite{Z} and so on. For more details,
see \cite{GW,DS,HY,W1,W2}.
It is well known that Lie bialgebras \delete{and Leibniz
bialgebras} have important applications in both mathematics and
mathematical physics \cite{CP,D}. A natural idea is to study
``conformal analogues" of bialgebras.
Liberati developed the theory of Lie conformal bialgebras in
\cite{L} as the conformal analogues of Lie bialgebras, which
includes the introduction of the notions of conformal classical
Yang-Baxter equation, conformal Manin triples and conformal
Drinfeld's doubles. The operator forms of conformal classical
Yang-Baxter equation were investigated in \cite{HB1}. Similarly,
the theories of antisymmetric infinitesimal conformal bialgebras
and left-symmetric conformal bialgebras were developed in
\cite{HB,HL}.
However, there is no conformal analogue of ``Leibniz bialgebras"
as far as we know. Thus the purpose of this paper is to
investigate a bialgebra theory for Leibniz conformal algebras,
namely, Leibniz conformal bialgebras, giving the conformal
analogues of Leibniz bialgebras. That is, the following diagram is
commutative.
\begin{equation}\label{diag2}
\xymatrix@C=1.5cm{
  & \txt{   Leibniz algebras } \ar[d]^{bialgebras} \ar[rr]^{conformal~~ analogues\qquad}  &&\txt{ Leibniz conformal algebras} \ar[d]^{bialgebras}         \\
  & \txt{ Leibniz bialgebras} \ar[rr]^{conformal~~ analogues\qquad} && \txt{Leibniz conformal bialgebras}               }
\end{equation}
It is worth mentioning that \cite{TS}, \cite{RS} and \cite{BB}
have defined three different kinds of Leibniz bialgebras
respectively and our main idea is to extend the study of Leibniz
bialgebras given in \cite{TS}.

In order to characterize Leibniz conformal bialgebras, we
introduce the notion of matched pairs of Leibniz conformal
algebras and conformal Manin triples of Leibniz conformal
algebras. That is, under some conditions, Leibniz conformal
bialgebras are equivalent to some matched pairs of Leibniz
conformal algebras as well as conformal Manin triples of Leibniz
conformal algebras. In the study of a special class of Leibniz
conformal bialgebras called coboundary Leibniz conformal
bialgebras,  we introduce the notion of classical Leibniz
conformal Yang-Baxter equation (CLCYBE), as a conformal analogue
of the classical  Leibniz Yang-Baxter equation. Notably, its
symmetric solutions directly give Leibniz conformal bialgebras. In
addition, the CLCYBE is interpreted in terms of its operator forms
by $\mathcal{O}$-operators on Leibniz conformal algebras. In
particular, a symmetric solution of the CLCYBE corresponds to the
symmetric part of a conformal linear map $T$, where
$T_0=T_\lambda|_{\lambda=0}$ is an $\mathcal{O}$-operator in the
conformal sense, which implies that an $\mathcal{O}$-operator on a
Leibniz conformal algebra gives a symmetric solution of the CLCYBE
in a semi-direct product Leibniz conformal algebra. Furthermore,
we introduce the notion of Leibniz-dendriform conformal algebras
and show that for a Leibniz-dendriform conformal algebra which is
finite and free as a ${\bf k}[\partial]$-module, there is a
natural $\mathcal{O}$-operator on the associated Leibniz conformal
algebra. Therefore there are constructions of symmetric solutions
of the CLCYBE and hence Leibniz conformal bialgebras from
$\mathcal{O}$-operators on Leibniz conformal algebras as well as
Leibniz-dendriform conformal algebras. In summary, they are
depicted in the diagram
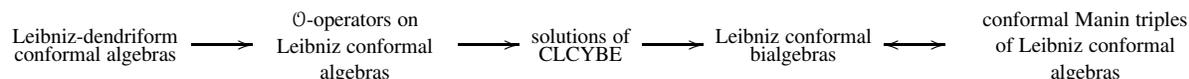
\begin{equation*}
    {\tiny\xymatrix@C=0.8cm@R=0.4cm{
            \text{\tiny Leibniz-dendriform}\atop\text{\tiny conformal algebras} \ar[r] &
            \parbox{2.5cm}{\tiny\centering $\mathcal{O}$-operators on\\ Leibniz conformal\\ algebras} \ar[r] &
            \text{\tiny solutions of}\atop \text{\tiny CLCYBE} \ar[r] &
            \text{\tiny Leibniz conformal}\atop \text{\tiny bialgebras} \ar@{<->}[r] &
            \parbox{3.5cm}{\tiny\centering conformal Manin triples \\of Leibniz conformal\\ algebras}
    }}
    \label{eq:bigliediag}
\end{equation*}

\delete{In this paper, we utilize a non-degenerate skew-symmetric invariant conformal bilinear form to introduce the concepts of conformal dual representations of Leibniz conformal algebras and to define the conformal Manin triples of  Leibniz conformal algebras. 
 Subsequently, we define the notion of Leibniz conformal bialgebras and prove that Leibniz conformal bialgebras are equivalent to  some matched pairs of Leibniz conformal algebras as well as conformal Manin triples of Leibniz conformal algebras. Additionally, by \eqref{colcnd} and the dual representation of Leibniz conformal algebras, we obtain the dual representation of Novikov dialgebras, and then define the notion of Novikov bi-dialgebras. Furthermore, we show that Novikov bi-dialgebras  correspond to a class of Leibniz conformal bialgebras.}
\delete{ As in the  case of Leibniz conformal bialgebras, the
definition of coboundary Leibniz conformal bialgebras is
introduced and its study is also significative. It leads to the
formulation of the classical Leibniz conformal Yang-Baxter
equation, or the CLCYBE in short, as a conformal analogue of the
classical  Leibniz Yang-Baxter equation. Notably, its symmetric
solutions directly give Leibniz conformal bialgebras. In addition,
the CLCYBE is interpreted in terms of its operator forms by
introducing the notion of $\mathcal{O}$-operators of Leibniz
conformal algebras. Particularly, a symmetric solution of the
CLCYBE corresponds to the symmetric part of a conformal linear map
$T$, where $T_0=T_\lambda|_{\lambda=0}$ is an
$\mathcal{O}$-operator in the conformal sense, which implies that
an $\mathcal{O}$-operator on a Leibniz conformal algebra gives a
symmetric solution of the CLCYBE in a semi-direct product Leibniz
conformal algebra. Furthermore, we introduce the notion of
Leibniz-dendriform conformal algebras and show that for a
Leibniz-dendriform conformal algebra which is finite and free as a
${\bf k}[\partial]$-module, there is a natural
$\mathcal{O}$-operator on the associated with   Leibniz conformal
algebra. Therefore there are constructions of symmetric solutions
of the CLCYBE and hence triangular conformal bialgebras from
$\mathcal{O}$-operators of Leibniz conformal algebras as well as
Leibniz-dendriform conformal algebras.}

On the other hand, based on the result on \cite{ZH}, we show that
there is the following correspondence between Novikov dialgebras
and a class of Leibniz conformal algebras.
\begin{equation}\label{colcnd}
 \text{ A class of Leibniz conformal algebras} \xleftrightarrow{ \text{Proposition} ~\ref{lcnd}}  \text {Novikov dialgebras }
\end{equation}
This fact generalizes the known correspondence between Novikov
algebras and a class of Lie conformal algebras \cite{Xu}.
\begin{equation}\label{colcnd2}
 \text{ A class of Lie conformal algebras} \longleftrightarrow  \text {Novikov algebras }
\end{equation}
Note that the notion of Novikov dialgebras was introduced in
\cite{KS} and perm algebras with derivations can naturally produce
Novikov dialgebras \cite{KS}. Moreover, it is worth mentioning
that the correspondence given in (\ref{colcnd2}) was lifted to the
context of bialgebras in \cite{HBG}, which shows that Novikov
bialgebras can produce Lie conformal bialgebras. Hence it is
natural to introduce the notion of Novikov bi-dialgebras and show
that there is a correspondence between Novikov bi-dialgebras and a
class of Leibniz conformal bialgebras, lifting the correspondence
given in (\ref{colcnd}) to the context of bialgebras. Moreover, we
introduce the notion of classical duplicate Novikov Yang-Baxter
equation (CDNYBE), whose symmetric solutions produce Novikov
bi-dialgebras and show that symmetric solutions of the CDNYBE
correspond to a class of solutions of the CLCYBE.

 \delete{Novikov dialgebras,  
introduced in \cite{KS}, are  algebraic structures equipped with two operations. As shown in \cite{ZH}, a correspondence exists between Novikov dialgebras and a special class of Leibniz conformal algebras as follows.
\begin{equation}\label{colcnd}
 \text{ a class of Leibniz conformal algebras} \xleftrightarrow{ \text{Proposition} ~\ref{lcnd}~~}  \text {Novikov dialgebras }
\end{equation}
Meanwhile, it is hard to give the dual representation of this type of dialgebras equipped with two operations directly. It is worth mentioning that the dual representation of Gel'fand-Dorfman algebras can be given from the dual representation of Lie conformal algebras in \cite{HBG}. The diagram \eqref{colcnd} gives us a natural idea, that is, through this correspondence, the dual representation of Novikov dialgebras may be obtained from the dual representation of Leibniz conformal algebras. Then this leads us to lift \eqref{colcnd} to the level of bialgebras
by developing the bialgebra theories of Leibniz conformal algebras and Novikov dialgebras.}

\delete{We introduce the notion of classical duplicate Novikov classical Yang-Baxter equation, or CDNCYBE in short, for which symmetric solutions can be given as Novikov bi-dialgebras and show that symmetric solutions of the CDNCYBE correspond to a class of solutions of the CLCYBE. Moreover, through our study of the connection between Leibniz conformal algebras and Novikov dialgebras in Diagram \eqref{colcnd} we obtain the following commutative diagram.
\begin{equation}\label{diag2}
\xymatrix@C=1.5cm{
  & \txt{   solutions of \\the CDNCYBE } \ar[d]^{\text{Prop.} \ref{relcdn}} \ar[r]^{\quad \text{Coro.} \ref{condba}}  &\txt{ Novikov\\ bi-dialgebras} \ar[d]^{\text{Th.}\ref{thnblb}}         \\
  & \txt{ solutions of \\the LCYBE} \ar[u]\ar[r]^{\quad \text{Coro.} \ref{colcba}} & \txt{Leibniz \\ conformal bialgebras}\ar[u]                }
\end{equation}
}
\delete{\begin{equation}
\begin{tikzcd}
    \begin{array}{c} \text{solutions of} \\ \text{the CDNCYBE} \end{array} & {\text{NBDAs}} & \begin{array}{c} \text{Matched pairs of} \\ \text{NDAs} \end{array} \\
    \begin{array}{c} \text{solutions of} \\\text{the CLCYBE}  \end{array} & \begin{array}{c} \text{Leibniz conformal} \\ \text{bialgebras} \end{array} & \begin{array}{c} \text{conformal matched}\\\text{ pairs of Leibniz}\\\text{ conformal algebras}  \end{array}
    \arrow["{Coro. \ref{condba}}", from=1-1, to=1-2]
    \arrow["{Prop. \ref{relcdn}}", tail reversed, from=1-1, to=2-1]
    \arrow[tail reversed, from=1-2, to=1-3]
    \arrow["{Th.\ref{thnblb}}", tail reversed, from=1-2, to=2-2]
    \arrow["{Prop.\ref{ppmp2}}", tail reversed, from=1-3, to=2-3]
    \arrow["{Coro. \ref{colcba}}", from=2-1, to=2-2]
    \arrow["{Prop. \ref{pmb}}", tail reversed, from=2-2, to=2-3]
\end{tikzcd}
\end{equation}}

This paper is organized as follows.
 In Section \ref{2}, we recall the definitions of Leibniz conformal algebras and their representations, and show that there is a correspondence between Novikov dialgebras and a class of Leibniz conformal algebras.
  Furthermore, we introduce the definitions of matched pairs of Leibniz conformal algebras and conformal Manin triples of Leibniz conformal algebras, and show
  that under some conditions, a special matched pair of Leibniz conformal algebras is equivalent to a  conformal Manin triple of Leibniz conformal algebras.
In Section \ref{3}, the notion of Leibniz conformal bialgebras is
introduced and we show that under some conditions, there is an
equivalence between conformal Manin  triples of Leibniz conformal
algebras and Leibniz conformal bialgebras. In Section \ref{4}, the
notion of CLCYBE is introduced in the study of a class of Leibniz
conformal bialgebras called coboundary Leibniz conformal
bialgebras.  In addition, we introduce the definitions of
$\mathcal{O}$-operators on Leibniz conformal algebras and
Leibniz-dendriform conformal algebras to construct (symmetric)
solutions of the CLCYBE and hence give rise to Leibniz conformal
bialgebras. In Section \ref{5},  we define Novikov bi-dialgebras
and show that there is a correspondence between Novikov
bi-dialgebras and a class of Leibniz conformal bialgebras.
Furthermore, we introduce the notion of CDNYBE and show that its
symmetric solutions can produce Novikov bi-dialgebras and thus
Leibniz conformal bialgebras.

Throughout this paper, let  $\bf k$ be a field of characteristic zero. All tensors over ${\bf k}$ are denoted by $\otimes$. We denote the identity map by $\id$. Moreover, if $A$ is a vector space, then we denote the space of polynomials of $\lambda$ with coefficients in $A$ by $A[\lambda]$. For a vector space or a ${\bf k}[\partial]$-module $A$, let
\begin{equation*}
\tau:A\otimes A\rightarrow A\otimes A,\qquad a\otimes b\mapsto b\otimes a\;\;\;\;\text{for all $ a,b\in A$,}
\end{equation*}
be the flip operator. Then we denote $\tau_{12}=\tau\otimes \id$, $\tau_{23}=\id \otimes \tau$, $\tau_{13}=(\tau\otimes \id)(\id\otimes \tau)(\tau\otimes \id)$.

\delete{\cm{Since we do not use ``$\forall$" in the equations, I do not
know whether we use ``for all $a\in A$" or ``for $a\in A$".
Anyway, we should unify the notations.}}

\section{Preliminaries on Leibniz conformal algebras}\label{2}
We recall some facts about Leibniz conformal algebras and show
that there is a correspondence between Novikov dialgebras and a
class of Leibniz conformal algebras. Moreover, we introduce the
notions of dual representations, matched pairs and conformal Manin
triples of Leibniz conformal algebras and show that under some
conditions, some matched pair of Leibniz conformal algebras is
equivalent to a conformal Manin triple of Leibniz conformal
algebras.
\subsection{Leibniz conformal algebras and Novikov dialgebras}
\begin{defi}\cite{BK} \label{lca}
A {\bf Leibniz conformal algebra} $Q$ is a ${\bf k}[\partial]$-module with a $\lambda$-bracket $[\cdot_\lambda\cdot]$ which is a
${\bf k}$-bilinear map from $Q\times Q\rightarrow Q[\lambda]$, satisfying
\begin{align}
&[\partial a_\lambda b]=-\lambda[a_\lambda b],   \quad\quad\quad\quad   [ a_\lambda \partial b]=(\lambda+\partial)[a_\lambda b],\tag{conformal sesquilinearity} \label{cs}\\
&[a_\lambda[b_\mu c]]=[[a_\lambda b]_{\lambda+\mu}c]+[b_\mu[a_\lambda c]]  \tag{Jacobi identity}   \label{li} \;\;\;\text{for all $a,b,c\in Q$.}
\end{align}

A  Leibniz conformal algebra $Q$ is called {\bf finite}, if it is  finitely generated as a ${\bf k}[\partial]$-module. Otherwise, we call it {\bf infinite}.
\end{defi}
\begin{rmk}
Recall \cite{K1} that a {\bf Lie conformal algebra} $Q$ is a Leibniz conformal algebra satisfying 
\begin{equation}
[a_\lambda b]=-[b_{-\lambda-\partial} a]\tag{skew-symmetry}\;\;\;\text{for all $a, b\in Q$.}
\end{equation}
Thus Lie conformal algebras are Leibniz conformal algebras.
\end{rmk}

\delete{\begin{rmk}
In fact, \eqref{li} is also called the left Leibniz identity. Similarly, there is a right Leibniz identity:
\begin{equation*}
[a_\lambda[b_\mu c]]=[[a_\lambda b]_{\lambda+\mu}c]-[[a_\lambda c]_{-\mu-\partial}b],
\end{equation*}
for all $ a,b,c\in A$. If a  ${\bf k}[\partial]$-module $A$ with a $\lambda$-bracket $[\cdot_\lambda\cdot]$ satisfying conformal sesquilinearity and right Leibniz identity, then $A$ is called the right  Leibniz conformal algebra. the left Leibniz conformal algebra is usually called the Leibniz conformal algebra.\par
On the other hand, any Leibniz conformal algebra $A$ can be endowed with a right LCA structure   $(A,[\cdot_\lambda \cdot])'$ by $[a_\lambda b]'=-[b_{-\lambda-\partial}a]$ for all $a, b\in A$ and vice versa. 
\end{rmk}}

Recall \cite{Lo} that a {\bf Leibniz algebra} $(A,[\cdot,\cdot])$ is a vector space $A$ with a bilinear operation $[\cdot,\cdot]: A\times A\rightarrow A$, satisfying
\begin{equation*}
	[a,[b,c]]=[[a,b],c]+[b,[a,c]]\qquad \text{for all $a, b,c\in A$.}
\end{equation*}
\begin{ex}
Let $(A, [\cdot, \cdot])$ be a Leibniz algebra. Then we can
naturally define a Leibniz conformal algebra structure on
$\text{Cur}~A={\bf k}[\partial]\otimes A$ with the
$\lambda$-bracket:
\begin{equation*}
[a_\lambda b]=[a,b]\;\;\;\text{for all $a, b\in A$.}
\end{equation*}
$\text{Cur}~A$ is called the {\bf current Leibniz conformal algebra} associated with $(A, [\cdot,\cdot])$.
\end{ex}
\begin{pro}\cite{HY} \label{ll}
Let $Q={\bf k}[\partial] L$ be a Leibniz conformal algebra which is a free ${\bf k}[\partial]$-module of rank 1. Then $Q$ is either abelian or isomorphic to the Virasoro Lie conformal algebra, that is, $[L_\lambda L]=(\partial+2\lambda)L$.
\end{pro}

Next, we recall the definition of Novikov dialgebras.

\begin{defi}\cite{KS}
Let $A$ be a vector space with two binary operations $\vdash$ and $\dashv$. If for all $a$, $b$, $c\in A$, they satisfy the following equalities
\begin{eqnarray*}
&&(a\vdash b)\dashv c=(a\vdash c)\vdash b=(a\dashv c)\vdash b,\;\;a\dashv(b\dashv c)=a\dashv(b\vdash c),\\
&&(a \dashv b)\dashv c=(a\dashv c)\dashv b,\;\;(a\dashv b)\dashv c-a\dashv (b\dashv c)=(b\vdash a)\dashv c-b\vdash (a\dashv c),\\
&&(a\vdash b)\vdash c-a\vdash (b\vdash c)=(b\vdash a)\vdash c-b\vdash (a\vdash c),
\end{eqnarray*}
then $(A,\dashv,\vdash)$ is called a {\bf(left) Novikov dialgebra}.

\end{defi}
\begin{rmk}\label{rmk-ND1}
There is also the definition of right Novikov dialgebras. A {\bf right Novikov dialgebra} is a triple $(A, \dashv, \vdash)$ where $A$ is a vector space with two binary operations satisfying
\begin{eqnarray*}
&&a \vdash (b \vdash c)=b \vdash (a\vdash c),~~a \vdash (b \dashv c)=b \dashv (a \dashv c),\;\;(a \vdash b-a \dashv b)\vdash c=0,\\
&&a \dashv (b\vdash c-b\dashv c)=0,\;\;a  \vdash (b  \vdash c-c \dashv b)=(a  \vdash b)  \vdash c-(a \vdash c) \dashv b,\\
&&a \dashv (b\vdash c-c\dashv b)=(a\dashv  b)\dashv  c-(a\dashv c)\dashv b\;\;\;\text{for all $a, b, c\in A$.}
\end{eqnarray*}
Note that such algebra structures appear in the study of quadratic
Leibniz conformal algebras  \cite{ZH}.
Moreover, it is easy to see that  $(A, \dashv, \vdash)$ is a
Novikov dialgebra if and only if  $(A, \dashv', \vdash')$ is a
right Novikov dialgebra where $ \vdash'$ and $\dashv'$ are defined
by
\begin{eqnarray*}
a\dashv' b:=b\vdash a,\;\;a\vdash' b:= b\dashv a\;\;\;\text{for all $a,b \in A$.}
\end{eqnarray*}
\end{rmk}
\delete{We can find that the quadratic Leibniz conformal algebra corresponds to a quadruple in \cite{ZH}, if $[\cdot,\cdot]$ is trivial and we set $\dashv'$ and $\vdash'$ by $a\dashv' b=b\vdash a $ and $a\vdash' b= b\dashv a$, then $(A,\dashv',\vdash')$ is a NDA, and we call $(A,\dashv,\vdash)$ a right NDA. According to Theorem 3.1 in \cite{ZH}, the following proposition can be obtained.}

Recall that a {\bf perm algebra} is a pair $(A,\cdot)$, where $A$
is a vector space, $\cdot$ is a binary operation and the following
equalities hold.
\begin{equation}
(a\cdot b)\cdot c=a\cdot (b\cdot c)=b\cdot (a\cdot c)\;\;\;\text{for all $a,b,c\in
A$}.
\end{equation}

\begin{lem} {\rm \cite{KS}} \label{lem:perm}
Let $(A,\cdot)$ be a perm algebra and $D$ be a derivation, that
is, $D$ satisfies $D(a\cdot b)=D(a)\cdot b+a\cdot D(b)$ for all
$a,b\in A$. Define two binary operations $\vdash$ and $\dashv$ by
\begin{equation}\label{eq:perm}
a\vdash b:=a\cdot D(b),\;\; a\dashv b: =D(b)\cdot a\;\;\;\text{for all $a,b\in A$.}
\end{equation}
Then $(A,\dashv,\vdash)$ is a Novikov dialgebra.
\end{lem}

Next, we give a correspondence between Novikov dialgebras and a class of Leibniz conformal algebras.
\begin{pro} \label{lcnd}
Let $(A,\dashv,\vdash)$ be a vector space with binary operations $\vdash$ and $\dashv$. Equip the free ${\bf k}[\partial]$-module $Q:={\bf k}[\partial]A$ $(={\bf k}[\partial]\otimes A)$ with the $\lambda$-bracket
\begin{equation}\label{lamnd}
[a_\lambda b]:=\partial(b\dashv a)+\lambda(a\vdash b+b\dashv a)\;\;\;\text{for all $a, b\in A.$}
\end{equation}
Then $(Q, [\cdot_\lambda\cdot])$ is a Leibniz conformal algebra if and only if $(A,\dashv,\vdash)$ is a Novikov dialgebra. We call $Q$ the {\bf Leibniz conformal algebra corresponding to} $(A,\dashv,\vdash)$.
\end{pro}
\begin{proof}
Define
\begin{equation}\label{lamnd}
[a_\lambda b]:=\partial(a\vdash' b)+\lambda(a\vdash' b+b\dashv' a)\;\;\;\text{for all $a, b\in A$,}
\end{equation}
where $\vdash'$ and $\dashv'$ are two binary operations on $A$.
By \cite[Theorem 3.1]{ZH}, $(Q, [\cdot_\lambda\cdot])$ is a Leibniz conformal algebra if and only if $(A, \dashv', \vdash')$ is a right Novikov dialgebra. Set \begin{eqnarray*}
a\dashv b:=b\vdash' a,\;\;a\vdash b:= b\dashv' a\;\;\;\text{for all $a,b \in A$.}
\end{eqnarray*}
By Remark \ref{rmk-ND1}, $(A, \dashv', \vdash')$ is a right
Novikov dialgebra if and only if $(A, \dashv,\vdash)$ is a Novikov
dialgebra. Then the conclusion follows.
\end{proof}
\begin{cor}\label{copermlcn}
Let $(A, \cdot)$ be a perm algebra and $D$ be a derivation on $(A,
\cdot)$. Equip the free ${\bf k}[\partial]$-module $Q:={\bf
k}[\partial]A$ $(={\bf k}[\partial]\otimes A)$ with the
$\lambda$-bracket
\begin{equation}\label{lamnd}
[a_\lambda b]:=\partial(D(a)\cdot b)+\lambda(a\cdot D( b)+D(a)\cdot b)\;\;\;\text{for all $a, b\in A$.}
\end{equation}
Then $(Q, [\cdot_\lambda\cdot])$ is a Leibniz conformal algebra.
\end{cor}
\begin{proof}
Define two binary operations $\vdash$ and $\dashv$ by
(\ref{eq:perm}).
 By Lemma~\ref{lem:perm}, $(A, \dashv, \vdash)$ is a
Novikov dialgebra. Then it follows directly from Proposition
\ref{lcnd}.
\end{proof}
\subsection{Matched pairs of Leibniz conformal algebras}
\begin{defi}\cite{K1}
Let $U$ and $V$ be two ${\bf k}[\partial]$-modules.\delete{ A {\bf left conformal linear map} $a_\lambda: U\rightarrow V[\lambda]$ such that $a_\lambda(\partial u)=-\lambda a_\lambda u$ for all $u\in U$.}
A {\bf conformal linear map} from $U$ to $V$ is a ${\bf k}$-linear map $a_\lambda: U\rightarrow V[\lambda]$ such that $a_\lambda(\partial u)=(\lambda+\partial) a_\lambda u$ for all $u\in U$.
\delete{A right conformal linear map is usually called a {\bf conformal linear map}.} The vector space of all conformal linear maps from $U$ to $V$, denoted by $\text{Chom}(U,V)$ , has a $ {\bf k}[\partial]$-module structure given by
\begin{equation*}
(\partial a)_\lambda u=-\lambda a_\lambda u\;\;\; \text{for all $a\in \text{Chom}(U,V)$ and $u\in U$.}
\end{equation*}
In particular, if $U=V$, then we denote $\text{Chom}(U,V)$ by  $\text{Cend}(V)$. If $V$ is a finite  $ {\bf k}[\partial]$-module, then  $\text{Cend}(V)$ is a Lie conformal algebra defined by
$$
[f_\lambda g]_\mu v=f_\lambda(g_{\mu-\lambda}v)-g_{\mu-\lambda}(f_\lambda v)\;\;\;\text{for all $f, g\in \text{Cend}(V)$ and $v\in V$}.
$$
We define the conformal dual of a $ {\bf k}[\partial]$-module $U$ as $ U^{*c}=\text{Chom}(U,{\bf k})$, where ${\bf k}$ is viewed as the trivial ${\bf k}[\partial]$-module, that is,
$U^{*c}=\{a:U\rightarrow {\bf k}[\lambda]~~|~~\text{ ${\bf k}$-linear and }a_\lambda(\partial u)=\lambda a_\lambda u ~~\text{for all $u\in U$}\}$. For convenience in the following, we set $\langle a,u\rangle_\lambda=a_\lambda u$ for all $a\in U^{*c}$ and $u\in U$.

Let $U$, $V$ and $W$ be ${\bf k}[\partial]$-modules. A {\bf conformal bilinear map} from $ U\times V$ to  $W$ is a ${\bf k}$-bilinear map $f_\lambda: U\times V\rightarrow W[\lambda]$, such that $f_\lambda(\partial u,v)=-\lambda f_\lambda(u,v)$ and $f_\lambda( u,\partial v)=(\lambda+\partial) f_\lambda(u,v)$
for all $u\in U$ and $ v\in V$.
\end{defi}

\begin{defi}\cite{Z}
{\bf A representation} of a Leibniz conformal algebra $(Q,[\cdot_\lambda\cdot])$ is a triple $(M,l,r)$, where $M$ is a ${\bf k}[\partial]$-module and $l,r$: $Q\rightarrow \text{Cend}(M)$ are two ${\bf k}[\partial]$-module homomorphisms satisfying
\begin{align}
\label{m1}&l(a)_\lambda(l(b)_\mu v)=l([a_\lambda b])_{\lambda+\mu}v+l(b)_\mu(l(a)_\lambda v),\\
\label{m2}&l(a)_\lambda(r(b)_{-\mu-\partial}v)=r(b)_{-\lambda-\mu-\partial}(l(a)_\lambda v)+r([a_\lambda b])_{-\mu-\partial}v,\\
\label{m3}&r(b)_{-\lambda-\mu-\partial}(r(a)_{-\mu-\partial}v)=-r(b)_{-\lambda-\mu-\partial}(l(a)_\lambda v)\;\;\;\text{for all $a, b\in Q$ and $v\in M$}.
\end{align}

A representation $(M,l,r)$ of $(Q,[\cdot_\lambda\cdot])$ is called {\bf finite}, if $M$ is a finitely generated ${\bf k}[\partial]$-module.
\end{defi}
\delete{It should be pointed out that some synonyms for expressing the fact that $(M,l,r)$ is a representation of $A$ include that $M$ is an $A$-bimodule and that $A$ acts on $M$.}

\begin{ex}
Let $(Q, [\cdot_\lambda \cdot])$ be a Leibniz conformal algebra. Define two ${\bf k}[\partial]$-module homomorphisms $L$ and $R$ from $Q$ to $\text{Cend}(Q)$
 by $L(a)_\lambda b=[a_\lambda b]$ and $R(a)_\lambda b=[b_{-\lambda-\partial}a]$ for all $a$, $b\in Q$. Then $(Q,L,R)$ is a representation of $(Q, [\cdot_\lambda \cdot])$,
 which is called the {\bf regular representation} of $(Q, [\cdot_\lambda \cdot])$.
If there is a Leibniz conformal algebra structure on the conformal
dual $Q^{*c}$ of $Q$, then we denote the regular representation of
$(Q^{\ast c}, [\cdot_\lambda \cdot])$ by $(Q^{\ast c},
\mathcal{L}, \mathcal{R})$.
\end{ex}

\delete{Define two ${\bf k}[\partial]$-module homomorphisms $L^*$ and $R^*$ from $A$ to $Cend(A^{*c}) $ by
\begin{align}\label{ld}
\langle L^*(a)_\lambda f,b\rangle_\mu=-\langle f,[a_\lambda b]\rangle_{\mu-\lambda},\qquad
\langle R^*(a)_\lambda f,b\rangle_\mu=-\langle f,[b_{-\lambda-\partial} a]\rangle_{\mu-\lambda}
\end{align}
If there is a Leibniz conformal algebra structure on the conformal dual $A^{*c}$ of $A$, then we denote the left and right $\lambda$-multiplications by $\mathcal{L}$ and $\mathcal{R}$ respectively, that is $\mathcal{L}(x)_\lambda y=[x_\lambda y]_{A^{*c}} $ and $ \mathcal{R}(x)_\lambda y=[y_{-\lambda-\partial}x]_{A^{*c}}$ for all $x,y\in A^{*c}$.}
\begin{pro}\label{ss}
Let $Q$ be a Leibniz conformal algebra.
Then $(M,l,r)$ is a representation of $Q$ if and only if  there is a Leibniz conformal algebra structure on the direct sum $Q\oplus M$ as ${\bf k}[\partial]$-modules given by
\begin{equation}\label{sdlp}
[(a+x)_\lambda(b+y)]=[a_\lambda b]+l(a)_\lambda y+r(b)_{-\lambda-\partial}x\;\;\;\text{for all $a, b\in Q$ and $x, y\in M$}.
\end{equation}
Denote this Leibniz conformal algebra by $Q\ltimes_{l,r} M$, which is called a {\bf semi-direct product} of $Q$ and $(M,l,r)$, or a {\bf semi-direct product Leibniz conformal algebra}.
\end{pro}
\begin{proof}
It is straightforward.
\end{proof}

\begin{pro}\label{dual}
Let $(M,l,r)$ be a finite representation of a Leibniz conformal algebra $(Q, [\cdot_\lambda \cdot])$, which is free as a ${\bf k}[\partial]$-module.
Let $l^\ast $ and $r^* $  be two ${\bf k}[\partial]$-module homomorphisms from $Q$ to $\text{Cend}(M^{*c}) $ given by
\begin{equation}\label{lrm}
\langle l^*(a)_\lambda f,v\rangle_\mu=-\langle f,l(a)_\lambda v\rangle_{\mu-\lambda},\quad
\langle r^*(a)_\lambda f,v\rangle_\mu=-\langle f,r(a)_{\lambda} v\rangle_{\mu-\lambda}a,
\end{equation}
for all $a\in A$, $f\in M^{*c}$ and $v\in M$. Then $(M^{\ast c},l^\ast,-l^\ast-r^\ast)$ is a representation of  $(Q, [\cdot_\lambda \cdot])$, which is called the {\bf dual representation} of $(M,l,r)$.
\end{pro}
\begin{proof}
Let $a$, $b\in Q$, $f\in M^{*c}$ and $v\in M$. We have
{\small
\begin{align*}
&\langle l^\ast(a)_\lambda(-(l^\ast+r^\ast)(b)_{-\mu-\partial}f)
+(l^\ast+r^\ast)(b)_{-\lambda-\mu-\partial}(l^\ast(a)_\lambda f)
+(l^\ast+r^\ast)([a_\lambda b])_{-\mu-\partial}f,v\rangle_\theta\\
=&\langle(l^\ast+r^\ast)(b)_{-\mu-\partial}f, l(a)_\lambda v\rangle_{\theta-\lambda}-\langle l^\ast(a)_\lambda f, (l+r)(b)_{-\lambda-\mu+\theta}v\rangle_{\lambda+\mu}-\langle f, (l+r)([a_\lambda b])_{-\mu+\theta}v\rangle_\mu\\
=&\langle f,-(l+r)(b)_{-\lambda-\mu+\theta}(l(a)_\lambda v) +l(a)_\lambda((l+r)(b)_{-\lambda-\mu+\theta}v)-(l+r)([a_\lambda b])_{-\mu+\theta}v \rangle_\mu\\
=&-\langle f, l(b)_{-\lambda-\mu+\theta}(l(a)_\lambda v)-l(a)_\lambda(l(b)_{-\lambda-\mu+\theta}v)+l([a_\lambda b])_{-\mu+\theta}v\rangle_\mu \\
&-\langle f,r(b)_{\theta-\lambda-\partial}(l(a)_\lambda v)-l(a)_\lambda(r(b)_{\theta-\partial}v)+r([a_\lambda b])_{\theta-\partial}v \rangle_\mu=0,
\end{align*}}
which means that \eqref{m2} holds.
Similarly, we can verify that \eqref{m1} and \eqref{m3} hold. Therefore, $(M^{\ast c},l^\ast,-l^\ast-r^\ast)$ is a representation of  $(Q, [\cdot_\lambda \cdot])$.
\end{proof}
\begin{ex}
Let $(Q, [\cdot_\lambda \cdot])$ be a finite Leibniz conformal algebra which is free as a ${\bf k}[\partial]$-module. By Proposition \ref{dual}, the dual representation of the regular representation $(Q, L, R)$ is $(Q^{\ast c}, L^\ast, -L^\ast-R^\ast)$.
\end{ex}

\begin{defi}
Let $Q_1$ and $Q_2$ be two Leibniz conformal algebras. If $(Q_2,l,r)$ is a representation of $Q_1$, $(Q_1,\phi,\psi)$ is a representation of $Q_2$ and the following equations
\begin{align}
\label{mp1}&[x_\lambda(r(a)_{-\mu-\partial}y)]+r(\phi(y)_\mu a)_{-\lambda-\partial}x=r(a)_{-\lambda-\mu-\partial}([x_\lambda y])+[y_\mu(r(a)_{-\lambda-\partial}x)]+r(\phi(x)_\lambda a)_{-\mu-\partial}y,\\
\label{mp2}&l(a)_\lambda([x_\mu y])=[(l(a)_\lambda x)_{\lambda+\mu}y]+l(\psi(x)_{-\lambda-\partial}a)_{\lambda+\mu}y+[x_\mu(l(a)_\lambda y)]+r(\psi(y)_{-\lambda-\partial}a)_{-\mu-\partial}x,\\
\label{mp3}&[x_\lambda(l(a)_\mu y)]+r(\psi(y)_{-\mu-\partial}a)_{-\lambda-\partial}x=l(\phi(x)_\lambda a)_{\lambda+\mu}y+l(a)_\mu([x_\lambda y])+[(r(a)_{-\lambda-\partial}x)_{\lambda+\mu}y],\\
\label{mp4}&[a_\lambda(\psi(x)_{-\mu-\partial}b)]+\psi(l(b)_\mu x)_{-\lambda-\partial}a=\psi(x)_{-\lambda-\mu-\partial}[a_\lambda b]+[b_\mu (\psi(x)_{-\lambda-\partial}a)]+\psi(l(a)_\lambda x)_{-\mu-\partial}b,\\
\label{mp5}&[a_\lambda(\phi(x)_\mu b)]+\psi(r(b)_{-\mu-\partial}x)_{-\lambda-\partial}a=[(\psi(x)_{-\lambda-\partial }a)_{\lambda+\mu}b]+\phi(l(a)_\lambda x)_{\lambda+\mu}b+\phi(x)_\mu([a_\lambda b]),\\
&[(\phi(x)_\lambda a)_{\lambda+\mu}b]+\phi(r(a)_{-\lambda-\partial}x)_{\lambda+\mu}b+[(\psi(x)_{-\mu-\partial}a)_{\lambda+\mu}b]+\phi(l(a)_\mu x)_{\lambda+\mu}b=0, \label{mp6}
\end{align}
hold for all $a,b\in Q_1$ and $x,y\in Q_2$, then $(Q_1,Q_2;(l,r),(\phi,\psi) )$ is called a {\bf matched pair} of  Leibniz conformal algebras $Q_1$ and $Q_2$.
\end{defi}

\delete{It should be pointed that the matched pair of $A$ and $B$ is equivalent to the bicross product of $A$ and $B$ in \cite{HY}. By this relation, we can define a Leibniz conformal algebra structure on the matched pair of Leibniz conformal algebras.}
\begin{pro}\cite{HY}\label{mplca}
Let $(Q_1,Q_2;(l,r),(\phi,\psi))$ be a matched pair of Leibniz conformal algebras. Then there is a Leibniz conformal algebra structure on the direct sum $Q_1\oplus Q_2$ of ${\bf k}[\partial]$-modules  defined by
\begin{equation}\label{bp1}
[(a+x)_\lambda(b+y)]=[a_\lambda b]+\phi(x)_\lambda b+\psi(y)_{-\lambda-\partial}a+[x_\lambda y]+l(a)_\lambda y+r(b)_{-\lambda-\partial}x,
\end{equation}
for all $a,b\in Q_1$, $x,y\in Q_2$. Denote this Leibniz conformal
algebra by $Q_1\Join Q_2$, which is called a {\bf bicrossed
product} of $Q_1$ and $Q_2$ associated with the matched pair
$(Q_1,Q_2;(l,r),(\phi,\psi))$ of Leibniz conformal algebras.
Moreover, any Leibniz conformal algebra $E=Q_1\oplus Q_2$ where
the sum is the direct sum of ${\bf k}[\partial]$-modules such that
$Q_1$ and $Q_2$ are two Leibniz conformal subalgebras of $E$, is
isomorphic to $Q_1\Join Q_2$  associated with some matched pair of
Leibniz conformal algebras.
\end{pro}



\subsection{Conformal Manin triples of Leibniz conformal algebras}
\begin{defi}
Let $M$ be a ${\bf k}[\partial]$-module. A {\bf conformal bilinear form} on $M$ is a ${\bf k}$-bilinear map $B_\lambda(\cdot,\cdot): M\times M\rightarrow {\bf k}[\lambda]$ satisfying
\begin{equation*}
B_\lambda(\partial u,v)=-\lambda B_\lambda(u,v), \,\,\,\,\,\,\,\,\,\,B_\lambda(u,\partial v)=\lambda B_\lambda(u,v)\;\;\;\text{for all $u, v\in M$.}
\end{equation*}
A conformal bilinear form $B_\lambda(\cdot,\cdot)$ on $M$ is called {\bf skew-symmetric} if
\begin{equation*}
B_\lambda(u,v)=-B_{-\lambda}(v,u)\;\;\;\text{for all $u, v\in M$}.
\end{equation*}
\end{defi}

\begin{rmk}
If there is a conformal bilinear form on a ${\bf k}[\partial]$-module $M$, then we have a ${\bf k}[\partial]$-module homomorphism
$\phi: M\rightarrow M^{\ast c}$ sending $v$ to $\phi_v$ where $\phi_v$ is defined by
\begin{equation*}
(\phi_v)_\lambda w=B_\lambda(v,w)\;\;\;\text{for all $v, w\in M$}.
\end{equation*}
Moreover, if $\phi$ is an isomorphism of ${\bf k}[\partial]$-modules, then this conformal bilinear form is called
{\bf non-degenerate}.
\end{rmk}

\delete{
\begin{defi}
A {\bf quadratic Leibniz conformal algebra} is a Leibniz conformal algebra $(Q,[\cdot_\lambda\cdot])$ equipped with a non-degenerate skew-symmetric conformal bilinear form $B_\lambda(\cdot,\cdot)$ such that the following invariant condition holds:
\begin{equation}\label{stlca}
B_\mu (a,[b_\lambda c])=B_{-\lambda}([a_\mu c]+[c_{-\mu-\partial}a],b)\;\;\;\text{for all $a, b, c\in Q$}.
\end{equation}
Denote it by $(Q, B_\lambda(\cdot,\cdot))$.
\end{defi}

\begin{ex}
Let $(Q_{f(\lambda, \partial)}={\bf k}[\partial] a\oplus{\bf k}[\partial]b, [\cdot_\lambda \cdot])$ be the Leibniz conformal algebra defined by
\begin{equation}\label{abf}
[a_\lambda a]=f(\lambda,\partial)b,\qquad [a_\lambda b]=[b_\lambda a]=[b_\lambda b]=0,
\end{equation}
where $f(\lambda,\partial)\in {\bf k}[\lambda,\partial]$. 
If the polynomial $f(\lambda,\partial)$ satisfies
\begin{eqnarray}
 f(\lambda,\partial)=-f(\partial,\lambda)-f(-\lambda-\partial,\lambda),
\end{eqnarray}
then $(Q_{f(\lambda, \partial)}, B_\lambda(\cdot,\cdot))$ is a quadratic Leibniz conformal algebra with the following conformal bilinear form:
\begin{equation} \label{abb}
B_\lambda(a,b)=-B_\lambda(b,a)=1,\qquad B_\lambda(a,a)=B_\lambda(b,b)=0.
\end{equation}
In particular, $(Q_{\lambda}, B_\lambda(\cdot,\cdot))$ is a quadratic Leibniz conformal algebra. \delete{then $(A={\bf k}[\partial] a\oplus{\bf k}[\partial]b,[\cdot_\lambda\cdot] )$ where $[\cdot_\lambda\cdot ]$ is given by \eqref{abf}  is a STLeibniz conformal algebra with the conformal bilinear form $B_\lambda(\cdot,\cdot)$ defined by \eqref{abb}.}
\end{ex}

\begin{ex}
Let $(A,\omega(\cdot,\cdot))$ be a skew-symmetric quadratic Leibniz algebra given in \cite{F}, that is, $A$ is a Leibniz algebra equipped with a non-degenerate skew-symmetric bilinear form $\omega(\cdot,\cdot)$ satisfying
$$
\omega(a,[b,c])=\omega([a,c]+[c,a],b),\;\;\; a, b, c\in A.
$$
Then $(\text{Cur}~ A,\omega_\lambda(\cdot,\cdot))$ is a quadratic Leibniz conformal algebra with a conformal bilinear form $\omega_\lambda(\cdot,\cdot) $ defined by
$$
\omega_\lambda(f(\partial)a,g(\partial)b)=f(-\lambda)g(\lambda)\omega(a,b),\;\;f(\partial),g(\partial)\in{\bf k}[\partial],\;\;a, b\in A.
$$
\end{ex}}

\begin{defi}
A {\bf conformal Manin triple} of Leibniz conformal algebras is a triple of finite Leibniz conformal algebras $(Q,Q_0,Q_1)$, where
\begin{enumerate}
\item 
there is a non-degenerate skew-symmetric conformal bilinear form $B_\lambda(\cdot,\cdot)$ on $Q$ such that the following invariant condition holds:
\begin{equation}\label{stlca}
B_\mu (a,[b_\lambda c])=B_{-\lambda}([a_\mu c]+[c_{-\mu-\partial}a],b)\;\;\;\text{for all $a, b, c\in Q$},
\end{equation}
\item $Q_0$ and $Q_1$ are Leibniz conformal subalgebras of $Q$, where $Q=Q_0\oplus Q_1$ as a direct sum of ${\bf k}[\partial]$-modules,
\item $Q_0$ and $Q_1$ are isotropic with respect to $B_\lambda(\cdot,\cdot)$, i.e. $B_\lambda(Q_i,Q_i)=0$ for $i=0,1$.
\end{enumerate}

\end{defi}
\delete{\begin{rmk}
Let $(Q,Q_0,Q_1)$ be a conformal Manin triple of Leibniz conformal algebras. By Proposition \ref{mplca}, $Q$ is isomorphic to  $Q_0\Join Q_1$  associated with some matched pair of Leibniz conformal algebras $Q_0$ and $Q_1$.
\end{rmk}

\cm{I doubt whether we should keep the above remark since it seems
quite obvious}}
\begin{ex}
Let $Q$ be a finite Leibniz conformal algebra which is free as a ${\bf k}[\partial]$-module. Then $(Q\ltimes_{L^*,-L^*-R^*} Q^{*c},Q,Q^{*c} )$ is a conformal Manin triple of Leibniz conformal algebras, where the non-degenerate skew-symmetric conformal bilinear form $B_\lambda(\cdot,\cdot)$ on $Q\ltimes_{L^*,-L^*-R^*} Q^{*c}$  is given by
\begin{equation}\label{mt}
B_\lambda(a+f,b+g)=\langle f,b\rangle_\lambda-\langle g,a\rangle_{-\lambda}\;\;\;\text{for all $a, b\in Q$ and $f, g\in Q^{\ast c}$}.
\end{equation}
\delete{Since for all $a,b,c\in A $ and $x,y,z\in A^{*c}$ , we have
\begin{align*}
&B_\mu(a+x,[(b+y)_\lambda (c+z)])=B_\mu(a+x,[b_\lambda c]+L^*(b)_\lambda z-(L^*+R^*)(c)_{\lambda-\partial}y )\\
=&\langle x,[b_\lambda c]\rangle_\mu-\langle L^*(b)_\lambda z- (L^*+R^*)(c)_{\lambda-\partial}y ,a\rangle_{-\mu}\\
=&\langle x,[b_\lambda c]\rangle_\mu+\langle z,[b_\lambda a]\rangle_{-\lambda-\mu}-\langle y,[c_{-\lambda-\mu} a]+[a_\mu c]\rangle_\lambda,
\end{align*}
and
\begin{align*}
&B_{-\lambda}([(a+x)_\mu (c+z)]+[(c+z)_{-\mu-\partial} (a+x)],b+y)\\
=&B_{-\lambda}([a_\mu c]-R^*(c)_{-\mu-\partial}x+[c_{-\mu-\partial}a]-R^*(a)_\mu z,b+y )\\
=&-\langle R^*(c)_{-\mu-\partial}x+ R^*(a)_\mu z,b\rangle_{-\lambda}-\langle  y, [a_\mu c]+[c_{-\mu-\partial}a]\rangle_\lambda\\
=&\langle x,[b_\lambda c]\rangle_\mu+\langle z,[b_\lambda a]\rangle_{-\lambda-\mu}-\langle  y, [a_\mu c]+[c_{-\mu-\partial}a]\rangle_\lambda,
\end{align*}
then the conformal bilinear form $B_\lambda(\cdot,\cdot)$ satisfies invariant condition \eqref{stlca}.
Thus $(A\ltimes_{L^*,-L^*-R^*} A^{*c},R,A^{*c} )$ is a conformal Manin triple of Leibniz conformal algebras.}

\end{ex}
In the sequel, for convenience, we set $\langle a,f\rangle_{-\lambda}:=-\langle f,a\rangle_\lambda$ for all $a\in Q$ and $f\in Q^{\ast c}$.
\begin{thm}\label{thmm}
Let $Q$ and $Q^{*c}$ be two finite Leibniz conformal algebras which are free as  ${\bf k}[\partial]$-modules. Then $(Q\bowtie Q^{*c},Q,Q^{*c})$ is a conformal Manin triple of Leibniz conformal algebras where the invariant skew-symmetric conformal bilinear form on $Q\oplus Q^{*c}$ is given by \eqref{mt} if and only if $(Q,Q^{*c};(L^*,-L^*-R^*),(\mathcal{L}^*,-\mathcal{L}^*-\mathcal{R}^*))$ is a matched pair of Leibniz conformal algebras.
\end{thm}
\begin{proof}
Let $(Q,Q^{*c};(L^*,-L^*-R^*),(\mathcal{L}^*,-\mathcal{L}^*-\mathcal{R}^*))$ be a matched pair of Leibniz conformal algebras. Then there is a Leibniz conformal algebra structure on $Q\oplus Q^{*c}$ with the $\lambda$-bracket given by \eqref{bp1}. Obviously, $B_\lambda(\cdot,\cdot)$ is non-degenerate and skew-symmetric. For all $a,b,c\in Q$ and $f,g,h\in Q^{*c}$,
we have
\begin{align*}
&B_\mu(a+f,[{(b+g)}_\lambda (c+h)])\\
=&B_\mu(a+f,[b_\lambda c]+\mathcal{L}^*(g)_\lambda c +(-\mathcal{L}^*-\mathcal{R}^*)(h)_{-\lambda-\partial}b+[g_\lambda z]+L^*(b)_\lambda h+(-L^*-R^*)(c)_{-\lambda-\partial}g )\\
=&\langle f,[b_\lambda c]\rangle_\mu+\langle f,\mathcal{L}^*(g)_\lambda c\rangle_\mu -\langle f,\mathcal{L}^*(h)_{-\lambda-\partial}b+\mathcal{R}^*(h)_{-\lambda-\partial}b\rangle_\mu
-\langle[g_\lambda h ],a\rangle_{-\mu}\\
&-\langle L^*(b)_\lambda h,a\rangle_{-\mu}+\langle L^*(c)_{-\lambda-\partial}g,a\rangle_{-\mu}+\langle R^*(c)_{-\lambda-\partial}g,a\rangle_{-\mu}\\
=&\langle f,[b_\lambda c]\rangle_\mu+\langle c,[g_\lambda f]\rangle_{-\mu-\lambda}-\langle b,[h_{-\mu-\lambda}f]\rangle_\lambda-\langle b,[f_\mu h]\rangle_\lambda+\langle a,[g_\lambda h]\rangle_\mu\\
&+\langle h,[b_\lambda a]\rangle_{-\lambda-\mu}-\langle g,[c_{-\lambda-\mu }a]\rangle_\lambda-\langle g,[a_\mu c]\rangle_\lambda.
\end{align*}
On the other hand, we have
\begin{align*}
&B_{-\lambda}( [{(a+f)}_\mu (c+h)]+[{(c+h)}_{-\mu-\partial} (a+f)] ,b+g)\\
=&B_{-\lambda}([a_\mu c]+\mathcal{L}^*(f)_\mu c-(\mathcal{L}^*+\mathcal{R}^*)(h)_{-\mu-\partial}a+[f_\mu h]+L^*(a)_\mu h-(L^*+R^*)(c)_{-\mu-\partial}f,b+g)\\
&+B_{-\lambda}([c_{-\mu-\partial} a]+\mathcal{L}^*(h)_{-\mu-\partial} a-(\mathcal{L}^*+\mathcal{R}^*)(f)_\mu c+[h_{-\mu-\partial} f]+L^*(c)_{-\mu-\partial} f-(L^*+R^*)(a)_\mu h,b+g)\\
=&\langle [a_\mu c]+[c_{-\mu-\partial }a],g\rangle_{-\lambda}-\langle c,[f_\mu g]\rangle_{-\lambda-\mu}+\langle a,[h_{-\lambda-\mu}g]\rangle_\mu+\langle a,[g_\lambda h]\rangle_\mu\\
&+\langle c,[g_\lambda f]\rangle_{-\lambda-\mu}+\langle[f_\mu h]+[h_{-\mu-\partial}f],b\rangle_{-\lambda}+\langle f,[b_\lambda c]\rangle_\mu+\langle h,[b_\lambda a]\rangle_{-\lambda-\mu}\\
=&\langle f,[b_\lambda c]\rangle_\mu+\langle c,[g_\lambda f]\rangle_{-\lambda-\mu}-\langle b,[h_{-\lambda-\mu} f]\rangle_\lambda-\langle b,[f_\mu h]\rangle_\lambda
+\langle a,[g_\lambda h]\rangle_\mu\\
&+\langle h,[b_\lambda a]\rangle_{-\lambda-\mu}-\langle g,[c_{-\lambda-\mu}a]\rangle_\lambda-\langle g,[a_\mu c]\rangle_\lambda.
\end{align*}
Thus $B_{\lambda}(\cdot,\cdot)$ satisfies  \eqref{stlca}. Therefore, $(Q\bowtie Q^{*c},Q,Q^{*c})$ is a conformal Manin triple of Leibniz conformal algebras where the invariant skew-symmetric conformal bilinear form $B_{\lambda}(\cdot,\cdot)$ on $Q\oplus Q^{*c}$ is given by \eqref{mt}.

Conversely, let $(Q\Join Q^{*c},Q,Q^{*c})$ be a conformal Manin triple of Leibniz conformal algebras with the non-degenerate skew-symmetric invariant bilinear form $B_\lambda(\cdot,\cdot)$ given by \eqref{mt}. Assume that $Q\Join Q^{*c}$ is a bicrossed product associated with a matched pair $(Q_1,Q_2;(l,r),(\phi,\psi))$ of Leibniz conformal algebras. Therefore, for all $a\in Q$, $f\in Q^{*c}$, assume that
\begin{align*}
&[a_\lambda f]=l(a)_\lambda f+\psi(f)_{-\lambda-\partial}a,\;\;[f_\lambda a]=\phi(f)_\lambda a+r(a)_{-\lambda-\partial}f.
\end{align*}
\delete{for some ${\bf k}[\partial]$-module homomorphisms $l,~r:Q\rightarrow \text{Cend}(Q^{*c})$ and $\phi,~\psi:Q^{*c}\rightarrow \text{Cend}(Q)$. }Let $a$, $b\in Q$ and $f$, $g\in Q^{*c}$. By \eqref{stlca}, we get
\begin{eqnarray*}
&&\langle g,\psi(f)_{-\lambda-\partial}a\rangle_\mu=B_\mu(g,[a_\lambda f])=B_{-\lambda}([g_\mu f]+[f_{-\mu-\partial}g], a)=-\langle a,[g_\mu f]+[f_{-\mu-\lambda}g]\rangle_\lambda\\
&&\qquad=\langle\mathcal{L}^*(f)_{-\lambda-\partial}a+\mathcal{R}^*(f)_{-\lambda-\partial}a,g\rangle_{-\mu}=-\langle g,(\mathcal{L}^*+\mathcal{R}^*)(f)_{-\lambda-\partial} a\rangle_\mu,
\end{eqnarray*}
which means that $\psi=-\mathcal{L}^*-\mathcal{R}^*$. Moreover, we have
\begin{align*}
&\langle l(a)_\lambda f,b\rangle_\mu=B_{\mu}([a_\lambda f],b)
=-B_{-\lambda}([b_{-\mu}f]+[f_{\mu-\partial}b], a )=-\langle f,[a_\lambda b]\rangle_{\mu-\lambda}
=\langle L^*(a)_\lambda f,b\rangle_\mu,
\end{align*}
which means that $l=L^*$. Similarly, we can get $\phi=\mathcal{L}^*$ and $r=-L^*-R^*$.
Hence, $(Q,Q^{*c};(L^*,-L^*-R^*),(\mathcal{L}^*,-\mathcal{L}^*-\mathcal{R}^*))$ is a matched pair of Leibniz conformal algebras.
\end{proof}

\section{Leibniz conformal bialgebras}\label{3}
We introduce the notion of  Leibniz conformal bialgebras and then
establish the equivalences among some matched pairs of Leibniz
conformal algebras, conformal Manin triples of Leibniz conformal
algebras, and Leibniz conformal bialgebras.

First, we introduce the notion of Leibniz conformal coalgebras.
\begin{defi}
A {\bf Leibniz conformal coalgebra} is a ${\bf k}[\partial]$-module $Q$ endowed with a ${\bf k}[\partial]$-module homomorphism $\Delta :Q\rightarrow Q\otimes Q$ such that
\begin{equation}\label{co-eq}
(\id\otimes \Delta)\Delta(a)=(\Delta\otimes \id)\Delta(a)+\tau_{12} (\id\otimes \Delta)\Delta (a)\;\;\;\text{for all $a\in Q$},
\end{equation}
where the module action of ${\bf k}[\partial] $ on $Q\otimes Q$ is defined as
$\partial (a\otimes b)=(\partial a)\otimes b+a\otimes (\partial b)$ for all $a, b\in Q$.
\end{defi}

Let $Q$ be a ${\bf k}[\partial]$-module. For all $f, g\in Q^{*c}$ and $a,b\in Q$, we define
\begin{equation*}
\langle f\otimes g,a\otimes b\rangle_{\lambda,\mu}=\langle f,a\rangle_\lambda\langle g,b\rangle_\mu.
\end{equation*}

\begin{pro}\label{ppc}
\begin{enumerate}
\item Let $(Q,\Delta)$ be a finite Leibniz conformal coalgebra. Then $Q^{*c}$ is a Leibniz conformal algebra with the following $\lambda$-bracket:
\begin{equation}\label{ca}
\langle[f_\lambda g],a\rangle_\mu=\langle f\otimes g,\Delta (a)\rangle_{\lambda,\mu-\lambda}\;\;\; \text{for all $f, g\in Q^{\ast c}$ and $a\in Q$}.
\end{equation}
\item Let $Q$ be a finite Leibniz conformal algebra which is free as a  ${\bf k}[\partial]$-module, that is, $Q=\sum^n_{i=1}{\bf k}[\partial]a^i$, where $\{a^1, \cdots, a^n \}$ is a ${\bf k}[\partial]$-basis of $Q$. Then $Q^{*c}=\text{Chom}(A,{\bf k})=\sum^n_{i=1}{\bf k}[\partial]a_i $, where $\{a_1, \cdots, a_n \}$ is the dual ${\bf k}[\partial]$-basis in the sense that $(a_i)_\lambda a^j=\delta_{ij} $, is a Leibniz conformal coalgebra with the following co-bracket:
\begin{equation}\label{cob}
\Delta(f)=\sum_{i,j}f_\mu([a^i_\lambda a^j])(a_i\otimes a_j)|_{\lambda=\partial\otimes \id,\mu=-\partial\otimes \id-\id\otimes\partial}.
\end{equation}
More precisely, if $[a^i_\lambda
a^j]=\sum_k P^{ij}_k(\lambda,\partial)a^k$, where
$P^{ij}_k(\lambda,\partial)\in {\bf k}[\lambda,\partial]$, then
the co-bracket is
\begin{equation}
\Delta(a_k)=\sum_{i,j}Q^{ij}_k(\partial\otimes \id,\id\otimes\partial)a_i\otimes a_j,
\end{equation}
where $Q^{ij}_k(x,y)=P^{ij}_k(x,-x-y)$.
\end{enumerate}
\end{pro}
\begin{proof}
It follows directly by the proof of \cite[Proposition 2.13]{L}.
\end{proof}

\delete{It is easy to check that the $\lambda$-bracket  defined by \eqref{ca} satisfies \eqref{cs}.
For all $a\in R$ and $f,g,h\in A^{*c}$, we have
\begin{align*}
&\langle [f_\lambda[g_\mu h]],a  \rangle_\theta=\langle f\otimes [g_\mu h],\Delta (a) \rangle_{\lambda,\theta-\lambda}=\langle f\otimes g\otimes h,(Id\otimes\Delta)\circ \Delta(a)\rangle_{\lambda,\mu,\theta-\lambda-\mu},\\
&\langle [[f_\lambda g]_{\lambda+\mu} h],a  \rangle_\theta=\langle [f_\lambda g]\otimes h,\Delta (a) \rangle_{\lambda+\mu,\theta-\lambda-\mu}=\langle f\otimes g\otimes h, (\Delta\otimes Id)\circ\Delta(a) \rangle_{ \lambda,\mu,\theta-\lambda-\mu},\\
&\langle [g_\mu[f_\lambda h]],a \rangle_\theta=\langle g\otimes [f_\lambda h],\Delta(a)\rangle_{\mu,\theta-\mu}=\langle f\otimes g\otimes h, \tau_{12}(Id\otimes\Delta)\circ\Delta(a) \rangle_{ \lambda,\mu,\theta-\lambda-\mu},
\end{align*}
and then by \eqref{co-eq}, we get  the $\lambda$-bracket satisfies \eqref{li}. Thus the conclusion holds.
\end{proof}

\begin{pro} \label{cb}
Let $A$ be a finite Leibniz conformal algebra which is free as a  ${\bf k}[\partial]$-module, that is, $A=\sum^n_{i=1}{\bf k}[\partial]a^i$, where $\{a^i \}_{i=1}^n$ is a ${\bf k}[\partial]$-basis of $A$. Then $A^{*c}=\text{Chom}(A,\mathbb{C})=\sum^n_{i=1}{\bf k}[\partial]a_i $, where $\{a_i \}_{i=1}^n$ is a dual ${\bf k}[\partial]$-basis, is a Leibniz conformal coalgebra with the following co-bracket:
\begin{equation}\label{cob}
\Delta(f)=\sum_{i,j}f_\mu([a^i_\lambda a^j])(a_i\otimes a_j)|_{\lambda=\partial\otimes Id,\mu=-\partial\otimes Id-Id\otimes\partial}.
\end{equation}
Moreover precisely, if $[a^i_\lambda a^j]=\sum_k P^{ij}_k(\lambda,\partial)a^k$, where $P^{ij}_k$ are some polynomials in $\lambda$ and $\partial$, then the co-bracket is
\begin{equation}
\Delta(a_k)=\sum_{i,j}Q^{ij}_k(\partial\otimes Id,Id\otimes\partial)a_i\otimes a_j,
\end{equation}
where $Q^{ij}_k(x,y)=P^{ij}_k(x,-x-y)$.
\end{pro}
\begin{proof}
By \eqref{cob}, we have
\begin{align*}
\Delta(a_k)&=\sum_{i,j}{a_k}_\mu([{a^i}_\lambda a^j])(a_i\otimes a_j)|_{\lambda=\partial\otimes Id,\mu=-\partial\otimes Id-Id\otimes\partial}\\
&=\sum_{i,j}{a_k}_\mu(\sum_s P^{ij}_s(\lambda,\partial)a^s )(a_i\otimes a_j)|_{\lambda=\partial\otimes Id,\mu=-\partial\otimes Id-Id\otimes\partial}\\
&=\sum_{i,j }P^{ij}_k(\lambda,\mu)(a_i\otimes a_j)|_{\lambda=\partial\otimes Id,\mu=-\partial\otimes Id-Id\otimes\partial}\\
&=\sum_{i,j } Q^{ij}_k(\partial\otimes Id,Id\otimes\partial)a_i\otimes a_j.
\end{align*}
Next, we just need to check \eqref{co-eq}. We have
\begin{align*}
&(Id\otimes \Delta)\Delta(a_k)\\
=&(Id\otimes \Delta)\sum_{i,j} Q^{ij}_k(\partial\otimes Id,Id\otimes \partial)a_i\otimes a_j\\
=&\sum_{i,j,l,r}Q^{ij}_k(\partial\otimes Id\otimes Id,Id\otimes\partial\otimes Id+Id\otimes Id\otimes\partial)Q^{lr}_j(Id\otimes\partial\otimes Id,Id\otimes Id\otimes\partial)(a_i\otimes a_l\otimes a_r )
\end{align*}
and
\begin{align*}
&(\Delta\otimes Id)\Delta(a_k)\\
=&\sum_{i,j,l,r}Q^{ij}_k(\partial\otimes Id\otimes Id+Id\otimes\partial\otimes Id,Id\otimes Id\otimes\partial)Q^{lr}_i(\partial\otimes Id\otimes Id,Id\otimes\partial\otimes Id)(a_l\otimes a_r\otimes a_j )\\
=&\sum_{i,j,l,r}Q^{jr}_k(\partial\otimes Id\otimes Id+Id\otimes\partial\otimes Id,Id\otimes Id\otimes\partial)Q^{rl}_j(\partial\otimes Id\otimes Id,Id\otimes\partial\otimes Id)(a_i\otimes a_l\otimes a_r ).
\end{align*}
On the other hand, by \eqref{li} we can get
\begin{align*}
&[a^i_\lambda[a^l_\mu a^r]]-[[a^i_\lambda a^l]_{\lambda+\mu}a^r ]-[a^l_\mu[a^i_\lambda a^r] ]=0\\
=&[a^i_\lambda (\sum_j P^{lr}_j(\mu,\partial)a^j)]-[(\sum_j P^{il}_j(\lambda,\partial) a^j)_{\lambda+\mu}a^r]
-[a^l_\mu (P^{ir}_j(\lambda,\partial)a^j)]\\
=&\sum_{j,k}(P^{lr}_j(\mu,\lambda+\partial)P_k^{ij}(\lambda,\partial)
-P^{il}_j(\lambda,-\lambda-\mu)P^{jr}_k(\lambda+\mu,\partial)-P^{ir}_j(\lambda,\mu+\partial)P^{lj}_k(\mu,\partial))a^k,
\end{align*}
which means that
\begin{equation*}
\sum_{j}P^{lr}_j(\mu,\lambda+\partial)P_k^{ij}(\lambda,\partial)=
\sum_{j}P^{il}_j(\lambda,-\lambda-\mu)P^{jr}_k(\lambda+\mu,\partial)+\sum_{j}P^{ir}_j(\lambda,\mu+\partial)P^{lj}_k(\mu,\partial)).
\end{equation*}
By $P^{ij}_k(x,-x-y) = Q^{ij}_k(x,y)$, we can get \eqref{co-eq}.
Hence the conclusion holds.
\end{proof}}

\begin{cor}
 \delete{Let $(Q={\bf k}[\partial]a,\Delta)$ be a Leibniz conformal coalgebra. which is free and rank 1 as a ${\bf k}[\partial]$-module and $\Delta (p(\partial)a)=p(\partial^{\otimes^2})(\partial a\otimes a-a\otimes \partial a)$ with $\partial^{\otimes^2}=\partial \otimes \id+\id\otimes \partial$. Then it is a Leibniz conformal coalgebra and $(Q^{*c},\Delta^{*c})$ is isomorphic to Vir as Lie conformal algebras. Moreover,}
 Let $(Q, \Delta)$ be a Leibniz conformal coalgebra where $Q$ is free of rank $1$ as a ${\bf k}[\partial]$-module. Then $(Q, \Delta)\simeq ({\bf k}[\partial]a, \Delta^\prime)$ with $\Delta^\prime(a)=0$ or $\Delta^\prime(a)=c(\partial a\otimes a-a\otimes \partial a)$, $c\in {\bf k}\backslash\{0\}$. 
\end{cor}
\begin{proof}
It follows directly by Propositions \ref{ll} and \ref{ppc}.
\end{proof}

In the sequel, we always set $
\partial^{\otimes^2}:=\partial \otimes \id+\id\otimes \partial.
$

Next, we introduce the definition of Leibniz conformal bialgebras.
\begin{defi}
Let $(Q,[\cdot_\lambda\cdot])$ be a Leibniz conformal algebra and $(Q,\Delta)$ be a Leibniz conformal coalgebra. Then $(Q,[\cdot_\lambda\cdot],\Delta)$ is called a {\bf Leibniz conformal bialgebra} if the following conditions hold:
\begin{align}
\label{lb1}&(R(a)_{\lambda-\partial\otimes \id}\otimes \id) \Delta(b)=\tau ((R(b)_{-\lambda-\partial\otimes \id}\otimes \id) \Delta(a)),\\
\label{lb2}& \Delta([a_\lambda b])=(\id\otimes R(b)_{-\lambda-\partial^{\otimes^2}}-L(b)_{-\lambda-\partial^{\otimes^2}}\otimes \id-R(b)_{-\lambda-\partial^{\otimes^2}}\otimes \id) (\tau+\id)\Delta(a)\\
&\qquad \qquad  +(\id\otimes L(a)_\lambda+L(a)_\lambda\otimes \id)\Delta(b)\;\;\;\text{for all $a, b\in Q$.} \nonumber
\end{align}
\end{defi}

We give a remark to illustrate the
relationship between Leibniz bialgebras and Leibniz conformal
algebras.
\begin{rmk} Recall \cite{TS} that a Leibniz bialgebra is a triple $(A, [\cdot,\cdot], \delta)$, where $(A, [\cdot,\cdot])$ is a Leibniz algebra, $(A, \delta)$ is a Leibniz coalgebra in the sense that
    $$
    (\id\otimes \delta)\delta(a)=(\delta\otimes \id)\delta(a)+\tau_{12} (\id\otimes \delta)\delta (a)\;\;\;\text{for all $a\in A$},
    $$
 and they satisfy the following equalities
\begin{align*}
&(R_A(a)\otimes \id) \delta(b)=\tau ((R_A(b)\otimes \id) \delta(a)),\\
&\delta([a, b])=(\id\otimes R_A(b)-L_A(b)\otimes \id-R_A(b)\otimes \id) (\tau+\id)\delta(a) +(\id\otimes L_A(a)+L_A(a)_\lambda\otimes \id)\delta(b),
\end{align*}
for all $a, b\in A$, where $L_A(a)(b)=[a,b]=R_A(b)(a)$.

Let $(A, [\cdot,\cdot], \delta)$ be a Leibniz bialgebra. Then $(\text{Cur}\;A={\bf k}[\partial]\otimes A, [\cdot_\lambda \cdot], \Delta)$ is a Leibniz conformal bialgebra, where $(\text{Cur}\;A, [\cdot_\lambda \cdot])$ is the current Leibniz conformal algebra associated with  $(A, [\cdot,\cdot])$ and $\Delta$ is a ${\bf k}[\partial]$-module homomorphism defined by
\begin{eqnarray*}
\Delta(a)=\delta(a)\;\;\;\text{for all $a\in A$.}
\end{eqnarray*}
$(\text{Cur}\;A={\bf k}[\partial]\otimes A, [\cdot_\lambda \cdot], \Delta)$ is called the {\bf current Leibniz conformal bialgebras} associated with $(A, [\cdot,\cdot], \delta)$.

Conversely, let $(Q,[\cdot_\lambda\cdot],\Delta)$ be a Leibniz conformal bialgebra and $\pi: Q\rightarrow Q/\partial Q$ be the natural projection. Set $[a_\lambda b]=\sum_{i=0}^{n_{a,b}}\frac{\lambda^i}{i!}a_{(i)}b$ for all $a$, $b\in Q$. Then it is easy to see that $(Q/\partial Q, [\cdot,\cdot], \delta)$ is a Leibniz bialgebra, where $[\cdot,\cdot]$ and $\delta$ are defined by
\begin{eqnarray*}
[a+\partial Q, b+\partial Q]:=\pi(a_{(0)} b),
 \;\;\;\delta(a+\partial Q):=(\pi\otimes \pi)\Delta(a)\;\;\; \text{for all $a$, $b\in Q$.}
\end{eqnarray*}
Hence in the above sense, Leibniz conformal bialgebras are regarded as ``conformal analogues" of Leibniz bialgebras.
\end{rmk}

\begin{pro} \label{pmb}
Let $(Q, [\cdot_\lambda \cdot])$ be a finite Leibniz conformal algebra which is free as a ${\bf k}[\partial]$-module. Suppose that there is a Leibniz conformal algebra structure on $Q^{\ast c}$
which is obtained from a ${\bf k}[\partial]$-module homomorphism $\Delta: Q\rightarrow Q\otimes Q$.
Then $(Q,Q^{*c};(L^*,-L^*-R^*),(\mathcal{L}^*,-\mathcal{L}^*-\mathcal{R}^*))$ is a matched pair of Leibniz conformal algebras if and only if $(Q,[\cdot_\lambda\cdot],\Delta)$ is a Leibniz conformal bialgebra.
\end{pro}
\begin{proof}
Let $\{a^1, \cdots, a^n\}$ be a ${\bf k}[\partial]$-basis of $Q$, $\{a_1, \cdots, a_n\}$ be the dual ${\bf k}[\partial]$-basis of $Q^{*c}$ and $I=\{1, \cdots, n\}$. Set
\begin{equation*}
[{a^i}_\lambda a^j]=\sum_k P^{ij}_k(\lambda,\partial) a^k,\qquad \quad [{a_i}_\lambda a_j]=\sum_k H^{ij}_k(\lambda,\partial) a_k,
\end{equation*}
where $P^{ij}_k(\lambda,\partial) $, $H^{ij}_k(\lambda,\partial) \in {\bf k}[\lambda,\partial]$.
 By  Proposition \ref{ppc}, we get
$\Delta(a^k)=\sum_{i,j}Q_k^{ij}(\partial\otimes \id, \id\otimes \partial)a^i\otimes a^j$, where $Q^{ij}_k(x,y)=H^{ij}_k(x,-x-y)$.
Since
\begin{equation*}
\langle \mathcal{L}^*(a_i)_\lambda a^j, a_k\rangle_\mu=-\langle{a^j},[{a_i}_\lambda a_k]\rangle_{\mu-\lambda}=-H^{ik}_j(\lambda,\mu-\lambda),
\end{equation*}
we get $\mathcal{L}^*(a_i)_\lambda a^j=-\sum_k H^{ik}_j(\lambda,-\lambda-\partial)a^k $.
Similarly, we can get
\begin{align*}
&\mathcal{R}^*(a_i)_\lambda a^j=-\sum_k H^{ki}_j(\partial,-\lambda-\partial)a^k,\qquad\qquad
L^*(a^i)_\lambda a_j=-\sum_k P^{ik}_j(\lambda,-\lambda-\partial)a_k, \\
&R^*(a^i)_\lambda a_j=-\sum_k P^{ki}_j(\partial,-\lambda-\partial)a_k.
\end{align*}
Considering the coefficient of $a^j\otimes a^h$ in
$(R(a^i)_{\lambda-\partial\otimes \id}\otimes \id)\Delta(a^k)=\tau ((R(a^k)_{-\lambda-\partial\otimes \id}\otimes \id) \Delta(a^i))$,
we get
\begin{align}\label{ba1}
\sum_{t}H^{th}_k(\lambda,-\lambda-\id\otimes\partial)P^{ti}_j(-\lambda,\partial\otimes \id)=
\sum_{t}H^{tj}_i(-\lambda,\lambda-\partial\otimes \id)P^{tk}_h(\lambda,\id\otimes\partial).
\end{align}
On the other hand, \eqref{mp6} holds for all $a,b\in Q$ and $x\in Q^{*c}$ if and only if
\begin{align*}
-\langle\mathcal{L}^*(R^*(a^i)_{-\lambda-\partial}a_j)_{\lambda+\mu}a^k, a_h\rangle_\theta
=\langle[(\mathcal{R}^*(a_j)_{-\mu-\partial}a^i)_{\lambda+\mu}a^k], a_h\rangle_\theta\;\;\;\text{for all $i,j,k,h\in I$},
\end{align*}
holds if and only if
\begin{equation}\label{ba2}
\sum_{t}H^{th}_k(\lambda+\mu,-\lambda-\mu+\theta)P^{ti}_j(-\lambda-\mu,\lambda)=
\sum_{t}H^{tj}_i(-\lambda-\mu,\mu)P^{tk}_h(\lambda+\mu,-\theta)\;\;\;\text{for all $i,j,k,h\in I$},
\end{equation}
holds.
Obviously, \eqref{ba1} is equivalent to \eqref{ba2} by replacing $\lambda$ by $\lambda+\mu$, $\id\otimes \partial $ by $-\theta$ and $\partial\otimes \id $ by $\lambda$ in \eqref{ba1}. Then \eqref{lb1} holds if and only if \eqref{mp6} holds.

 Note that \eqref{mp5} holds for all $a,b\in Q$ and $x\in Q^{*c}$ if and only if for all $ i,j,k,h\in I$,
\begin{align*}
&\langle[{a^i}_\lambda (\mathcal{L}^*(a_j)_\mu a^k)]+(\mathcal{L}^*+\mathcal{R}^*)((L^*+R^*)(a^k)_{-\mu-\partial} a_j)_{-\lambda-\partial}a^i, a_h\rangle_\theta\\
=&\langle\mathcal{L}^*(a_j)_\mu [{a^i}_\lambda a^k]-[(\mathcal{L}^*(a_j)_\mu a^i)_{\lambda+\mu}a^k]+\mathcal{L}^*((L^*+R^*)(a^i)_{-\mu-\partial}a_j)_{\lambda+\mu}a^k, a_h\rangle_\theta
\end{align*}
holds, which is equivalent to
{\small \begin{equation}\label{ba3}
\begin{aligned}
&\sum_t(P^{kt}_j(-\lambda-\mu+\theta,\mu)+P^{tk}_j(\lambda-\theta,\mu))(H^{th}_i(-\lambda+\theta,\lambda)+H^{ht}_i(-\theta,\lambda))\\
&-\sum_t H^{jt}_k(\mu,-\mu-\lambda+\theta)P^{it}_h(\lambda,-\theta)= -\sum_t P^{ik}_t(\lambda,\mu-\theta)H^{jh}_t(\mu,-\mu+\theta)\\
&+ \sum_t(H^{jt}_i(\mu,\lambda)P^{tk}_h(\lambda+\mu,-\theta)+(P^{it}_j(\lambda,\mu)+P^{ti}_j(-\lambda-\mu,\mu))H^{th}_k(\lambda+\mu,-\lambda-\mu+\theta)).
\end{aligned}
\end{equation}}
On the other hand, \eqref{lb2} holds if and only if for all $i,k\in I$,
\begin{equation}\label{ba4}
\begin{aligned}
\Delta([{a^i}_\lambda a^k])=&(\id\otimes R(a^k)_{-\lambda-\partial^{\otimes^2}}-L(a^k)_{-\lambda-\partial^{\otimes^2}}\otimes \id-R(a^k)_{-\lambda-\partial^{\otimes^2}}\otimes \id )(\tau+\id)\Delta(a^i)\\
&+(\id\otimes L(a^i)_\lambda+L(a^i)_\lambda\otimes \id)\Delta(a^k)
\end{aligned}
\end{equation}
holds. Considering the coefficient of $a^j\otimes a^h$ in \eqref{ba4}, we get
{\small
\begin{equation}\label{ba5}
\begin{aligned}
&\sum_t P^{ik}_t(\lambda,\partial^{\otimes^2})H^{jh}_t(\partial\otimes \id,-\partial^{\otimes^2})
=\sum_t (H^{jt}_i(\partial\otimes \id,\lambda)P^{tk}_h(\lambda+\partial\otimes \id,\id\otimes \partial) \\
&+H^{tj}_i(-\lambda-\partial\otimes \id,\lambda)P^{tk}_h(\lambda+\partial\otimes \id,\id\otimes\partial)
-H^{th}_i(-\lambda-\id\otimes\partial,\lambda)P^{kt}_j(-\lambda-\partial^{\otimes^2},\partial\otimes \id)\\
&-H^{ht}_i(\id\otimes \partial,\lambda)P^{kt}_j(-\lambda-\partial^{\otimes^2},\partial\otimes \id)
-H^{th}_i(-\lambda-\id\otimes \partial,\lambda)P^{tk}_j(\lambda+\id\otimes\partial,\partial\otimes \id)\\
&-H^{ht}_i(\id\otimes\partial,\lambda)P^{tk}_j(\lambda+\id\otimes\partial,\partial\otimes \id)
+H^{jt}_k(\partial\otimes \id,-\partial\otimes \id-\lambda-\id\otimes\partial)P^{it}_h(\lambda,\id\otimes\partial)\\
&+H^{th}_k(\lambda+\partial\otimes \id,-\lambda-\partial^{\otimes^2})P^{it}_j(\lambda,\partial\otimes \id)).
\end{aligned}
\end{equation}}
By \eqref{lb1}, we have
\begin{equation*}
\sum_t H^{th}_k(\lambda+\partial\otimes \id,-\lambda-\partial^{\otimes^2})P^{ti}_j(-\lambda-\partial\otimes \id,\partial\otimes \id)
=\sum_t H^{tj}_i(-\lambda-\partial\otimes \id,\lambda)P^{tk}_h(\lambda+\partial\otimes \id,\id\otimes\partial).
\end{equation*}
Then it is clear that \eqref{ba3} is equivalent to \eqref{ba5} by replacing $\mu$ by $\partial\otimes \id$ and $\theta$ by $-\id\otimes\partial$. Moreover, \eqref{mp4} can be deduced by \eqref{mp5} and \eqref{mp6}.

For all $a,b\in Q$ and $x,y\in Q^{*c}$, by some calculations we have
\begin{align*}
&\langle L^*(a)_\lambda([x_\mu y])-[(L^*(a)_\lambda x)_{\lambda+\mu}y]+L^*((\mathcal{L}^*+\mathcal{R}^*)(x)_{-\lambda-\partial}a)_{\lambda+\mu}y\\
\quad&-[x_\mu(L^*(a)_\lambda y)]-(L^*+R^*)((\mathcal{L}^*+\mathcal{R}^*)(y)_{-\lambda-\partial}a)_{-\mu-\partial}x ,b\rangle_\theta\\
=&\langle-\mathcal{L}^*(x)_\mu([a_\lambda b])-\mathcal{L}^*(L^*(a)_\lambda x)_{\lambda+\mu}b+[((\mathcal{L}^*+\mathcal{R}^*)(x)_{-\lambda-\partial}a)_{\lambda+\mu}b]\\
\quad&+[a_\lambda (\mathcal{L}^*(x)_\mu b)]+(\mathcal{L}^*+\mathcal{R}^*)((L^*+R^*)(b)_{-\mu-\partial}x)_{-\lambda-\partial}a  ,y \rangle_{\lambda+\mu-\theta},
\end{align*}
which means that \eqref{mp2} is equivalent to \eqref{mp5}. Similarly, 
we can obtain that \eqref{mp3} is equivalent to \eqref{mp6} and  \eqref{mp1} is equivalent to \eqref{mp4}.
Thus, $(Q,Q^{*c};(L^*,-L^*-R^*),(\mathcal{L}^*,-\mathcal{L}^*-\mathcal{R}^*))$ is a matched pair of Leibniz conformal algebras if and only if \eqref{lb1} and \eqref{lb2} hold.
\end{proof}

By Theorem \ref{thmm} and Proposition \ref{pmb}, we have the
following conclusion.
\begin{cor}\label{commb}
Let $(Q, [\cdot_\lambda \cdot])$ be a finite Leibniz conformal algebra which is free as a ${\bf k}[\partial]$-module. Suppose that there is a Leibniz conformal algebra structure on $Q^{\ast c}$
which is obtained from a ${\bf k}[\partial]$-module homomorphism $\Delta: Q\rightarrow Q\otimes Q$.
Then the following conditions are equivalent:
\begin{enumerate}
\item $(Q,Q^{*c};(L^*,-L^*-R^*),(\mathcal{L}^*,-\mathcal{L}^*-\mathcal{R}^*))$ is a matched pair of Leibniz conformal algebras;
\item $(Q\oplus Q^{*c},Q,Q^{*c}) $ is a conformal Manin triple of Leibniz conformal algebras $Q$ and $Q^{*c}$, where the invariant  skew-symmetric bilinear form $B_\lambda(\cdot,\cdot)$ on $Q\oplus Q^{*c} $ is given by \eqref{mt};
\item $(Q,[\cdot_\lambda\cdot],\Delta)$  is a Leibniz conformal bialgebra.
\end{enumerate}
\end{cor}
\delete{\begin{rmk}
Let $A$ and $A^{*c}$ be two finite Leibniz conformal algebras which are free as ${\bf k}[\partial]$-modules.
If $(A$, $A^{*c}$; $(L^*,-L^*-R^*)$, $(\mathcal{L}^*,-\mathcal{L}^*-\mathcal{R}^*))$ is a Leibniz conformal bialgebra, then $(A^{*c}$, $A$; $(\mathcal{L}^*,-\mathcal{L}^*-\mathcal{R}^*)$, $(L^*,-L^*-R^*))$ is also a Leibniz conformal bialgebra.
\end{rmk}}

\section{Coboundary Leibniz conformal bialgebras and classical Leibniz conformal Yang-Baxter equation}\label{4}
We define the notion of coboundary Leibniz conformal bialgebras
which leads us to introduce the notion of classical Leibniz
conformal Yang-Baxter equation (CLCYBE). Moreover, the notions of
$\mathcal{O}$-operators on Leibniz conformal algebras and
Leibniz-dendriform conformal algebras are introduced to  construct
symmetric solutions of the CLCYBE.

For convenience, set $\partial_1:=\partial\otimes\id\otimes \id$, $\partial_2:=\id \otimes \partial\otimes \id$, $\partial_3:=\id\otimes\id\otimes \partial$ and $\partial^{\otimes^3}:=\partial_1+ \partial_2+\partial_3$.
\subsection{Coboundary Leibniz conformal bialgebras}\label{subsection1}

Let $Q$ be a ${\bf k}[\partial]$-module. Define a linear map $F:Q\rightarrow \text{End}_{\bf k}(Q\otimes Q)$  by
\begin{equation}\label{f}
F(a)=L(a)_\lambda\otimes \id+R(a)_\lambda\otimes \id-\id\otimes
R(a)_\lambda|_{\lambda=-\partial^{\otimes^2}}\;\;\;\text{for
all $a\in Q$}.
\end{equation}
If an element $r\in Q\otimes Q$ satisfies $F(a)r=0$ for all $a\in Q$, then we call $r$  {\bf invariant}.

\begin{defi}
For a Leibniz conformal bialgebra $(Q,[\cdot_\lambda\cdot],\Delta)$, if there exists $r\in Q\otimes Q$ such that
\begin{equation}\label{dr}
\Delta(a)=\Delta_r(a)=F(a)r\;\;\;\text{for all $a\in Q$},
\end{equation}
where $F$ is given by \eqref{f}, then $(Q,[\cdot_\lambda\cdot],\Delta=\Delta_r)$ is called a {\bf coboundary Leibniz conformal bialgebra}.
\end{defi}
\begin{defi}
Let $Q$ be a Leibniz conformal algebra and $r=\sum_i a_i\otimes b_i\in Q\otimes Q$. Define
\begin{equation}
[[r,r]]=[{r_{12}}_{\mu} r_{13}]|_{\mu=\partial_2}+[{r_{23}}_{\mu} r_{13}]|_{\mu=\partial_2}-[{r_{12}}_{\mu} r_{23}]|_{\mu=\partial_1}-[{r_{23}}_{\mu} r_{12}]|_{\mu=-\partial_1-\partial_2},
\end{equation}
where
\begin{align*}
&[{r_{12}}_{\mu} r_{13}]=\sum_{i,j}[{a_i}_\mu a_j ]\otimes b_i\otimes b_j,\quad
[{r_{23}}_{\mu} r_{13}]=\sum_{i,j}{a_i}\otimes a_j \otimes [{b_j}_\mu b_i],\\
&[{r_{12}}_{\mu} r_{23}]=\sum_{i,j}a_i\otimes [{b_i}_\mu a_j ]\otimes b_j,\quad
[{r_{23}}_{\mu} r_{12}]=\sum_{i,j}a_i\otimes [{a_j}_\mu b_i ]\otimes b_j.
\end{align*}
Then $[[r,r]]=0\;\text{mod}\;(\partial^{\otimes^3} )$ in $Q\otimes Q \otimes Q$
is called the {\bf classical Leibniz conformal Yang-Baxter equation (CLCYBE)} in $Q$.
\end{defi}

\begin{pro}\label{pprtr}
Let $Q$ be a Leibniz conformal algebra  and $r\in Q\otimes
Q$. Then $r$ is a solution of the CLCYBE in $Q$ if and only if
$\tau r$ is a solution of the CLCYBE in $Q$.
\end{pro}
\begin{proof}
	Set $r=\sum_i a_i\otimes b_i$.
In $Q\otimes Q \otimes Q$, we have
\begin{eqnarray*}
[[\tau r,\tau r]]&=&\sum_{i,j}[{b_i}_\mu b_j ]\otimes a_i\otimes a_j|_{\mu=\partial_2}+{b_i}\otimes b_j \otimes [{a_j}_\mu a_i]|_{\mu=\partial_2}\\
&&\quad -b_i\otimes [{a_i}_\mu b_j ]\otimes a_j|_{\mu=\partial_1}
-b_i\otimes [{b_j}_\mu a_i ]\otimes a_j|_{\mu=-\partial_1-\partial_2}\;\;\; (\text{mod}~~\partial^{\otimes^3})\\
&=&\tau_{13}[[r,r]]\;\;(\text{mod}~~\partial^{\otimes^3}).
\end{eqnarray*}
Thus  this conclusion holds.
\end{proof}

\begin{pro}
Let $Q$ be a Leibniz conformal algebra and $r=\sum_i a_i\otimes b_i\in Q\otimes Q$. Define a co-bracket $\Delta_r: Q\rightarrow Q\otimes Q$ by \eqref{dr}. Then $(Q,\Delta_r)$ is a Leibniz conformal coalgebra if and only if for all $a\in Q$ the following equality holds:

\begin{equation}\label{rcoeq}
\begin{aligned}
&(\id\otimes\id\otimes R(a)_{-\partial^{\otimes^3} }-\id\otimes (L+R)(a)_{-\partial^{\otimes^3} }\otimes \id)( \tau_{12}[[r,r]]-\sum_j F(a_j)(r-\tau (r))\otimes b_j)\\
&+\sum_j( (L+R)(a_j)_{\partial_3}\otimes \id\otimes \id)(\tau \circ F(a)(r-\tau(r))\otimes b_j)\\
&+((L+R)(a)_{-\partial^{\otimes^3}}\otimes \id\otimes \id)[[r,r]]=0.
\end{aligned}
\end{equation}
\end{pro}
\begin{proof}
For all $a\in Q$, we have
\begin{align*}
(\Delta_{r}\otimes \id)\Delta_{r}(a)+\tau_{12} (\id\otimes \Delta_{r})\Delta_{r} (a)-(\id\otimes \Delta_{r})\Delta_{r}(a)=\sum^9_{k=1} A_k,
\end{align*}
where
\begin{flalign*}
 A_1=&\sum_{i,j}(-[{a_i}_{ -\partial^{\otimes^2}\otimes \id} a_j ]\otimes b_j\otimes [{b_i}_{\partial^{\otimes^2}\otimes \id}a]- [{a_j}_{ \partial_2} a_i]\otimes b_j\otimes [{b_i}_{\partial^{\otimes^2}\otimes \id}a]\\
&+a_j\otimes [{b_j}_{\partial_1}{a_i}]\otimes[{b_i}_{\partial^{\otimes^2}\otimes \id}a]
+a_j\otimes a_i\otimes[{b_j}_{\partial_1 }[{b_i}_{\partial_2}a] ]-a_i\otimes a_j\otimes[{b_j}_{\partial_2 }[{b_i}_{\partial_1}a] ])\\
=&(\id\otimes \id\otimes R(a)_{-\partial^{\otimes^3} })(a_j\otimes [{b_j}_{\partial_1}{a_i}]\otimes b_i+ a_i\otimes a_j\otimes[{b_i}_{\partial_1} b_j]\\
&- [{a_i}_{ -\partial^{\otimes^2}\otimes \id} a_j ]\otimes b_j\otimes b_i-[{a_j}_{ \partial_2} a_i]\otimes b_j\otimes b_i ) \\
=&(\id\otimes \id\otimes R(a)_{-\partial^{\otimes^3} })( \tau_{12}[[r,r]]-\sum_j F(a_j)(r-\tau (r))\otimes b_j)\\
A_2=&\sum_{i,j}([[a_{-\partial^{\otimes^3}}a_i]_{-\partial^{\otimes^2}\otimes \id}a_j]\otimes b_j\otimes b_i+[[{a_i}_{\partial_3}a]_{-\partial^{\otimes^2}\otimes \id}a_j]\otimes b_j\otimes b_i)=0,\\
A_3=&\sum_{i,j}(-[[{b_i}_{\partial_2}a]_{ -\partial_1-\partial_3} a_j]\otimes a_i\otimes b_j-[{a_j}_{\partial_3}[{b_i}_{\partial_2}a] ]\otimes a_i\otimes b_j),\\
A_4=&\sum_{i,j}([{a_j}_{\partial_2}[a_{-\partial^{\otimes^3}}a_i]]\otimes b_j\otimes b_i-[a_{-\partial^{\otimes^3} }a_i]\otimes [{b_i}_{ -\id\otimes\partial^{\otimes^2}} a_j ]\otimes b_j-[a_{-\partial^{\otimes^3}}a_i]\otimes [{a_j}_{ \partial_3} b_i]\otimes b_j\\
&+[a_{-\partial^{\otimes^3}}a_i]\otimes a_j\otimes[{b_j}_{\partial_2}b_i])
=(L(a)_{-\partial^{\otimes^3}}\otimes \id\otimes \id)[[r,r]]+\sum_{i,j}([[{a_i}_{\partial_2}a]_{-\partial_1-\partial_3}]\otimes b_i\otimes b_j),\\
A_5=&\sum_{i,j}([{b_i}_{ -\partial_1-\partial_3} a_j ]\otimes [a_{-\partial^{\otimes^3} }a_i]\otimes b_j+ [{a_j}_{ \partial_3} b_i]\otimes  [a_{-\partial^{\otimes^3}}a_i]\otimes b_j -a_j\otimes[a_{-\partial^{\otimes^3}}a_i]\otimes[{b_j}_{\partial_1}b_i] )\\
=&-(\id\otimes L(a)_{-\partial^{\otimes^3} }\otimes \id)\tau_{12}[[r,r]]+\sum_{i,j}(b_i\otimes[a_{-\partial^{\otimes^3}}[{a_i}_{\partial_1}a_j]]\otimes b_j),\\
A_6=&\sum_{i,j}(-a_j\otimes [{b_j}_{\partial_1}[a_{-\partial^{\otimes^3}}a_i]   ] \otimes b_i+a_i\otimes  [[{b_i}_{\partial_1}a]_{ -\id\otimes\partial^{\otimes^2}} a_j]\otimes b_j)
=\sum_{i,j}-a_i\otimes  [a_{-\partial^{\otimes^3} }[{b_i}_{\partial_1} a_j]\otimes b_j, \\
A_7=&\sum_{i,j}([{a_j}_{\partial_2}[{a_i}_{\partial_3}a]]\otimes b_j\otimes b_i-[{a_i}_{\id\otimes\partial^{\otimes^2}}a]\otimes [{b_i}_{ -\id\otimes\partial^{\otimes^2}} a_j ] \otimes b_j   -  [{a_i}_{\id\otimes\partial^{\otimes^2} }a] \otimes [{a_j}_{ \partial_3} b_i] \otimes b_j \\
&  + [{a_i}_{\id\otimes\partial^{\otimes^2}}a]\otimes a_j \otimes[{b_j}_{\partial_2}b_i])
=(R(a)_{-\partial^{\otimes^3} }\otimes \id\otimes \id)[[r,r]]+\sum_{i,j}([{a_j}_{\partial_3}[{a_i}_{\partial_2}a]  ]\otimes b_i\otimes b_j),\\
A_8=&\sum_{i,j}([{b_i}_{ -\partial_1-\partial_3} a_j ]\otimes [a_{-\partial^{\otimes^3} }a_i]\otimes b_j+ [{a_j}_{ \partial_3} b_i]\otimes  [a_{-\partial^{\otimes^3}}a_i]\otimes b_j
-a_j\otimes[a_{-\partial^{\otimes^3}}a_i]\otimes[{b_j}_{\partial_1}b_i] )  \\
=&-(\id\otimes R(a)_{-\partial^{\otimes^3} }\otimes \id)\tau_{12}[[r,r]]+\sum_{i,j}b_i\otimes[[{a_i}_{\partial_1} a_j]_{\partial_1+\partial_3}a]\otimes b_j, \\
A_9=&\sum_{i,j}(a_i\otimes[{a_j}_{\partial_3}[{b_i}_{\partial_1}a] ]\otimes  b_j -a_j\otimes [{b_j}_{\partial_1}[{a_i}_{\partial_3}a]  ] \otimes b_i)=-\sum_{i,j}a_i\otimes[[{b_i}_{\partial_1}a_j]_{ \partial_1+\partial_3}a]\otimes b_j.&
\end{flalign*}
Thus we have
{\small
\begin{flalign*}
    &\sum^9_{k=1} A_k -(\id\otimes \id\otimes R(a)_{-\partial^{\otimes^3} })( \tau_{12}[[r,r]]-\sum_j F(a_j)(r-\tau (r))\otimes b_j)\\
&-((L+R)(a)_{-\partial^{\otimes^3}}\otimes \id\otimes \id)[[r,r]]+(\id\otimes (L+R)(a)_{-\partial^{\otimes^3} }\otimes \id)\tau_{12}[[r,r]]\\
    &-\sum_j( (L+R)(a_j)_{\partial_3}\otimes \id\otimes \id)(\tau \circ(F(a))(r-\tau(r))\otimes b_j)\\
    =&\sum_{i,j}(b_i\otimes[a_{-\partial^{\otimes^3}}[{a_i}_{\partial_1}a_j]]\otimes b_j -a_i\otimes  [a_{-\partial^{\otimes^3} }[{b_i}_{\partial_1} a_j]\otimes b_j   +b_i\otimes[[{a_i}_{\partial_1} a_j]_{\partial_1+\partial_3}a]\otimes b_j    \\
    & -a_i\otimes[[{b_i}_{\partial_1}a_j]_{ \partial_1+\partial_3}a]\otimes b_j
    -[{b_i}_{-\partial_1-\partial_3}a_j]\otimes[a_{-\partial^{\otimes^3}}a_i]\otimes b_j- [{b_i}_{ -\partial_1-\partial_3}a_j]\otimes[{a_i}_{\partial_1+\partial_3}a]\otimes b_j\\
    &+[{a_i}_{-\partial_1-\partial_3}a_j]\otimes[a_{-\partial^{\otimes^3}}b_i]\otimes b_j+ [{a_i}_{ -\partial_1-\partial_3}a_j]\otimes[{b_i}_{\partial_1+\partial_3}a]\otimes b_j      -[{a_j}_{\partial_3}b_i]\otimes[a_{-\partial^{\otimes^3}}a_i]\otimes b_j \\
    &- [{a_j}_{\partial_3}b_i]\otimes[{a_i}_{\partial_1+\partial_3}a]\otimes b_j
    +[{a_j}_{\partial_3}a_i]\otimes[a_{-\partial^{\otimes^3}}b_i]\otimes b_j+ [{a_j}_{ \partial_3}a_i]\otimes[{b_i}_{\partial_1+\partial_3}a]\otimes b_j  )\\
    =&(\id\otimes (L+R)(a)_{-\partial^{\otimes^3} }\otimes \id)\sum_j F(a_j)(r-\tau (r))\otimes b_j,&
\end{flalign*}}
which means that \eqref{rcoeq} holds, since
{\small
\begin{flalign*}
&\sum_{i,j}([[{a_i}_{\partial_2}a]_{-\partial_1-\partial_3}]\otimes b_i\otimes b_j+[{a_j}_{\partial_3}[{a_i}_{\partial_2}a]  ]\otimes b_i\otimes b_j    -[[{b_i}_{\partial_2}a]_{ -\partial_1-\partial_3} a_j]\otimes a_i\otimes b_j \\
&-[{a_j}_{\partial_3}[{b_i}_{\partial_2}a] ]\otimes a_i\otimes b_j)\\
=&\sum_j( (L+R)(a_j)_{\partial_3}\otimes \id\otimes \id)(\tau \circ(F(a))(r-\tau(r))\otimes b_j)    +\sum_{i,j}(-[{b_i}_{-\partial_1-\partial_3}a_j]\otimes[a_{-\partial^{\otimes^3}}a_i]\otimes b_j  \\
&- [{b_i}_{ -\partial_1-\partial_3}a_j]\otimes[{a_i}_{\partial_1+\partial_3}a]\otimes b_j
+[{a_i}_{-\partial_1-\partial_3}a_j]\otimes[a_{-\partial^{\otimes^3}}b_i]\otimes b_j+ [{a_i}_{ -\partial_1-\partial_3}a_j]\otimes[{b_i}_{\partial_1+\partial_3}a]\otimes b_j\\
&-[{a_j}_{\partial_3}b_i]\otimes[a_{-\partial^{\otimes^3}}a_i]\otimes b_j- [{a_j}_{\partial_3}b_i]\otimes[{a_i}_{\partial_1+\partial_3}a]\otimes b_j   +[{a_j}_{\partial_3}a_i]\otimes[a_{-\partial^{\otimes^3}}b_i]\otimes b_j \\
&+ [{a_j}_{ \partial_3}a_i]\otimes[{b_i}_{\partial_1+\partial_3}a]\otimes b_j).&
\end{flalign*}}
 Therefore, this conclusion holds.
\end{proof}

\begin{thm}\label{thm-cob}
Let $Q$ be  a Leibniz conformal algebra  and $r=\sum_i a_i\otimes b_i\in Q\otimes Q$.
The triple $(Q, [\cdot_\lambda\cdot], \Delta_r)$ is a Leibniz conformal bialgebra where  $\Delta_r$ is defined by \eqref{dr}
 if and only if \eqref{rcoeq} and the following conditions hold for all $a$, $b\in Q:$
\begin{eqnarray}
&&\tau(R(b)_{-\lambda-\partial\otimes \id}\otimes \id)F(a)(r-\tau(r))=0,\label{dlb1}\\
&&((L+R)(b)_{-\lambda-\partial^{\otimes^2}}\otimes \id)\tau( F(a)(r-\tau(r)))
-\tau(R(b)_{-\lambda-\partial^{\otimes^2}}\otimes \id )( F(a)(r-\tau(r)))=0.\label{dlb2}
\end{eqnarray}

\end{thm}
\begin{proof}
We only need to check that $\Delta_r$ satisfies \eqref{lb1} and \eqref{lb2} if and only if (\ref{dlb1}) and (\ref{dlb2}) hold.
By  \eqref{lb1}, we have
\begin{flalign*}
&(R(a)_{\lambda-\partial\otimes \id}\otimes \id)\Delta_r(b)-\tau ((R(b)_{-\lambda-\partial\otimes \id}\otimes \id) \Delta_r(a))\\
=&\sum_i([[b_{-\lambda-\id\otimes\partial}a_i]_{-\lambda}a ]\otimes b_i  +[[{a_i}_{\id\otimes\partial}b]_{-\lambda}a ]\otimes b_i -[{a_i}_{-\lambda}a]\otimes[{b_i}_\lambda b]\\
&- b_i\otimes [[a_{\lambda-\partial\otimes \id}a_i]_\lambda b]-b_i\otimes[[{a_i}_{\partial\otimes \id}a ]_\lambda b]+[{b_i}_{-\lambda}a] \otimes[{a_i}_{\lambda}b])&\\
=&\sum_i(-[{a_i}_{-\lambda}a]\otimes[{b_i}_\lambda b]+[{b_i}_{-\lambda}a] \otimes[{a_i}_{\lambda}b])=-\tau(R(b)_{-\lambda-\partial\otimes \id}\otimes \id)F(a)(r-\tau(r)),
\end{flalign*}
which means that \eqref{lb1} holds if and only if \eqref{dlb1} holds.
By \eqref{lb2}, we have
{\small
\begin{flalign*}
&\Delta_r([a_\lambda b])-(\id\otimes R(b)_{-\lambda-\partial^{\otimes^2}}-L(b)_{-\lambda-\partial^{\otimes^2}}\otimes \id-R(b)_{-\lambda-\partial^{\otimes^2}}\otimes \id)(\tau+\id)\Delta_r(a)\\
&-(\id\otimes L(a)_\lambda+L(a)_\lambda\otimes \id)\Delta_r(b)\\
=&\sum_i([[a_\lambda b]_{-\partial^{\otimes^2}} a_i]\otimes b_i+[{a_i}_{\id\otimes \partial}[a_\lambda b]]\otimes b_i-a_i\otimes [{b_i}_{\partial\otimes \id}[a_\lambda b] ]          -([b_{-\lambda-\partial^{\otimes^2}}a_i]\otimes[a_\lambda b_i]  \\
&+[{a_i}_{\lambda+\id\otimes \partial}b]\otimes[a_\lambda b_i]-a_i\otimes[a_\lambda[{b_i}_{\partial\otimes \id}b]]   +[a_\lambda[b_{-\lambda-\partial^{\otimes^2}}a_i]]\otimes b_i+[a_\lambda[{a_i}_{\id\otimes\partial}b]]\otimes b_i    \\
&-[a_\lambda a_i]\otimes[{b_i}_{\lambda+\partial\otimes \id}b]     )
-([a_\lambda a_i]\otimes [{b_i}_{\lambda+\partial\otimes \id}b]+[{a_i}_{-\lambda-\partial\otimes \id}a]\otimes[{b_i}_{\lambda+\partial\otimes \id}b]\\
& -a_i\otimes[[{b_i}_{\partial\otimes \id}a]_{\lambda+\partial\otimes \id}b]    +b_i\otimes[[a_\lambda a_i]_{\lambda+\partial\otimes \id}b]+b_i\otimes[[{a_i}_{\partial\otimes \id}a]_{\lambda+\partial\otimes \id}b]
-[b_{-\lambda-\partial^{\otimes^2}}[a_\lambda a_i]]\otimes b_i   \\
&  -[{b_i}_{-\lambda-\partial\otimes \id}a]\otimes[{a_i}_{\lambda+\partial\otimes \id}b]   -[b_{-\lambda-\partial^{\otimes^2}}[{a_i}_{\id\otimes\partial}a]]\otimes b_i+[b_{-\lambda-\partial^{\otimes^2}} a_i  ]\otimes[{b_i}_{-\lambda-\id\otimes\partial}a]   \\
&     -[b_{-\lambda-\partial^{\otimes^2}}b_i ]\otimes[a_\lambda a_i]   -[b_{-\lambda-\partial^{\otimes^2}}b_i]\otimes[{a_i}_{-\lambda-\id\otimes\partial}a]+[b_{-\lambda-\partial^{\otimes^2}}[{b_i}_{\id\otimes\partial}a]]\otimes a_i       \\
&   -[[a_\lambda a_i]_{\lambda+\id\otimes\partial}b]\otimes b_i      -[[{a_i}_{\id\otimes\partial}a]_{\lambda+\id\otimes\partial}b]\otimes b_i    +[{a_i}_{\lambda+\id\otimes\partial}b]\otimes[{b_i}_{-\lambda-\id\otimes\partial} a]\\
& -[{b_i}_{\lambda+\id\otimes\partial}b ]\otimes[a_\lambda a_i] -[{b_i}_{\lambda+\id\otimes\partial}b]\otimes[{a_i}_{-\lambda-\id\otimes\partial}a]   +[[{b_i}_{\id\otimes\partial}a]_{\lambda+\id\otimes\partial}b]\otimes a_i )   )=\sum_{j=1}^4B_j,&
\end{flalign*}}
where
{\small
\begin{flalign*}
B_1=& \sum_i([{a_i}_{\id\otimes \partial}[a_\lambda b]]\otimes b_i-[{a_i}_{\lambda+\id\otimes \partial}b]\otimes[a_\lambda b_i]-[a_\lambda[{a_i}_{\id\otimes\partial}b]]\otimes b_i   +[{b_i}_{\lambda+\id\otimes\partial}b ]\otimes[a_\lambda a_i]  \\
&-[[{b_i}_{\id\otimes\partial}a]_{\lambda+\id\otimes\partial}b]\otimes a_i)
=(R(b)_{-\lambda-\partial^{\otimes^2}}\otimes \id)(\tau-\id)(F(a)(r-\tau(r))),\\
B_2=&-\sum_i( [b_{-\lambda-\partial^{\otimes^2}}a_i]\otimes[a_\lambda b_i]-[b_{-\lambda-\partial^{\otimes^2}}[{a_i}_{\id\otimes\partial}a]]\otimes b_i+[b_{-\lambda-\partial^{\otimes^2}} a_i  ]\otimes[{b_i}_{-\lambda-\id\otimes\partial}a]\\
&-[b_{-\lambda-\partial^{\otimes^2}}b_i ]\otimes[a_\lambda a_i]-[b_{-\lambda-\partial^{\otimes^2}}b_i]\otimes[{a_i}_{-\lambda-\id\otimes\partial}a]+[b_{-\lambda-\partial^{\otimes^2}}[{b_i}_{\id\otimes\partial}a]]\otimes a_i)&\\
=&(L(b)_{-\lambda-\partial^{\otimes^2} }\otimes \id)\tau(F(a)(r-\tau(r))),\\
B_3=&-\sum_i([{a_i}_{-\lambda-\partial\otimes \id}a]\otimes[{b_i}_{\lambda+\partial\otimes \id}b]-[{b_i}_{-\lambda-\partial\otimes \id}a]\otimes[{a_i}_{\lambda+\partial\otimes \id}b])&\\
=&-\tau(R(b)_{-\lambda-\partial^{\otimes^2}}\otimes \id)(F(a)(r-\tau(r))),\\
B_4=&-\sum_i([{a_i}_{\lambda+\id\otimes\partial}b]\otimes[{b_i}_{-\lambda-\id\otimes\partial} a]-[{b_i}_{\lambda+\id\otimes\partial}b]\otimes[{a_i}_{-\lambda-\id\otimes\partial}a])
=(R(b)_{-\lambda-\partial^{\otimes^2}}\otimes \id)(F(a)(r-\tau(r))).
\end{flalign*}}
Thus, \eqref{lb2} holds if and only if \eqref{dlb2} holds, finishing the proof.
\end{proof}
By Theorem \ref{thm-cob}, we get the following conclusion.
\begin{cor}\label{colcba}
Let $(Q, [\cdot_\lambda \cdot])$ be a Leibniz conformal algebra
and $r\in Q\otimes Q$ be a solution of the CLCYBE in $Q$. Then
$(Q,[\cdot_\lambda\cdot],\Delta_r) $ is a Leibniz conformal
bialgebra if $r-\tau(r)$ is invariant, and we call it {\bf a
quasi-triangular Leibniz conformal bialgebra}. In particular, if
$r$ is symmetric,  we call $(Q,[\cdot_\lambda\cdot],\Delta_r) $
{\bf a  triangular Leibniz conformal bialgebra}.

\end{cor}


\subsection{Operator forms of the classical Leibniz conformal Yang-Baxter equation}


First we recall the notion of { $\mathcal{O}$-operators} on Leibniz conformal algebras.
\begin{defi}\cite{FC}
Let $Q$ be a Leibniz conformal algebra and $(M,l,r)$ be a representation of $Q$.
 A ${\bf k}[\partial]$-module homomorphism $T: M\rightarrow Q$ is called an {\bf $\mathcal{O}$-operator} on $Q$ associated to   $(M, l, r)$ if for all $x,y\in M$, it satisfies
\begin{equation}\label{oeq}
[T(x)_\lambda T(y)]=T(l(T(x))_\lambda y+ r(T(y))_{-\lambda-\partial} x ).
\end{equation}

\end{defi}

Let $Q$ be a finite Leibniz conformal algebra which is free as a ${\bf k}[\partial]$-module. Define a linear map $\sigma: Q\otimes Q\rightarrow \text{Chom}(Q^{*c},Q) $ \delete{and $\eta: \text{Chom}(Q^{*c},Q) \rightarrow Q\otimes Q$} by
\begin{align}
&\langle f,\sigma(a\otimes b)_\lambda g\rangle_\mu=\langle g\otimes f,a\otimes b\rangle_{-\lambda-\mu,\mu}\;\;\;\;\text{for all $f$, $g\in Q^{*c}$ and $a$, $b\in Q$}. \label{dl1}
\end{align}
 Thus $\sigma (a\otimes b)_\lambda g=\langle g,a\rangle_{-\lambda-\partial^Q}b$, where $\partial^Q$ represents the action of $\partial$ on $Q$. Moreover, $\sigma$ is a ${\bf k}[\partial]$-module isomorphism \cite{HB}.

\begin{lem}\label{lecar}
Let $Q$ be a finite Leibniz conformal algebra which is free as a ${\bf k}[\partial]$-module and $r\in Q\otimes Q$. Define a co-bracket $\Delta_r: Q\rightarrow Q\otimes Q$ by \eqref{dr}. Then $Q^{*c}$ is endowed with a Leibniz conformal algebra structure $(Q^{*c}, {[\cdot_\lambda \cdot]}^*)$, where the $\lambda$-bracket is given by
\begin{equation}
[f_\lambda  g]^*=L^*(\sigma(r)_0f)_\lambda g -(L^*+R^*)(\sigma(\tau r)_0g)_{-\lambda-\partial}f\;\;\;\;\text{for all $f$, $g\in Q^{\ast c}$},
\end{equation}
where $\sigma(r)_0f=\sigma(r)_\lambda f|_{\lambda=0}$.
\end{lem}
\begin{proof}
Using the assumptions from the proof of Proposition \ref{pmb}, and letting $r=\sum_{i,j}r_{ij}(\partial\otimes \id,\id\otimes\partial) a^i\otimes a^j$ where $r_{ij}(\lambda,\mu)\in {\bf k}[\lambda,\mu]$, we have
\begin{align*}
\Delta_r(a^k)=&(L(a^k)_{\partial^{\otimes^2}}\otimes \id+R(a^k)_{\partial^{\otimes^2}}\otimes \id-\id\otimes R(a^k)_{\partial^{\otimes^2}})r\\
=&\sum_{i,j,s}(r_{ij}(-\id\otimes\partial,\id\otimes\partial)P_s^{ki}(-\partial^{\otimes^2},\partial\otimes \id)a^s\otimes a^j+r_{ij}(-\id\otimes\partial,\id\otimes\partial)P_s^{ik}(\id\otimes \partial,\partial\otimes \id)\\
&a^s\otimes a^j-r_{ij}(\partial\otimes \id,-\partial\otimes
\id)P_s^{jk}(\partial\otimes \id,\id\otimes\partial)a^i\otimes
a^s),
\end{align*}
and $ \sigma(r)_\lambda(a_i)=\sum_j
r_{ij}(-\lambda-\partial,\partial)a^j$. Then the facts that
\begin{align*}
&\langle {a_m}_\lambda a_n,a^k\rangle_\mu=\langle a_m\otimes a_n,\Delta_r (a^k)\rangle_{\lambda,\mu-\lambda}\\
=&\sum_ir_{in}(\lambda-\mu,\mu-\lambda)(P^{ki}_m(-\mu,\lambda)+P^{ik}_m(\mu-\lambda,\lambda))-\sum_jr_{mj}(\lambda,-\lambda)P^{jk}_n(\lambda,\mu-\lambda)
\end{align*}
and
\begin{align*}
&\langle L^*(\sigma(r)_0a_m)_\lambda a_n -(L^*+R^*)(\sigma(\tau r)_0a_n)_{-\lambda-\partial}a_m,a^k \rangle_\mu\\
=&\langle \sum_j r_{mj}L^*(a^j)_\lambda a_n,a^k \rangle_\mu+\sum_i r_{in}(\lambda-\mu,\mu-\lambda)\langle (L^*+R^*)(a^i)_{-\lambda+\mu}a_m,a^k\rangle_\mu\\
=&-\sum_jr_{mj}(\lambda,-\lambda)P^{jk}_n(\lambda,\mu-\lambda)+\sum_ir_{in}(\lambda-\mu,\mu-\lambda)(P^{ki}_m(-\mu,\lambda)+P^{ik}_m(\mu-\lambda,\lambda))
\end{align*}
imply that the conclusion holds.
\end{proof}

\begin{thm}\label{thdr}
Let $Q$ be a finite Leibniz conformal algebra which is free as a ${\bf k}[\partial]$-module and $r\in Q\otimes Q$. Then $r$ is a solution of the CLCYBE in $Q$ if and only if $\sigma(\tau r)\in \text{Chom}(Q^{*c},Q)$ satisfies 
\begin{equation}
[{\sigma(\tau r)_0x}_\lambda \sigma(\tau r)_0y]=\sigma(\tau r)_0(L^*(\sigma( r)_0x)_\lambda y-(L^*+R^*)(\sigma(\tau r)_0y)_{-\lambda-\partial}x)\;\;\;\;\text{for all $x, y\in Q^{*c}$}.
\end{equation}
Moreover, if $r-\tau r $ is invariant, then $r$ is a solution of the CLCYBE in $Q$ if and only if $\sigma( r)$ satisfies
\begin{equation}
[{\sigma(r)_0x}_\lambda \sigma(r)_0y]=\sigma(r)_0(L^*(\sigma(r)_0x)_\lambda y-(L^*+R^*)(\sigma(\tau r)_0y)_{-\lambda-\partial}x)\;\;\;\text{for all $x,y\in Q^{*c}$}.
\end{equation}

\end{thm}
\begin{proof}
Set $r=\sum_i a_i\otimes b_i$. By \eqref{dl1},  we have
$$
 \sigma(r)_\lambda x=\sum_i\langle x,a_i\rangle_{-\lambda-\partial^Q} b_i,\qquad
\sigma(\tau r)_\lambda x=\sum_i\langle x,b_i\rangle_{-\lambda-\partial^Q} a_i\;\;\;\text{for all $x\in Q^{*c}$}.
$$
It is obvious that the following cases are equivalent:
\begin{enumerate}
\item $[[r,r]]$ \text{mod} $(\partial^{\otimes^3})=0$;
\item $\langle x\otimes y\otimes z,[[r,r]] ~~\text{mod}~~(\partial^{\partial^3})\rangle_{\lambda,\mu,\xi}=0 $ for all $x,y,z\in Q^{*c}$;
\item $\langle x\otimes y\otimes z,[[r,r]] \rangle_{\lambda,\mu,\xi}=0 $ \text{mod} $(\lambda+\mu+\xi)$ for all $x,y,z\in Q^{*c}$.
\end{enumerate}
For all $x,y,z\in Q^{*c}$, we have
\begin{align*}
&\sum_{i,j}\langle x\otimes y\otimes z,[{a_i}_{\partial_2}a_j]\otimes b_i\otimes b_j\rangle_{\lambda,\mu,\xi}
=\sum_{i,j}\langle x,[{a_i}_{\mu}a_j] \rangle_\lambda\langle y,b_i\rangle_\mu\langle z,b_j\rangle_\xi\\
=&\langle x,[(\sum_i\langle y,b_i\rangle_\mu {a_i})_\mu\sum_j\langle z,b_j\rangle_\xi a_j]\rangle_\lambda
=\langle x,[\sigma(\tau r)_0(y)_\mu\sigma(\tau r)_{-\xi-\mu-\lambda}z]\rangle_\lambda,
\end{align*}
and
\begin{align*}
&\sum_{i,j}\langle x\otimes y\otimes z,a_i\otimes a_j\otimes [{b_j}_{\partial_2}b_i]\rangle_{\lambda,\mu,\xi}\\
=&\sum_i\langle x,a_i\rangle_\lambda\langle z,[\sigma(r)_0(y)_\mu b_i]\rangle_\xi=-\sum_i\langle x,a_i\rangle_\lambda\langle L^*(\sigma(r)_0(y))_\mu z,b_i\rangle_{\mu+\xi}\\
=&-\sum_i\langle x,\langle L^*(\sigma(r)_0(y))_\mu z,b_i\rangle_{\mu+\xi}a_i\rangle_\lambda=-\langle x,\sigma(\tau r)_{-\lambda-\mu-\xi}(L^*(\sigma(r)_0(y))_\mu z)\rangle_\lambda.
\end{align*}
Similarly, we can get
\begin{align*}
&\sum_{i,j}\langle x\otimes y\otimes z,a_i\otimes [{b_i}_{\partial_1}a_j]\otimes b_j\rangle_{\lambda,\mu,\xi}= -\langle x,\sigma(\tau r)_0(R^*(\sigma(\tau r)_{-\lambda-\mu-\xi}(z))_{-\mu-\partial}y)\rangle_\lambda,\\
&\sum_{i,j}\langle x\otimes y\otimes z,a_i\otimes [{a_j}_{-\partial^{\otimes^2}}b_i]\otimes b_j\rangle_{\lambda,\mu,\xi}= -\langle x,\sigma(\tau r)_0(L^*(\sigma(\tau r)_{-\lambda-\mu-\xi}(z))_{-\mu-\partial}y) \rangle_\lambda.
\end{align*}
Thus, $[[r,r]]$ mod $(\partial^{\otimes^3})=0$ if and only if
\begin{equation}\label{dr0}
[\sigma(\tau r)_0(y)_\mu\sigma(\tau r)_0(z)]-\sigma(\tau r)_{0}(L^*(\sigma(r)_0(y))_\mu z)+\sigma(\tau r)_0((L^*+R^*)(\sigma(\tau r)_{0}(z))_{-\mu-\partial}y)=0
\end{equation}
for all $y,z\in Q^{*c}$. The first conclusion follows by replacing $\mu$ by $\lambda$ in \eqref{dr0}.
On the other hand, we have
{\small
\begin{align*}
&\langle z,[{\sigma(r)_0x}_\lambda \sigma(r)_0y]-\sigma(r)_0(L^*(\sigma(r)_0x)_\lambda y-(L^*+R^*)(\sigma(\tau r)_0y)_{-\lambda-\partial}x) \rangle_\mu\\
=&\sum_{i,j}\langle x\otimes y\otimes z, a_i\otimes a_j\otimes [{b_i}_\lambda b_j]+ a_j\otimes [{b_j}_\lambda a_i]\otimes  b_i
-[{a_j}_{-\mu-\lambda}a_i ]\otimes b_j\otimes b_i -[{a_i}_\mu a_j]\otimes b_j\otimes b_i\rangle_{\lambda,-\lambda-\mu,\mu}\\
=&\sum_{j}\langle x\otimes y\otimes z,\tau_{12}[[r,r]]-F(a_j)(r-\tau r)\rangle_{\lambda,-\lambda-\mu,\mu},
\end{align*}}
for all $x,y,z\in Q^{*c}$, which means that the second conclusion holds.
\end{proof}

By Proposition \ref{pprtr}, Lemma \ref{lecar} and Theorem
\ref{thdr}, we obtain the following conclusion.
\begin{cor}
Let $Q$ be a finite Leibniz conformal algebra which is free as a ${\bf k}[\partial]$-module and $r$ be a solution of the CLCYBE in $Q$. Then the following conclusions hold:
\begin{enumerate}
\item $\delta(\tau r)_0$ is a Leibniz conformal algebra homomorphism from $(Q^{*c},[\cdot_\lambda\cdot]^*)$ to $(Q,[\cdot_\lambda\cdot])$.
\item If  $(Q,[\cdot_\lambda\cdot],\Delta_r)$ is a quasi-triangular Leibniz conformal bialgebra,
then
$\delta(r)_0$ and $\delta(\tau r)_0$ are both Leibniz conformal algebra homomorphisms from $(Q^{*c},[\cdot_\lambda\cdot]^*)$ to $(Q,[\cdot_\lambda\cdot])$. Particularly, if $r$ is symmetric, then $\delta(r)_0$ is an $\mathcal{O}$-operator on $Q$ associated to   $(Q^{*c}, L^\ast, -L^\ast-R^\ast)$.
\end{enumerate}
\end{cor}

Let $Q$ be a Leibniz conformal algebra and $(M,l_Q,r_Q)$ be a representation of $Q$, where $Q$ and $M$ are finite and free as ${\bf k}[\partial]$-modules. Then $(M^{*c},l^*_Q,$ $-l^*_Q-r^*_Q)$ is a representation of $Q$  by Proposition \ref{dual}. By \cite[Proposition 6.1]{BKL}, there is
an isomorphism of ${\bf k}[\partial]$-modules $\varphi: M^{*c}\otimes Q\rightarrow\text{Chom}(M,Q) $ defined by
$$
\varphi(f\otimes a)_\lambda v=\langle v,f\rangle_{-\lambda-\partial^Q}a\;\;\;\text{for all $a\in Q$, $v\in M$, $f\in M^{*c}$}.
$$
Therefore, the inverse ${\bf k}[\partial]$-module isomorphism of $\varphi$ is $\eta: \text{Chom}(M,Q)\rightarrow M^{*c}\otimes Q$ given by
$$\langle x\otimes g, \eta(T)\rangle_{\lambda,\mu}=\langle g,T_{-\lambda-\mu} (x)\rangle_\mu\;\; \;\text{for all $x\in M$, $g\in Q^{\ast c}$ and $T\in \text{Chom}(M,Q)$.}$$
Consequently, for any $T\in \text{Chom}(M,Q)$, we have $\eta(T)\in  M^{*c}\otimes Q\subset (Q\ltimes_{l^*_Q,-l^{*}_Q-r^{*}_Q}M^{*c})\otimes (Q\ltimes_{l^*_Q,-l^{*}_Q-r^{*}_Q}M^{*c})$.

\begin{thm}\label{thamc}
Let $Q$ be a Leibniz conformal algebra, $(M,l_Q,r_Q)$ be a
representation of $Q$ and both $Q$ and $M$ are finite and free as ${\bf k}[\partial]$-modules. Let $T\in
\text{Chom}(M,Q)$ and $\eta (T)\in M^{*c}\otimes Q\subset
(Q\ltimes_{l^*_Q,-l^{*}_Q-r^{*}_Q}M^{*c})\otimes
(Q\ltimes_{l^*_Q,-l^{*}_Q-r^{*}_Q}M^{*c})$. Then $r=\eta
(T)+\tau\eta(T)$ is a symmetric solution of the CLCYBE in the
Leibniz conformal algebra
$Q\ltimes_{l^*_Q,-l^{*}_Q-r^{*}_Q}M^{*c}$ if and only if
$T_0=T_\lambda|_{\lambda=0}$ is an $\mathcal{O}$-operator on $Q$
associated to   $(M,l_Q,r_Q)$.
\end{thm}
\begin{proof}
Let $\{ a^1,\cdots,a^n\}$ be a ${\bf k}[\partial]$-basis of $Q$,
 $\{v^1,\cdots,v^m\}$ be a ${\bf k}[\partial]$-basis of $M$ and $\{v_1,\cdots,v_m\}$ be a ${\bf k}[\partial]$-basis of $M^{*c}$. Assume that for all $i=1,\cdots,m$,
$$
T_\lambda(v_i)=\sum_{j=1}^n g_{ij}(\lambda,\partial)a^j,\;\; g_{ij}(\lambda,\partial)\in{\bf k}[\lambda,\partial].$$ Then we have
$
\eta(T)=\sum_{i}^m\sum_j^n g_{ij}(-\partial^{\otimes^2},\id\otimes \partial)v^*_i\otimes a^j.
$
Thus we obtain
$$
r=\eta (T)+\tau\eta(T)=\sum_{i}^m\sum_j^n( g_{ij}(-\partial^{\otimes^2},\id\otimes \partial)v^*_i\otimes a^j+ g_{ij}(-\partial^{\otimes^2}, \partial\otimes \id)a^j \otimes v^*_i ).
$$
On the other hand, by \eqref{lrm}, we have
\begin{equation*}
l^*_Q (a^i)_\lambda v_j^*=-\sum_t\langle v^*_j,l_Q(a^i)_\lambda v_t\rangle_{-\lambda-\partial}v^*_t,\qquad r^*_Q (a^i)_\lambda v_j^*=-\sum_t\langle v^*_j,r_Q(a^i)_\lambda v_t\rangle_{-\lambda-\partial}v^*_t.
\end{equation*}
Then we have
{\small
\begin{flalign*}
 &[[r,r]]\quad\text{mod $(\partial^{\otimes^3} )$}\\
=&\sum^m_{i,t}\sum^n_{s,j}(g_{ij}(0, \partial_2) g_{ts}(-\partial^{\otimes^3},\partial^{\otimes^2}\otimes \id) [{v^*_i}_{\partial_2}a^s ]\otimes a^j\otimes v^*_t\\
&+g_{ts}(0,-\partial_2)g_{ij}(-\partial^{\otimes^3},\partial_3)[{a^s}_{\partial_2}v^*_i]\otimes v^*_t\otimes a^j+g_{ij}(0,- \partial_2) g_{ts}(-\partial^{\otimes^3},\partial^{\otimes^2}\otimes \id)[{a^j}_{\partial_2}a^s ]\otimes v^*_i\otimes v^*_t\\
&+g_{ts}(0,-\partial_2)g_{ij}(-\partial^{\otimes^3},\id\otimes\partial^{\otimes^2}) v^*_i \otimes v^*_t \otimes[{a^s}_{\partial_2}a^j]+g_{ts}(0,\partial_2)g_{ij}(-\partial^{\otimes^3},\id\otimes\partial^{\otimes^2}) v^*_i \otimes a^s \otimes[{v^*_t}_{\partial_2}a^j]\\
&+g_{ij}(0,-\partial_2)g_{ts}(-\partial^{\otimes^3},\partial_1) a^s\otimes v^*_i \otimes  [{a^j}_{\partial_2}v^*_t]-g_{ij}(0,-\partial_1)g_{ts}(-\partial^{\otimes^3},\partial^{\otimes^2}\otimes \id) v^*_i \otimes  [{a^j}_{\partial_1}a^s] \otimes v^*_t&\\
&-g_{ij}(0,-\partial_1)g_{ts}(-\partial^{\otimes^3},\partial_3) v^*_i \otimes  [{a^j}_{\partial_1}v^*_t] \otimes a^s-g_{ij}(0,\partial_1)g_{ts}(-\partial^{\otimes^3},\partial^{\otimes^2}\otimes \id)a^j  \otimes  [{v^*_i}_{\partial_1}a^s] \otimes v^*_t\\
&-g_{ij}(0,-\partial_1)g_{ts}(-\partial^{\otimes^3},\partial_3)v^*_i \otimes  [{ v^*_t}_{-\partial^{\otimes^2}}a^j] \otimes a^s -g_{ij}(0,-\partial_1)g_{ts}(-\partial^{\otimes^3},\partial^{\otimes^2}\otimes \id) v^*_i \otimes  [{a^s }_{-\partial^{\otimes^2}}a^j] \otimes  v^*_t
\\
&-g_{ij}(0,\partial_1)g_{ts}(-\partial^{\otimes^3},\partial^{\otimes^2}\otimes \id) a^j  \otimes  [{a^s }_{-\partial^{\otimes^2}}v^*_i] \otimes  v^*_t) \quad\text{mod $(\partial^{\otimes^3} )$}\\
=&\sum_{i,t}^m([{v^*_i}_{\partial_2}T_0(v_t) ]\otimes T_0(v_i)\otimes v^*_t
+[{T_0(v_t)}_{\partial_2}v^*_i]\otimes v^*_t\otimes T_0(v_i)   +[{T_0(v_i)}_{\partial_2} T_0(v_t) ]\otimes v^*_i\otimes v^*_t   \\
&+v^*_i \otimes v^*_t \otimes[{T_0(v_t)}_{\partial_2}T_0(v_i)]
+v^*_i \otimes T_0(v_t) \otimes[{v^*_t}_{\partial_2} T_0(v_i)]
+T_0(v_t)\otimes v^*_i \otimes  [{T_0(v_i)}_{\partial_2}v^*_t]\\
&-v^*_i \otimes  [{T_0(v_i)}_{\partial_1} T_0(v_t)] \otimes v^*_t
-v^*_i \otimes  [{T_0(v_i)}_{\partial_1}v^*_t] \otimes T_0(v_t)\\
&-T_0(v_i) \otimes  [{v^*_i}_{\partial_1} T_0(v_t)] \otimes v^*_t
-v^*_i \otimes  [{ v^*_t}_{-\partial^{\otimes^2}}T_0(v_i)] \otimes T_0(v_t)\\
&-v^*_i \otimes  [{T_0(v_t) }_{-\partial^{\otimes^2}}T_0(v_i)] \otimes  v^*_t
 -T_0(v_i)  \otimes  [{T_0(v_t) }_{-\partial^{\otimes^2}}v^*_i] \otimes  v^*_t)\quad\text{mod $(\partial^{\otimes^3} )$}.&
\end{flalign*}}
Since $T_0$ is a ${\bf k}[\partial]$-module homomorphism,  we have
{\small\begin{flalign*}
&\sum_{i,t}^m[{v^*_i}_{\partial_2}T_0(v_t) ]\otimes T_0(v_i)\otimes v^*_t\quad\text{mod $(\partial^{\otimes^3} )$}
=-\sum_{i,t}^m(l_Q^*+r_Q^*)(T_0(v_t))_{-\partial^{\otimes^2}\otimes \id}v^*_i\otimes T_0(v_i)\otimes v_t^* \quad\text{mod $(\partial^{\otimes^3} )$}\\
=&-\sum_{i,t}^m(l_Q^*+r_Q^*)(T_0(v_t))_{\partial_3}v^*_i\otimes T_0(v_i)\otimes v_t^* \quad\text{mod $(\partial^{\otimes^3} )$}\\
=&\sum_{i,t,j}^m\langle v^*_i, (l_Q+r_Q)(T_0(v_t))_{\partial_3} v_j\rangle_{\partial_2}v^*_j\otimes T_0(v_i)\otimes v_t^* \quad\text{mod $(\partial^{\otimes^3} )$}\\
=&\sum_{i,t,j}^mv^*_j\otimes T_0(\langle v^*_i, (l_Q+r_Q)(T_0(v_t))_{\partial_3} v_j\rangle_{\partial_2}v_i)\otimes v_t^* \quad\text{mod $(\partial^{\otimes^3} )$}\\
=&\sum_{t,j}^mv^*_j\otimes T_0((l_Q+r_Q)(T_0(v_t))_{\partial_3}v_j)\otimes v_t^* \quad\text{mod $(\partial^{\otimes^3} )$}.&
\end{flalign*}}
Similarly, we get
{\small
\begin{flalign*}
&[[r,r]]\quad\text{mod $(\partial^{\otimes^3} )$}\\
=&\sum_{i,t}^m ( ([T_0(v_i)_{\partial_2}T_0(v_t)]-T_0(l_Q(T_0(v_i))_{\partial_2}v_t) - T_0( r_Q(T_0(v_t))_{\partial_3}v_i)  )\otimes v^*_i\otimes v^*_t\\
&+v^*_i\otimes( T_0(l_Q( T_0(v_i))_{\partial_1}v_t+r_Q(T_0(v_t))_{\partial_3}v_i)-[T_0(v_i)_{\partial_1} T_0(v_t) ]  )\otimes v^*_t&
\end{flalign*}
\begin{flalign*}
\quad&+v^*_i\otimes( T_0(l_Q( T_0(v_t))_{\partial_3}v_i+r_Q(T_0(v_i))_{\partial_1}v_t)-[T_0(v_t)_{\partial_3} T_0(v_i) ]  )\otimes v^*_t&\\
&+v^*_i\otimes v^*_t\otimes ([T_0(v_t)_{\partial_2}T_0(v_i)]- T_0(l_Q( T_0(v_t))_{\partial_2}v_i+r_Q(T_0(v_i))_{\partial_1}v_t)   )   )\quad\text{mod $(\partial^{\otimes^3} )$}.
\end{flalign*}}
Thus $r$ is a solution of the CLCYBE in the Leibniz conformal algebra $Q\ltimes_{l^*_Q,-l^{*}_Q-r^{*}_Q}M^{*c}$ if and only if for all $i,t\in\{1,\cdots,m\}$,
$$
[T_0(v_i)_{\lambda}T_0(v_t)]=T_0(l_Q(T_0(v_i))_{\lambda}v_t) +T_0( r_Q(T_0(v_t))_{-\lambda-\partial}v_i).
$$
Hence this conclusion holds.
\end{proof}

In what follows, we recall the notion of Leibniz-dendriform conformal algebras.
\begin{defi}\cite{GW}
A {\bf Leibniz-dendriform conformal algebra (LDCA)} is a triple $(Q,\lhd_\lambda,\rhd_\lambda)$ consisting of a ${\bf k}[\partial]$-module $Q$ and two $\lambda$-products $\lhd_\lambda$, $\rhd_\lambda: Q\times Q\rightarrow Q[\lambda]$ which satisfy the following conditions:
\begin{align}
&(\partial a)\lhd_\lambda b=-\lambda a\lhd_\lambda b,\quad a\lhd_\lambda (\partial b)=(\lambda+\partial)a\lhd_\lambda b,\label{ldca00}\\
&(\partial a)\rhd_\lambda b=-\lambda a\rhd_\lambda b,\quad a\rhd_\lambda (\partial b)=(\lambda+\partial)a\rhd_\lambda b,\label{ldca01}\\
&a\rhd_\lambda (b\rhd_\mu c+b\lhd_\mu c)=(a\rhd_\lambda b)\rhd_{\lambda+\mu}c+b\lhd_\mu(a\rhd_\lambda c), \label{ldca1}\\
&a\lhd_\lambda(b\rhd_\mu c)=(a\lhd_\lambda b)\rhd_{\lambda+\mu}c+b\rhd_\mu(a\rhd_\lambda c+a\lhd_\lambda c),\label{ldca2}\\
&a\lhd_\lambda (b\lhd_\mu c)=(a\rhd_\lambda b+a\lhd_\lambda b)\lhd_{\lambda+\mu}c+b\lhd_\mu(a\lhd_\lambda c)\label{ldca3}\;\;\text{for all $ a, b, c\in Q$.}
\end{align}
\end{defi}
\delete{
Let $(Q,\lhd_\lambda,\rhd_\lambda)$ and $(Q',\lhd'_\lambda,\rhd'_\lambda)$ be two LDCAs. A ${\bf k}[\partial]$-module homomorphism $f: Q\rightarrow Q'$ is called a homomorphism of LDCAs if for all $a,b\in Q$,
\begin{equation*}
f(a\lhd_\lambda b)=f(a)\lhd'_\lambda f(b),\qquad f(a\rhd_\lambda b)=f(a)\rhd'_\lambda f(b).
\end{equation*}
In this way, we define a category $\mathcal{LD}^c$ whose objects
are all LDCAs and morphisms are all homomorphisms of LDCAs. }

\begin{rmk}\cite{GW}
Let   $(Q,\lhd_\lambda,\rhd_\lambda)$ be a LDCA. Define $[\cdot_\lambda \cdot]_{\lhd,\rhd}$ by
\begin{equation}\label{ldca-lca}
[a_\lambda b]_{\lhd,\rhd}=a\rhd_\lambda b+a\lhd_\lambda b\;\;\;\text{for all $a, b\in Q$}.
\end{equation}
Then $(Q,[\cdot_\lambda \cdot]_{\lhd,\rhd})$ is a Leibniz conformal algebra.
\delete{\begin{enumerate}
\item
If the operation $\rhd_\lambda=0$, then $(Q,\lhd_\lambda) $  is exactly a Leibniz conformal algebra.
\item
If the operation $\rhd_\lambda$ satisfies $a\rhd_{-\lambda-\partial}b=-b\lhd_\lambda a$ for all $a,b\in Q$, then $(Q,\lhd_\lambda)  $ is a left-symmetric conformal algebra.
\item
Define $[\cdot_\lambda \cdot]_{\lhd,\rhd}$ by
\begin{equation}\label{ldca-lca}
[a_\lambda b]_{\lhd,\rhd}=a\rhd_\lambda b+a\lhd_\lambda b\;\;\;\text{for all $a, b\in Q$}.
\end{equation}
Then $(Q,[\cdot_\lambda \cdot]_{\lhd,\rhd})$ is a Leibniz conformal algebra.
\end{enumerate}}
\end{rmk}

Let $(Q,\lhd_\lambda,\rhd_\lambda)$ be a LDCA. Define two ${\bf k}[\partial]$-module homomorphisms $L_\lhd$ and $R_\rhd$ from $Q$ to $\text{Cend}(Q)$ by
$L_\lhd(a)_\lambda b=a\lhd_\lambda b$, $ R_\rhd(a)_{-\lambda-\partial}b=b\rhd_\lambda a$ for all $a, b\in Q$.

\begin{pro}\label{Idrb}
Let $(Q,\lhd_\lambda,\rhd_\lambda)$ be a LDCA. Then
$(Q,L_\lhd,R_\rhd)$ is a representation  of the Leibniz conformal
algebra $(Q,[\cdot_\lambda\cdot]_{\lhd,\rhd})$, where
$[\cdot_\lambda\cdot]_{\lhd,\rhd}$ is defined by
\eqref{ldca-lca}. Moreover, the identity map $\id: Q\rightarrow Q$
is an $\mathcal{O}$-operator  on $(Q,[\cdot_\lambda
\cdot]_{\lhd,\rhd})$ associated to   $(Q,L_\lhd,R_\rhd) $.
\end{pro}
\begin{proof}
It is straightforward.
\delete{It just need to check that $(Q,L_\lhd,R_\rhd)$ satisfies \eqref{m1}-\eqref{m3}.
For all $a,b,c\in Q$, we have
\begin{align*}
&L_\lhd(a)_\lambda(L_\lhd(b)_\mu c)=a\lhd_\lambda(b\lhd_\mu c)=(a\rhd_\lambda+a\lhd_\lambda b)\lhd_{\lambda+\mu}c+b\lhd_\mu(a\lhd_\lambda c)\\
=&L_\lhd([a_\lambda b])_{\lambda+\mu}c+L_\lhd(b)_\mu(L_\lhd(a)_\lambda c),
\end{align*}
which means that \eqref{m1} holds. Similarly, we can verify that \eqref{m2} and \eqref{m3} hold.
Moreover, we have
\begin{equation*}
\id(L_\lhd(\id(a)))_\lambda b+R_\rhd(\id(b))_{-\lambda-\partial}a=a\lhd_\lambda b+a\rhd_\lambda y=[\id(a)_\lambda \id(b)]_{\lhd,\rhd},
\end{equation*}
which means that $\id$ is an $\mathcal{O}$-operator on  $(Q,[\cdot_\lambda \cdot]_{\lhd,\rhd})$ associated to
$(Q,L_\lhd,R_\rhd) $.}
\end{proof}
\delete{
\begin{pro}\label{mlr}
Let  $(Q,\lhd_\lambda,\rhd_\lambda)$  and $(Q',\lhd'_\lambda,\rhd'_\lambda)$ be two LDCAs and $f$ be an LDCA homomorphism from $Q$ to $Q'$. Then $(f,f)$ is a homomorphism of $\mathcal{O}$-operators from $\id_Q$ to $\id_{Q'}$.
\end{pro}
\begin{proof}
This conclusion naturally holds.
\end{proof}}

\begin{pro} \cite[Proposition 2.8]{GW} \label{mld}
Let $T:M\rightarrow Q$ be an $\mathcal{O}$-operator on $(Q,[\cdot_\lambda \cdot])$ associated to   $(M,l_Q,r_Q)$. Then  there is an LDCA structure on $M$ given by
\begin{align*}
x\lhd_\lambda y=l_Q(T(x))_\lambda y,\qquad  x\rhd_\lambda y=r_Q(T(y))_{-\lambda-\partial}x\;\;\;\text{for all $x, y\in M$}.
\end{align*}
\end{pro}
\delete{\begin{proof}
\eqref{ldca00} and \eqref{ldca01} are naturally hold since $l_Q  $ and $r_Q: Q\rightarrow Cend(M) $.
For all $x,y,z\in M$, we have
\begin{align*}
&r_Q(T(r_Q(T(z))_{-\mu-\partial}y +l_Q(T(y))_\mu z))_{-\lambda-\partial}x=r_Q([T(y)_\mu T(z)] )_{-\lambda-\partial}x\\
=&r_Q(T(z))_{-\lambda-\mu-\partial}(r_Q(T(y))_{-\lambda-\partial}x)+l_Q(T(y))_\mu(r_Q(T(z))_{-\lambda-\partial}x),
\end{align*}
which means that \eqref{ldca1} holds. Similarly, we can verify that \eqref{ldca2} and \eqref{ldca3} hold and hence $(M,\lhd_\lambda,\rhd_\lambda)$ is an LDCA.

\end{proof}
\begin{ex}
Let $T:M\rightarrow Q$ be an $\mathcal{O}$-operator on a Leibniz conformal algebra $(Q,[\cdot_\lambda\cdot] )$ associated to   a representation $(M,l_Q,r_Q=0)$ of $Q$. Then there is an Leibniz conformal algebra structure on $M$ given by
$$
[x_\lambda y]=l_Q(T(x))_\lambda y,
$$
for all $x,y\in M$.
\end{ex}}

\delete{
\begin{pro}\label{mrl}
Let $T:M\rightarrow Q$ and $T':M'\rightarrow Q'$ be two $\mathcal{O}$-operators and $(\alpha,\beta)$ be a homomorphism of $\mathcal{O}$-operators from $T$ to $T'$. Then $\alpha$ is a homomorphism of  LDCAs from $(M,\lhd_\lambda,\rhd_\lambda)$ to $(M',\lhd'_\lambda,\rhd'_\lambda)$ which are defined by Proposition \ref{mld}.
\end{pro}
\begin{proof}
For all $x,y\in M$, we have
\begin{align*}
&\alpha(x\rhd_\lambda y)=\alpha(r_Q(T(y))_{-\lambda-\partial}x)=r'_Q(\beta(T(y))_{-\lambda-\partial}\alpha(x))
=r'_Q(T'(\alpha(y))_{-\lambda-\partial}\alpha(x))=\alpha(x)\rhd'_\lambda\alpha(y).
\end{align*}
and
\begin{align*}
&\alpha(x\lhd_\lambda y)=\alpha(l_Q(T(x))_{\lambda}y)=l'_Q(\beta(T(x))_{\lambda}\alpha(y))
=l'_Q(T'(\alpha(x))_{\lambda}\alpha(y))=\alpha(x)\lhd'_\lambda\alpha(y).
\end{align*}
Thus the conclusion holds.
\end{proof}
By Proposition \ref{Idrb}, Proposition \ref{mlr}, Proposition \ref{mld} and Proposition \ref{mrl}  we can define two functor $\mathcal{F}: \mathcal{LD}^c\rightarrow \mathfrak{R}_L$ and $\mathcal{G}: \mathfrak{R}_L\rightarrow \mathcal{LD}^c$ by
\begin{align*}
&\mathcal{F}:(Q,\lhd_\lambda,\rhd_\lambda)\rightarrow \id_Q,\qquad \mathcal{F}: f\rightarrow (f,f),\\
&\mathcal{G}: T\rightarrow (M,\lhd_\lambda,\rhd_\lambda),\qquad \mathcal{G}:(\alpha,\beta)\rightarrow \alpha,
\end{align*}
where $(Q,\lhd_\lambda,\rhd_\lambda)$ is any LDCA, $T: M\rightarrow Q$ is any $\mathcal{O}$-operator on $Q$ associated to   $M$, $f: Q\rightarrow Q'$ is any LDCA homomorphism and $(\alpha,\beta)$ is any
homomorphism of  $\mathcal{O}$-operators.\par

\begin{thm}
Let $\mathcal{F}: \mathcal{LD}^c\rightarrow \mathfrak{R}_L$ and $\mathcal{G}: \mathfrak{R}_L\rightarrow \mathcal{LD}^c$ be two functors defined above.
Then $\mathcal{F}$ is a left adjoint of $\mathcal{G}$.
\end{thm}
\begin{proof}
Let $(Q,\lhd_\lambda,\rhd_\lambda)$ and $(Q',\lhd'_\lambda,\rhd'_\lambda)$ be two LDCAs and $T:M\rightarrow Q'$ be an $\mathcal{O}$-operator on the Leibniz conformal algebra $Q'$ associated to   $M$.
We just need to prove that $Hom_{\mathfrak{R}_L}(\mathcal{F}(Q),T)\cong Chom_{\mathcal{LD}^c}(Q,\mathcal{G}(T))$.
Since $\mathcal{F}(Q)=\id_Q$ is an $\mathcal{O}$-operator on $(Q,[\cdot_\lambda\cdot]_{\lhd,\rhd})$ with $(Q, l_\lhd, r_\rhd)$, by Proposition \ref{mld} we have $\mathcal{G}(\mathcal{F}(Q))=\mathcal{G}(\id_Q)=Q$, and by Proposition \ref{mlr} and Proposition \ref{mrl} we have $\mathcal{G}(\mathcal{F}(f))=\mathcal{G}(f,f)=f$. Thus $\mathcal{G}\circ \mathcal{F}=\id_{\mathcal{LD}^c}$.
Then $Hom_{\mathfrak{R}_L}(\mathcal{F}(Q),T)\cong Chom_{\mathcal{LD}^c}(Q,\mathcal{G}(T))= Chom_{\mathcal{LD}^c}(\mathcal{G}(\mathcal{F}(Q)),\mathcal{G}(T))$.\par
There is a natural linear map $\mathcal{G}:Hom_{\mathfrak{R}_L}(\mathcal{F}(Q),T)\rightarrow  Chom_{\mathcal{LD}^c}(\mathcal{G}(\mathcal{F}(Q)),\mathcal{G}(T))$. It is clear that $\mathcal{G}$ is surjective.
By the definition of the homomorphism of $\mathcal{O}$-operators, for all $(\alpha,\beta)\in Hom_{\mathfrak{R}_L}(\mathcal{F}(Q),T)$ we have  $T\circ\alpha=\beta$ which means that $ker(\mathcal{G})=0$.
Therefore, $Hom_{\mathfrak{R}_L}(\mathcal{F}(Q),T)\cong Chom_{\mathcal{LD}^c}(Q,\mathcal{G}(T))$ which implies that
$\mathcal{F}$ is a left adjoint of $\mathcal{G}$.
\end{proof}  }
Finally, we give a construction of symmetric solutions of the CLCYBE from LDCAs.
\begin{thm}
Let $(Q,\lhd_\lambda,\rhd_\lambda )$ be a finite LDCA which is free as a ${\bf k}[\partial]$-module. Then
$$
r=\sum_{i=1}^n(a_i\otimes a^i+a^i\otimes a_i ) $$
is a solution of the CLCYBE in the Leibniz conformal algebra $Q\ltimes_{l^*_\lhd,-l^*_\lhd-r^*_\rhd}Q^{*c}$, where $\{a^i\}_{i=1}^n$ is a ${\bf k}[\partial]$-basis of $Q$ and $\{a_i\}_{i=1}^n $ is the dual ${\bf k}[\partial]$-basis of $Q^{*c}$.
\end{thm}
\begin{proof}
By Proposition \ref{Idrb}, $T=\id:Q\rightarrow Q$ is an
$\mathcal{O}$-operator on $(Q,[\cdot_\lambda
\cdot]_{\lhd,\rhd})$ associated
to $(Q,L_\lhd,R_\rhd)$. Then the conclusion follows from  Theorem
\ref{thamc}.
\end{proof}

\delete{
\cm{Since we only need to consider the correspondence between
Leibniz conformal bialgebras and Novikov bi-dialgebras, we might
consider to delete the following parts:
\begin{enumerate}
\item From Definition 5.1 to Corollary 5.5. \item From Definition
5.15 (with the above 2 lines) to Theorem 5.19
\end{enumerate}
}}

\section{A class of Leibniz conformal bialgebras corresponding to Novikov bi-dialgebras }\label{5}
We introduce the notion of Novikov bi-dialgebras and show that
there is a correspondence between Novikov bi-dialgebras and a
class of Leibniz conformal bialgebras. Furthermore, we introduce
the notion of classical duplicated Novikov Yang-Baxter equation
whose symmetric solutions give Novikov bi-dialgebras. In addition,
we show that symmetric solutions of the classical duplicated
Novikov Yang-Baxter equation correspond to a class of symmetric
solutions of the CLCYBE. In this section, we assume that $A$ is finite-dimensional as a vector space.

\subsection{Novikov bi-dialgebras } \delete{First, we use the notion of semi-direct
products to give the definition of representations of Novikov
dialgebras.
\begin{defi}
Let $(A,\dashv,\vdash)$ be a Novikov dialgebra, $V$ be a vector space and $l_A,r_A,\phi_A,\psi_A$ be four linear maps from $A$ to $\text{End}_{\bf k}(V)$.
Then $(V,l_A,r_A,\phi_A,\psi_A)$ is called a {\bf representation} of $(A,\dashv,\vdash)$ if
 $A\oplus V$ as a direct sum of $\bf k$-vector spaces can be endowed with a Novikov dialgebra structure as follows:
\begin{equation}\label{sdndp}
    (a+x)\dashv(b+y)=a\dashv b+l_A(a)y+ r_A(b)x,\quad
    (a+x)\vdash(b+y)=a\vdash b+\phi_A(a)y+\psi_A(b)x,
\end{equation}
for all $a, b\in A$ and $ x, y\in V$.
\end{defi}
If $(V,l_A,r_A,\phi_A,\psi_A)$ is a representation of $(A,\dashv,\vdash)$, we  set
\begin{eqnarray*}
  l_{\star A}(a)x=-l_A(a)x+r_A(a)x+\phi_A(a)x-\psi_A(a)x,\;\;\cm{ {\rm for}\;{\rm all}\;} a, b\in A,\;\;x\in V.
  \end{eqnarray*}

\begin{rmk}
Let $(A,\dashv,\vdash)$ be a Novikov dialgebra. Define four linear maps $L_\dashv,R_\dashv,L_\vdash,R_\vdash: A\rightarrow \text{End}_{\bf k}(V)$ by
\begin{equation}
 L_\dashv(a)b=a\dashv b=R_\dashv(b)a,\qquad L_\vdash(a)b=a\vdash b=R_\vdash(b)a,\qquad a,b\in A.
\end{equation}
Then $(A,L_\dashv,R_\dashv,L_\vdash,R_\vdash)$  is a representation of $(A,\dashv,\vdash)$, which is called the {\bf regular representation} of $(A,\dashv,\vdash)$. For convenience, we set $a\star b=-a\dashv b-b\vdash a+a\vdash b+b\dashv a=L_\star (a)b$, where $a$, $b\in A$.
\end{rmk}

\delete{
\begin{align}
& l(a \dashv b)x=r(b)(l(a)x),\label{ndam1}\\
&r(b)(r(a)x)=r(a)(r(b)x),\label{ndam2}\\
&l(a\vdash b)x=\psi(b)(\phi(a)x)=\psi(b)(l(a)x),\label{ndam3}\\
&r(b)(\phi(a)x)=\phi(a \vdash b)x=\phi(a \dashv b)x,\label{ndam4}\\
&r(b)(\psi(a)x)=\psi( a)(\psi(b)x)=\psi(a)(r(b) x ),\label{ndam5}\\
&l(a)(l(b)x)=l(a)(\phi(b)x),\label{ndam6}\\
&l(a)(r(b)x)=l(a)(\psi(b)x),\label{ndam7}\\
&r(a\dashv b)x=r(a\vdash b)x,\label{ndam8}\\
&l(a\dashv b )x-l(a)(l(b)x )=l(b\vdash a)x-\phi(b)(l(a)x),\label{ndam9}\\
&r(b)(l(a)x)-l(a)(r(b)x)=r(b)(\psi(a)x)-\psi(a \dashv b)x,\label{ndam10}\\
&r(b)(r(a)x)-r(a\dashv b)x=r(b)(\phi(a)x)-\phi(a)(r(b)x),\label{ndam11}\\
&\phi(a\dashv b)x-\phi(a)(\phi(b)x)=\phi(b\vdash a)x-\phi(b)(\phi(a)x),\label{ndam12}\\
&\psi(b)(l(a)x)-\phi(a)(\psi(b)x)=\psi(b)(\psi(a)x)-\psi(a\vdash b)x,\label{ndam13}
\end{align}
for all $a,b\in A$ and $x\in V$.}


\delete{
\begin{pro}\label{dsmn}
Let $A$ be a NDA.
Then $(V,l_A,r_A,\phi_A,\psi_A)$ is a representation of $A$ if and only if  $A\oplus V$ as a direct sum of $\bf k$-vector spaces can be endowed with a natural NDA structure as follows$:$
\begin{equation}\label{sdndp}
(a,x)\dashv(b,y)=(a\dashv b,l_A(a)y+ r_A(b)x),\quad
(a,x)\vdash(b,y)=(a\vdash b,\phi_A(a)y+\psi_A(b)x),
\end{equation}
for all $a,b\in A$ and $x,y\in V$.
Denote this NDA by $A\ltimes_{l_A,r_A,\phi_A,\psi_A} V$ which is called the semi-direct product of $A$ and $V$.
\end{pro}
\begin{proof}
It can be obtained by some straightforward calculation.
\end{proof}}

\begin{pro}\label{rplcnd}
Let $(\bar{A}={\bf k}[\partial]A,[\cdot_\lambda\cdot])$ be the Leibniz conformal algebra corresponding to a Novikov dialgebra $(A,\dashv,\vdash)$ and $M={\bf k}[\partial]V$ which is free as a ${\bf k}[\partial]$-module on $V$. Define two ${\bf k}[\partial]$-module homomorphisms $l_{r_A,\phi_A}$ and $r_{l_A,\psi_A}:\bar{A}\rightarrow \text{Cend}(M)$ by
\begin{align*}
&l_{r_A,\phi_A}(a)_\lambda v=\partial(r_A(a)v)+\lambda(\phi_A(a)v+r_A(a)v),\\
&r_{l_A,\psi_A}(a)_\lambda v=\partial(l_A(a)v)+(-\lambda-\partial)(\psi_A(a)v+l_A(a)v),\;\;a\in A, v\in V,
\end{align*}
where $l_A,r_A,\phi_A,\psi_A: A\rightarrow \text{End}_{\bf k}(V)$. Then $(M,l_{r_A,\phi_A},r_{l_A,\psi_A})$ is a representation of $(\bar{A},[\cdot_\lambda\cdot])$ if and only if  $(V,l_A,r_A,\phi_A,\psi_A)$ is a representation of $(A,\dashv,\vdash)$.
\end{pro}
\begin{proof}
Let $a,b\in A$ and $u,v\in V$. Note that
\begin{align*}
[(a+u)_\lambda(b+v)]=&[a_\lambda b]+l_{r_A,\phi_A}(a)_\lambda v+r_{l_A,\psi_A}(b)_{-\lambda-\partial}u\\
=&\partial(b\dashv a)+\lambda(b\dashv a+a\vdash b)+\partial(r_A(a)v)+\lambda(\phi_A(a)v+r_A(a)v)\\
&+\partial(l_A(b)u)+\lambda(\psi_A(b)u+l_A(b)u)\\
=&\partial(b\dashv a +r_A(a)v+l_A(b)u)+\lambda(b\dashv a+a\vdash b+\phi_A(a)v+r_A(a)v+\\ &\psi_A(b)u+l_A(b)u).
\end{align*}
We define
\begin{align*}
(a+u)\dashv_{A,V}(b+v)=a\dashv b+r_A(b)u+l_A(a)v,\quad (a+u)\vdash_{A,V}(b+v)=a\vdash b+\phi_A(a)v+\psi_A(b)u.
\end{align*}
By Proposition \ref{ss}, $(M,l_{r_A,\phi_A},r_{l_A,\psi_A})$ is a representation of $\bar{A}$ if and only if $\bar{A}\oplus M $ is a Leibniz conformal algebra with the $\lambda$-bracket above. It is clear that $\bar{A}\oplus M={\bf k}[\partial](A\oplus V) $ is a Leibniz conformal algebra if and only if $(A\oplus V,\dashv_{A,V},\vdash_{A,V}) $ is a Novikov dialgebra by Proposition \ref{lcnd}. Hence the conclusion holds.
\end{proof}
Let $U$ and $V$ be two vector spaces. For a linear map $\varphi: U\rightarrow \text{End}_{\bf k}(V)$, define its dual map $\varphi^*:U\rightarrow \text{End}_{\bf k}(V^*)$ by $\langle \varphi^*(a)f,v\rangle=-\langle f,\varphi(a)v\rangle$ for all $a\in U$, $f\in V^*$ and $v\in V$.

\begin{pro}\label{pro-dual}
Let $(A,\dashv,\vdash)$ be a Novikov dialgebra and $\bar{A}$ be the Leibniz conformal algebra corresponding to  $(A,\dashv,\vdash)$.
Suppose that $(V,l_A,r_A,\phi_A,\psi_A)$ is a representation of $(A,\dashv,\vdash)$ and $M={\bf k}[\partial]V$  is free as a ${\bf k}[\partial]$-module on $V$. Then $(M^{*c},l^*_{r_A,\phi_A},-l^*_{r_A,\phi_A}-r^*_{l_A,\psi_A})=({\bf k}[\partial]V^*,$ $ l_{-r^*_A,\phi^*_A+r^*_A },r_{{l_\star^*}_A,\psi^*_A-r^*_A })$ is a representation of $\bar{A}$.
\end{pro}
\begin{proof}
Note that $M^{*c}={\bf k}[\partial]V^*$. For all $a\in A$, $f\in V^*$ and $v\in V$, we have
\begin{align*}
&\langle l^*_{r_A,\phi_A}(a)_\lambda f,v\rangle_\mu=-\langle f,l_{r_A,\phi_A}(a)_\lambda v \rangle_{\mu-\lambda}\\
=&-\langle f,\partial(r_A(a)v)+\lambda(\phi_A(a)v+r_A(a)v)\rangle_{\mu-\lambda}
=-\langle f,\mu r_A(a)v)+ \lambda \phi_A(a)v\rangle\\
=&\langle (\mu r^*_A(a) +\lambda \phi^*_A(a))f,v\rangle
=\langle \partial (-r^*_A(a) f)+\lambda ((\phi^*_A+r^*_A)(a)f-r^*_A(a)f ),v \rangle_\mu\\
=&\langle l_{-r^*_A,\phi^*_A+r^*_A } (a)_\lambda f,v \rangle_\mu,
\end{align*}
and
\begin{align*}
&\langle (-l^*_{r_A,\phi_A}-r^*_{l_A,\psi_A})(a)_\lambda f,v\rangle_\mu=\langle f, (l_{r_A,\phi_A}+r_{l_A,\psi_A})(a)_\lambda v \rangle_{\mu-\lambda}\\
=&\langle f, \partial((l_A+r_A)(a)v)+\lambda((\phi_A+r_A)(a)v) +(-\lambda-\partial)((\psi_A+l_A)(a)v) \rangle_{\mu-\lambda}\\
=&\langle \partial({l_\star^*}_A(a)f)+(-\lambda-\partial)({l_\star^*}_A(a)f +(\psi^*_A-r^*_A)(a)f        ),v\rangle_\mu\\
=&\langle r_{{l_\star^*}_A,\psi^*_A-r^*_A } (a)_\lambda f,v \rangle_\mu.
\end{align*}
Thus by Proposition \ref{dual}, this conclusion holds.
\end{proof}
The following corollary follows straightforwardly from Propositions \ref{rplcnd} and \ref{pro-dual}.
\begin{cor}
Let $(A,\dashv,\vdash)$ be a Novikov dialgebra and  $(V,l_A,r_A,\phi_A,\psi_A)$ be a representation of $A$. Then $(V^*,{l_\star^*}_A,-r_A^*,\phi^*_A+r^*_A,\psi^*_A-r^*_A)$ is also a representation of $(A,\dashv,\vdash)$.\delete{ which is called the {\bf dual representation} of $(V,l_A,r_A,\phi_A,\psi_A) $.}
\end{cor}

\delete{
\begin{proof}
We just need to  verify that  $-l^*_A-\psi^*_A+r^*_A+\phi^*_A,-r_A^*,\phi^*_A+r^*_A,\psi^*_A-r^*_A$ satisfy \eqref{ndam1}-\eqref{ndam13} by  replacing $l,r,\phi,\psi$ by $-l^*_A-\psi^*_A+r^*_A+\phi^*_A,-r^*_A,\phi^*_A+r^*_A,\psi^*_A-r^*_A$.

For all $a,b\in A$ $f\in V^*$ and $v\in V$, we have
\begin{align*}
& \langle (-l^*_A-\psi^*_A+r^*_A+\phi^*_A)(a \dashv b)f+r^*(b)((-l^*_A-\psi^*_A+r^*_A+\phi^*_A) (a)f),v\rangle\\
=&\langle f, (l_A+\psi_A-r_A-\phi_A)(a \dashv b)v+ (-{l_\star}_A)(a)(r(b)v)\rangle.
\end{align*}
Since $(V,l_A,r_A,\phi_A,\psi_A)$ is an $A$-bimodule,
then by \eqref{ndam1}, \eqref{ndam5} and \eqref{ndam10}, we get $(l_A+\psi_A)(a\dashv b)v-(l+\psi)(a)(r(b)v)=0 $ and by \eqref{ndam2}, \eqref{ndam4} and \eqref{ndam11}, we get $-(r_A+\phi_A)(a\dashv b)v+(r_A+\phi_A)(a)(r(b)v)=0$, which means that \eqref{ndam1} holds. Then the conclusion follows from some similar calculations.
\end{proof}}

\delete{

\begin{defi}
Let $(A,\dashv,\vdash)$ and $(B,\dashv,\vdash)$ be two Novikov dialgebras. Let $l_A,r_A,\phi_A,\psi_A: A\rightarrow End_\mathbb{C}(B)$ and $l_B,r_B,\phi_B,\psi_B: B\rightarrow End_\mathbb{C}(A)$ be linear maps. If there is a Novikov dialgebra structure on the direct sum $A\oplus B$ of the underlying vector spaces of $A$ and $B$ given by
\begin{align}
&(a,x)\dashv (b,y)=(a\dashv b+l_B(x)b+r_B(y)a,x\dashv y+l_A(a)y+r_A(b)x),\\
&(a,x)\vdash (b,y)=(a\vdash b+\phi_B(x)b+\psi_B(y)a,x\vdash y+\phi_A(a)y+\psi_A(b)x),
\end{align}
for all $a,b\in A$ and $x,y\in B$. Then we call the resulting octuple $(A,B;(l_A,r_A,\phi_A,\psi_A), (l_B,r_B,\phi_B,$ $\psi_B))$ a {\bf matched pair} of Novikov dialgebras.
\end{defi}

\delete{
\begin{pro}
Let $(A,\dashv,\vdash)$ and $(B,\dashv,\vdash)$ be two NDAs. The octuple $(A,B;(l_A,r_A,\phi_A,\psi_A), (l_B,r_B,$ $\phi_B,\psi_B))$ is a matched pair of NDAs $A$ and $B$ if and only if $(A,l_B,r_B,\phi_B,\psi_B)$ is a $B$-bimodule, $(B,l_A,r_A,\phi_A,\psi_A)$ is an $A$-bimodule and they  satisfy the compatibility conditions as follows:
\begin{flalign}
&\label{ndmp1}r_B(x)(a\dashv b)=(r_B(x)a)\dashv b+l_B(l_A(a)x)b,\\
&\label{ndmp2}(l_B(x)a)\dashv b+l_B(r_A(a)x)b=(l_B(x)b)\dashv a+l_B(r_A(b)x)a,\\
&\label{ndmp3}r_B(x)(a\vdash b)=(\psi_B(x)a)\vdash b+\phi_B(\phi_A(a)x)b=(r_B(x)a)\vdash b+\phi_B(l_A(a)x)b,\\
&\label{ndmp4}(\psi_B(x)a)\dashv b+l_B(\phi_A(a)x)b=\psi_B(x)(a\vdash b)=\psi_B(x)(a\dashv b),\\
&\label{ndmp5}(\phi_B(x)a)\dashv b+l_B(\psi_A(a)x)b=(\phi_B(x)b)\vdash a+\phi_B(\psi_A(b)x)a=(l_B(x)b)\vdash a+\phi_B(r_A(b)x)a,\\
&\label{ndmp6}a\dashv(r_B(x)b)+r_B(l_A(b)x)a=a\dashv(\psi_B(x)b)+r_B(\phi_A(b)x)a,&
\end{flalign}
\begin{flalign}
&\label{ndmp7}a\dashv(l_B(x)b)+r_B(r_A(b)x)a=a\dashv(\phi_B(x)b)+r_B(\psi_A(b)x)a,\\
&\label{ndmp8}l_B(x)(a\dashv b)=l_B(x)(a\vdash b),\\
&\label{ndmp9}r_B(x)(a\dashv b)-a\dashv(r_B(x)b)-r_B(l_A(b)x)a=r_B(x)(b\vdash a)-b\vdash(r_B(x)a)-\psi_B(l_A(a)x)b,\\
&\label{ndmp10}(r_B(x)a)\dashv b+l_B(l_A(a)x)b-a\dashv(l_B(x)b)-r_B(r_A(b)x)a\\
&\qquad=(\phi_B(x)a)\dashv b+l_B(\psi_A(a)x)b-\phi_B(x)(a\dashv b),\nonumber\\
&\label{ndmp11}(l_B(x)a)\dashv b+l_B(r_A(a)x)b-l_B(x)(a\dashv b)\\
&\qquad=(\psi_B(x)a)\dashv b+l_B(\phi_A(a)x)b-a\vdash(l_B(x)b)-\psi_B(r_A(b)x)a,\nonumber\\
&\label{ndmp12}\phi_B(x)(a\vdash b)=-\phi_B(\phi_A(a)x-\psi_A(a)x)b+(\phi_B(x)a-\psi_B(x)a)\vdash b\\
&\qquad+\psi_B(\psi_A(b)x)a+a\vdash(\phi_B(x)b),\nonumber\\
&\label{ndmp13}\psi_B(x)(a\vdash b-b\vdash a)=a\vdash(\psi_B(x)b)-b\vdash(\psi_B(x)a)+\psi_B(\phi_A(b)x)a-\psi_B(\phi_A(a)x)b,\\
&\label{ndmp14}r_A(a)(x\dashv y)=(r_A(a)x)\dashv y+l_A(l_B(x)a)y,\\
&\label{ndmp15}(l_A(a)x)\dashv y+l_A(r_B(x)a)y=(l_A(a)y)\dashv x+l_A(r_B(y)a)x,\\
&\label{ndmp16}r_A(a)(x\vdash y)=(\psi_A(a)x)\vdash y+\phi_A(\phi_B(x)a)y=(r_A(a)x)\vdash y+\phi_A(l_B(x)a)y,\\
&\label{ndmp17}(\psi_A(a)x)\dashv y+l_A(\phi_B(x)a)y=\psi_A(a)(x\vdash y)=\psi_A(a)(x\dashv y),\\
&\label{ndmp18}(\phi_A(a)x)\dashv y+l_A(\psi_B(x)a)y=(\phi_A(a)y)\vdash x+\phi_A(\psi_B(y)a)x=(l_A(a)y)\vdash x+\phi_A(r_B(y)a)x,\\
&\label{ndmp19}x\dashv(r_A(a)y)+r_A(l_B(y)a)x=x\dashv(\psi_A(a)y)+r_A(\phi_B(y)a)x,\\
&\label{ndmp20}x\dashv(l_A(a)y)+r_A(r_B(y)a)x=x\dashv(\phi_A(a)y)+r_A(\psi_B(y)a)x,\\
&\label{ndmp21}l_A(a)(x\dashv y)=l_A(a)(x\vdash y),\\
&\label{ndmp22}r_A(a)(x\dashv y)-x\dashv(r_A(a)y)-r_A(l_B(y)a)x=r_A(a)(y\vdash x)-y\vdash(r_A(a)x)-\psi_A(l_B(x)a)y,\\
&\label{ndmp23}(r_A(a)x)\dashv y+l_A(l_B(x)a)y-x\dashv(l_A(a)y)-r_A(r_B(y)a)x\\
&\qquad=(\phi_A(a)x)\dashv y+l_A(\psi_B(x)a)y-\phi_A(a)(x\dashv y),\nonumber\\
&\label{ndmp24}(l_A(a)x)\dashv y+l_A(r_B(x)a)y-l_A(a)(x\dashv y)\\
&\qquad=(\psi_A(a)x)\dashv y+l_A(\phi_B(x)a)y-x\vdash(l_A(a)y)+\psi_A(r_B(y)a)x,\nonumber\\
&\label{ndmp25}\phi_A(a)(x\vdash y)=-\phi_A(\phi_B(x)a-\psi_B(x)a)y+(\phi_A(a)x-\psi_A(a)x)\vdash y\\
&\qquad+\psi_A(\psi_B(y)a)x+x\vdash(\phi_A(a)y),\nonumber\\
&\label{ndmp26}\psi_A(a)(x\vdash y-y\vdash x)=x\vdash(\psi_A(a)y)-y\vdash(\psi_A(a)x)+\psi_A(\phi_B(y)a)x-\psi_A(\phi_B(x)a)y,&
\end{flalign}
for all $a,b\in A$ and $x,y\in B$.
\end{pro}
\begin{proof}
The conclusion can be obtained by direct calculation, and we omit the details.
\end{proof}}

Next, we give a correspondence between matched pairs of Novikov dialgebras and certain matched pairs of
the corresponding Leibniz conformal algebras as follows.
\begin{pro}\label{ppmp2}
Let $\bar{A}={\bf k}[\partial] A$ and $\bar{B}={\bf k}[\partial] B$ be the Leibniz conformal algebras corresponding to Novikov dialgebras $(A,\dashv_A,\vdash_A)$ and $(B,\dashv_B,\vdash_B)$ respectively. Let $l_A,r_A,\phi_A,\psi_A:A\rightarrow End(B)$ and $l_B,r_B,\phi_B,\psi_B :B\rightarrow End(A) $ be linear maps. Then $(\bar{A},\bar{B};(l_{r_A,\phi_A },r_{l_A,\psi_A} ),(\phi_{r_B,\phi_B },\psi_{l_B,\psi_B} ) ) $ is a matched pair of Leibniz conformal algebras if and only if   $(A,B;(l_A,r_A,\phi_A,\psi_A), (l_B,r_B,$ $\phi_B,\psi_B))$ is a matched pair of Novikov dialgebras.
\end{pro}
\begin{proof}
The proof of this proposition is similar to Proposition \ref{rplcnd}.
\end{proof}}}

First, we introduce the notion of Novikov co-dialgebras.
\begin{defi}
A {\bf Novikov co-dialgebra} is a triple $(A,\delta,\Delta)$ where
$ \Delta, \delta: A\rightarrow A\otimes A$ are linear maps satisfying
\begin{align}
&(\delta\otimes \id)\delta(a)=\tau_{23}(\delta\otimes \id)\delta (a),\label{ncda1}\\
&(\Delta\otimes \id)\delta(a)=\tau_{23}(\Delta\otimes \id)\Delta(a)=\tau_{23}(\delta\otimes \id)\Delta(a),\label{ncda2}\\
&(\id\otimes \delta)\delta(a)=(\id\otimes \Delta)\delta(a),\label{ncda3}\\
&(\delta \otimes \id-\id\otimes \delta)\delta(a)=\tau_{12}((\Delta\otimes \id)\delta(a)-(\id\otimes \delta)\Delta(a) ),\label{ncda4}\\
&(\Delta\otimes \id-\id\otimes
\Delta)\Delta(a)=\tau_{12}(\Delta\otimes \id-\id\otimes
\Delta)\Delta(a)\;\;\;{\rm for}\;{\rm all}\;a\in
A.\label{ncda5}
\end{align}
\end{defi}
It is easy to check that $(A,\delta,\Delta)$ is a Novikov co-dialgebra if and only if $(A^*,\dashv,\vdash)$ is a Novikov dialgebra defined by
$$\langle x\dashv y, a\rangle=\langle x\otimes y,\delta(a)\rangle, \qquad \langle x\vdash y,a\rangle=\langle x\otimes y,\Delta(a)\rangle\;\; \;  {\rm for}\;{\rm all}\; x, y\in A^\ast,\;a, b\in A.$$
\begin{pro}
Let $\bar{A}={\bf k}[\partial]A$ be the free $ {\bf k}[\partial]$-module on a vector space $A$. Suppose that $\delta$ and $\Delta$ are ${\bf k}[\partial]$-module homomorphisms from $\bar{A}$ to $\bar{A}\otimes \bar{A}$ such that $\delta (A)\subset A\otimes A$ and $\Delta(A) \subset A\otimes A$. Define a ${\bf k}[\partial]$-module homomorphism $\alpha: \bar{A}\rightarrow \bar{A}\otimes \bar{A}$ by
\begin{equation}\label{cmnda}
\alpha(a)=(\partial\otimes \id)\Delta(a)-\tau(\partial\otimes
\id)\delta(a)\;\;\; {\rm for}\;{\rm all}\; a\in A.
\end{equation}
Then $(\bar{A},\alpha)$ is a Leibniz conformal coalgebra if and only if $(A,\delta,\Delta )$ is a Novikov co-dialgebra.
\end{pro}
\begin{proof}
 Let $a\in A$. By the definition of $\alpha$, we have
\begin{align*}
(\id\otimes \alpha)\alpha(a)=&(\id\otimes \alpha)((\partial\otimes \id)\Delta(a)-\tau(\partial\otimes \id)\delta(a))\\
=&(\partial\otimes \partial\otimes \id)(\id\otimes \Delta)\Delta(a)-(\partial\otimes \id \otimes \partial )\tau_{23}(\id\otimes \delta)\Delta(a)\\
&- \partial_2(\id\otimes\partial^{\otimes^2})(\id\otimes\Delta ) \tau\delta(a)+\partial_3(\id\otimes \partial^{\otimes^2})\tau_{23}(\id\otimes\delta)\tau\delta(a),\\
(\alpha\otimes \id)\alpha(a)=&(\alpha\otimes \id)((\partial\otimes \id)\Delta(a)-\tau(\partial\otimes \id)\delta(a))\\
=&\partial_1(\partial^{\otimes^2}\otimes \id)(\Delta\otimes \id)\Delta(a)-\partial_2(\partial^{\otimes^2}\otimes \id)\tau_{12}(\delta\otimes \id)\Delta(a)\\
&-(\partial\otimes \id\otimes \partial)(\Delta\otimes \id)\tau\delta(a)+(\id\otimes \partial\otimes\partial)\tau_{12}(\delta\otimes \id)\tau\delta(a).
\end{align*}
Then we consider
\begin{eqnarray}\label{eq-co}
(\id\otimes \alpha)\alpha(a)=(\alpha\otimes
\id)\alpha(a)+\tau_{12}(\id\otimes \alpha)\alpha(a) \;\;\; {\rm
for}\;{\rm all}\; a\in A.
\end{eqnarray}
Considering the terms of (\ref{eq-co})  in  $ A\otimes A\otimes \partial^2 A$, we can get
$$
\tau_{13}(\delta\otimes
\id)\delta(a)=\tau_{13}\tau_{23}(\delta\otimes
\id)\delta(a)\;\;\;{\rm for}\;{\rm all}\; a\in A,
$$
which is equivalent to  \eqref{ncda1}. Similarly, by considering
the terms of (\ref{eq-co}) in $\partial^2 A\otimes A\otimes A$ and
$A\otimes \partial^2 A\otimes A$, we get that \eqref{ncda2} holds.
By considering the terms of (\ref{eq-co}) in $\partial A\otimes
A\otimes \partial A$ and $A\otimes\partial A\otimes \partial A$,
we obtain \eqref{ncda3}. Moreover, we have \eqref{ncda4} and
\eqref{ncda5} by considering the terms of (\ref{eq-co}) in
$A\otimes \partial A\otimes \partial A$ and $\partial A\otimes
\partial A\otimes A$. \delete{Finally, we get \eqref{co-eq} holds
if and only if \eqref{ncda1}-\eqref{ncda5} hold.}Then this conclusion holds.
\end{proof}

\delete{
\begin{pro}  \label{prmpndba}
Let $A$ and $A^*$ be two Novikov dialgebras. Then $(A,A^*;(-l^*-\psi^*+r^*+\phi^*,-r^*,\phi^*+r^*,\psi^*-r^*),(-L^*-g^*+R^*+f^*,-R^*,f^*+R^*,g^*-R^*)) $ is a matched pair if and only if $\delta$ and $\Delta$ satisfy
\begin{flalign}
&\label{ndba1}(l(a)\otimes \id)\Delta(b)=-(\id\otimes (l+\psi)(b))\delta(a),\\
&\label{ndba2}(\id\otimes l(b))\delta (a)=(l(a)\otimes \id)\tau\delta(b),\\
&\label{ndba3}(\id\otimes (l+\psi)(a))(\delta(b)+\tau\Delta(b))=(\psi(b)\otimes \id)\Delta(a),\\
&\label{ndba4}\delta(a\dashv b)=(r(b)\otimes \id)\delta(a)+(\id\otimes (-l_\star)(a))(\id-\tau)(\delta(b)-\Delta(b)),\\
&\label{ndba5}\delta(a\vdash b)=(\psi(b)\otimes \id)(\delta(a)-\Delta(a))+(\id\otimes(\phi+r)(a))(\delta(b)+\tau\Delta(b)),\\
&\label{ndba6}\Delta(a\dashv b)=(\id\otimes (l+\psi)(a))(\id-\tau)(\Delta(b)-\delta(b))+(r(b)\otimes \id)\Delta(a),\\
&\label{ndba7}\Delta(a\vdash b)=((\psi-r)(b)\otimes \id)(\delta(a)-\Delta(a))+(\id\otimes (\phi+r)(a))(\tau\delta(b)+\Delta(b)),\\
&\label{ndba8}((l+\psi)(b)\otimes \id)(\delta(a)-\Delta(a))+(\id\otimes(-l_\star)(b))\tau(\delta(a)-\Delta(a))\\
&-((\phi+r)(a)\otimes \id)\delta(b)+(\id\otimes (\phi+r)(a))\tau\Delta(b)=0,\nonumber\\
&\label{ndba9}(\id\otimes \psi(a))(\id-\tau)(\delta(b)-\Delta(b))-(r(b)\otimes \id)(\Delta(a)+\tau\delta(a))\\
&=-(\id\otimes r(b))(\Delta(a)+\tau\delta(a))+((\psi-r)(a)\otimes \id)(\id-\tau)(\Delta(b)-\delta(b)),\nonumber&
\end{flalign}
for all $a,b\in A$.
\end{pro}
\begin{proof}
Let $A$ and $A^*$ be two Novikov dialgebras. Consider their bimodules $(A^*,-l^*-\psi^*+r^*+\phi^*,-r^*,\phi^*+r^*,\psi^*-r^*)$ and $(A,-L^*-g^*+R^*+f^*,-R^*,f^*+R^*,g^*-R^*)$.
For all $a,b\in A$ and $x,y\in A^*$, \eqref{ndmp6} becomes
\begin{align*}
&a\dashv(-R^*(x)b)-R^*((-l^*-\psi^*+r^*+\phi^* )(b)x )a-a\dashv((g^*-R^*)(x)b)+R^*((r^*+\phi^* )(b)x )a\\
=&R^*((l^*+\psi^*)(b)x)a-a\dashv(g^*(x)b).
\end{align*}
Furthermore, by straightforward  calculation, we have
\begin{align*}
&\langle R^*((l^*+\psi^*)(b)x)a- a\dashv(g^*(x)b),y\rangle\\
=&\langle (\id\otimes (l+\psi)(b))\delta (a)+(l(a)\otimes \id)\Delta(b) ,y\otimes x\rangle
\end{align*}
which means that \eqref{ndmp6} is equivalent to  \eqref{ndba1}.\par
Similarly, if \eqref{ndba1} holds then \eqref{ndmp7} is equivalent to \eqref{ndba2}.
\eqref{ndmp3} is equivalent to \eqref{ndba3} and \eqref{ndba5}.
\eqref{ndmp1} is equivalent to \eqref{ndba4}.
If \eqref{ndba4} and \eqref{ndba5} hold, then
\eqref{ndmp4} is equivalent to \eqref{ndba6} and \eqref{ndba7}.
If \eqref{ndmp1} and \eqref{ndmp3} hold, then
 \eqref{ndmp9} is equivalent to \eqref{ndba8}.
\eqref{ndmp5} is equivalent to \eqref{ndba9}  and
$(\psi(b)\otimes \id)(\Delta(a)+\tau\delta(a))=(\id\otimes \psi(a))(\delta(b)+\tau\Delta(b))$ which can be deduced by \eqref{ndba1}-\eqref{ndba3}.

\eqref{ndmp2} becomes
\begin{align*}
&((-L^*-g^*+R^*+f^*)(x)a)\dashv b+(-L^*-g^*+R^*+f^*)(-r^*(a)x)b\\
=&((-L^*-g^*+R^*+f^*)(x)b)\dashv a+(-L^*-g^*+R^*+f^*)(-r^*(b)x)a.
\end{align*}
Furthermore, by straightforward calculation, we have
\begin{align*}
&\langle ((-L^*-g^*+R^*+f^*)(x)a)\dashv b+(-L^*-g^*+R^*+f^*)(-r^*(a)x)b,y\rangle\\
=&\langle (\id\otimes r(b)-r(b)\otimes \id)(\id-\tau)(\delta(a)-\Delta(a)),x\otimes y\rangle,\\
&-\langle ((-L^*-g^*+R^*+f^*)(x)b)\dashv a+(-L^*-g^*+R^*+f^*)(-r^*(b)x)a,y\rangle\\
=&\langle (\id\otimes r(a)-r(a)\otimes \id)(\id-\tau)(\delta(b)-\Delta(b)),x\otimes y\rangle.
\end{align*}
Hence  by \eqref{ndba9} we can deduce that \eqref{ndmp2}. \par
Similarly, by \eqref{ndmp4} we can deduce that \eqref{ndmp8}, by \eqref{ndmp2} and \eqref{ndmp5} we can deduce that \eqref{ndmp10}, and by \eqref{ndmp2} and \eqref{ndmp4}, we can deduce that \eqref{ndmp11}. Furthermore, by \eqref{ndba5}, \eqref{ndba7} and \eqref{ndba9}, we can deduce that \eqref{ndmp12} and by \eqref{ndba4}-\eqref{ndba7} and \eqref{ndba8}, we can deduce that \eqref{ndmp13}. Hence, we show that \eqref{ndmp1}-\eqref{ndmp13} are equivalent to \eqref{ndba1}-\eqref{ndba9}.

On the other hand, \eqref{ndmp14} becomes
\begin{align*}
(-r^*)(a)(x\dashv y)+(r^*(a)x)\dashv y-(-l^*-\psi^*+r^*+\phi^*)((-L^*-g^*+f^*+R^*)(x)a)y=0.
\end{align*}
Furthermore, by straightforward calculation, we have
\begin{align*}
&\langle (-r^*)(a)(x\dashv y)+(r^*(a)x)\dashv y-(-l^*-\psi^*+r^*+\phi^*)((-L^*-g^*+f^*+R^*)(x)a)y,b\rangle\\
=&\langle x\dashv y, b\dashv a\rangle -\langle r^*(a)x, R^*(y)b \rangle+\langle y,  (-l_\star)((-L^*-g^*+f^*+R^*)(x)a) b\rangle\\
=&-\langle x,R^*(y)(b\dashv a) \rangle+\langle x, (R^*(y)b)\dashv a\rangle\\
&+\langle (-l^*-\psi^*+r^*+\phi^*)(b)y,  (-L^*-g^*+f^*+R^*)(x)a\rangle\\
=&\langle x,-R^*(y)(b\dashv a)+ (R^*(y)b)\dashv a\rangle-\langle (-L-g+f+R)(x)((-l^*-\psi^*+r^*+\phi^*)(b)y),a  \rangle\\
=&\langle x, -R^*(y)(b\dashv a)+ (R^*(y)b)\dashv a - (-L^*-g^*+f^*+R^*)((-l^*-\psi^*+r^*+\phi^*)(b)y)a\rangle,
\end{align*}
which means that \eqref{ndmp14} can be deduced by \eqref{ndmp1}. Similarly
\eqref{ndmp15}-\eqref{ndmp26} can be deduced by \eqref{ndmp2}-\eqref{ndmp13}. Therefore, the conclusion holds.
\end{proof}}

\begin{defi}
Let $(A,\dashv,\vdash)$ be a Novikov dialgebra and $(A,\delta,\Delta)$ be a Novikov co-dialgebra. If $\Delta$ and $\delta$ satisfy the following conditions for all $a, b\in A$: 
{\small
\begin{flalign}
    &\label{ndba1}( L_\dashv(a)\otimes \id)\Delta(b)=-(\id\otimes ( L_\dashv +R_\vdash)(b))\delta(a),\\
    &\label{ndba2}(\id\otimes  L_\dashv( b ))\delta (a)=( L_\dashv (a)\otimes \id)\tau\delta(b),\\
    &\label{ndba3}(\id\otimes (L_\dashv+R_\vdash)(a))(\delta(b)+\tau\Delta(b))=(R_\vdash(b)\otimes \id)\Delta(a),\\
    &\label{ndba4}\delta(a\dashv b)=(R_\dashv(b)\otimes \id)\delta(a)+(\id\otimes L_\star(a))(\id-\tau)(\delta(b)-\Delta(b)),\\
    &\label{ndba5}\delta(a\vdash b)=(R_\vdash(b)\otimes \id)(\delta(a)-\Delta(a))+(\id\otimes(L_\vdash+ R_\dashv)(a))(\delta(b)+\tau\Delta(b)),\\
    &\label{ndba6}\Delta(a\dashv b)=(\id\otimes (L_\dashv+R_\vdash)(a))(\id-\tau)(\Delta(b)-\delta(b))+(R_\dashv(b)\otimes \id)\Delta(a),\\
    &\label{ndba7}\Delta(a\vdash b)=((R_\vdash-R_\dashv)(b)\otimes \id)(\delta(a)-\Delta(a))+(\id\otimes (L_\vdash+R_\dashv)(a))(\tau\delta(b)+\Delta(b)),\\
    &\label{ndba8}((L_\dashv+R_\vdash)(b)\otimes \id)(\delta(a)-\Delta(a))+(\id\otimes L_\star(b))\tau(\delta(a)-\Delta(a))\\
    &-((L_\vdash+R_\dashv)(a)\otimes \id)\delta(b)+(\id\otimes (L_\vdash+R_\dashv)(a))\tau\Delta(b)=0,\nonumber\\
    &\label{ndba9}(\id\otimes R_\vdash(a))(\id-\tau)(\delta(b)-\Delta(b))-(R_\dashv(b)\otimes \id)(\Delta(a)+\tau\delta(a))\\
    &=-(\id\otimes R_\dashv(b))(\Delta(a)+\tau\delta(a))+((R_\vdash-R_\dashv)(a)\otimes \id)(\id-\tau)(\Delta(b)-\delta(b)),\nonumber&
\end{flalign}}
then $(A,\dashv,\vdash,\delta,\Delta)$ is called a {\bf Novikov bi-dialgebra}.
\end{defi}
\delete{It should be pointed out that if $ \vdash=\dashv=\circ$
then a Novikov dialgebra $(A,\dashv,\vdash)$ is a Novikov algebra
$(A,\circ)$. Moreover, we can find that in this case, $l=\phi$ and
$r=\psi$ implies $(V^*,l^*+r^*,-r^*)$ is a dual representation of
$A$. However this bialgebra theory is not equal to Novikov
bialgebras.} Next, we give a correspondence between Novikov
bi-dialgebras and a class of Leibniz conformal bialgebras as
follows.
\begin{thm}\label{thnblb}
Let $\bar{A}={\bf k}[\partial]A$ be the Leibniz conformal algebra
corresponding to a Novikov dialgebra $ (A,\dashv,\vdash)$. Suppose
that $\delta$ and $\Delta $ are two ${\bf k}[\partial]$-module
homomorphisms from $\bar{A}$ to $\bar{A}\otimes \bar{A}$ such that
$\delta(A)\subset A\otimes A$ and $\Delta(A)\subset A\otimes A$.
Then $(\bar{A},[\cdot_\lambda \cdot],\alpha) $ is a Leibniz
conformal bialgebra if and only if
$(A,\dashv,\vdash,\delta,\Delta)$  is a Novikov bi-dialgebra,
where $\alpha$ is given by \eqref{cmnda}.  We call
$(\bar{A},[\cdot_\lambda \cdot],\alpha) $ the {\bf Leibniz
conformal bialgebra corresponding to a Novikov bi-dialgebra}
$(A,\dashv,\vdash,\delta,\Delta)$.
\end{thm}
\begin{proof}
Let $a,b\in A$ and set $\Delta(a)=\sum a_{(1)}\otimes a_{(2)}$ and $\delta(a)=\sum a^{(1)}\otimes a^{(2)}$.
Then we have
{\small
\begin{align*}
&(R(a)_{\lambda-\partial\otimes \id}\otimes \id)\alpha(b)-\tau((R(b)_{-\lambda-\partial\otimes \id}\otimes \id)\alpha(a))\\
=&(R(a)_{\lambda-\partial\otimes \id}\otimes \id)((\partial\otimes \id)\Delta(b)-\tau(\partial\otimes \id)\delta(b))
-\tau((R(b)_{-\lambda-\partial\otimes \id}\otimes \id)((\partial\otimes \id)\Delta(a)-\tau(\partial\otimes \id)\delta(a))\\
=&\sum [{\partial b_{(1)}}_{-\lambda}a]\otimes b_{(2)}-\sum[b^{(2)}_{-\lambda}a]\otimes\partial b^{(1)}-\tau(\sum[\partial {a_{(1)}}_\lambda b]\otimes a_{(2)}-[a^{(2)}_{\lambda}b]\otimes\partial a^{(1)} )\\
=&\sum(\lambda(\partial(a\dashv b_{(1)})-\lambda(a\dashv b_{(1)}+b_{(1)}\vdash a)\otimes b_{(2)})
)  +\sum(-\partial(a\dashv b^{(2)})\otimes \partial b^{(1)} \\
&+\lambda(a\dashv b^{(2)}+b^{(2)}\vdash a  )\otimes \partial b^{(1)}) + \sum( a_{(2)}\otimes (\lambda\partial(b\dashv a_{(1)}) + \lambda^2 (b\dashv a_{(1)}+a_{(1)}\vdash b)))\\
&+ \sum (\partial a^{(1)}\otimes ( \partial(b\dashv a^{(2)})+\lambda(a^{(2)}\vdash b+  b\dashv a^{(2)} ))).
\end{align*}}
Thus \eqref{lb1} holds if and only if
{\small
\begin{align*}
&(L_\dashv(a)\otimes \id)\Delta(b)=-(\id\otimes (L_\dashv+R_\vdash)(b))\delta(a),\quad((L_\dashv+R_\vdash)a\otimes \id)\Delta(b)=(\id\otimes (L_\dashv+R_\vdash)b)\tau\Delta(a),\\
&(L_\dashv(a)\otimes \id)\tau\delta(b)=(\id\otimes L_\dashv(b))\delta(a),
\end{align*}}
by considering the coefficients of $\lambda(\partial A\otimes A)$,
$\lambda^2 A\otimes A$ and $\partial^2 A\otimes A $.
 That is, \eqref{lb1} holds if and only if
\eqref{ndba1}-\eqref{ndba3} hold. \par Similarly, we have
{\small
\begin{flalign*}
&(\id\otimes R(b)_{-\lambda-\partial^{\otimes^2}}-L(b)_{-\lambda-\partial^{\otimes^2}}\otimes \id-R(b)_{-\lambda-\partial^{\otimes^2}}\otimes \id)(\tau+\id)\Delta(a)\\
&+(\id\otimes L(a)_\lambda+L(a)_\lambda\otimes \id)\Delta(b)-\alpha([a_\lambda b])\\
=&(\partial\otimes \partial)((\id\otimes L_\dashv(b))(\id-\tau)(\Delta(a)-\delta(a))+( (L_\dashv-L_\vdash)(b)\otimes \id )(\id-\tau)(\Delta(a)-\delta(a))\\
&+(\id\otimes R_\dashv(a))\Delta(b)-(R_\dashv(a)\otimes \id)\tau\delta(b))      +(\partial^2\otimes \id)((\id\otimes (L_\dashv+R_\vdash)(b))(\id-\tau)(\Delta(a)-\delta(a))  \\
&+(R_\dashv(a)\otimes \id)\Delta(b) )+(\id\otimes \partial^2)((L_\star(b)\otimes \id)(\id-\tau)(\delta(a)-\Delta(a))-(\id\otimes R_\dashv(a))\tau\delta(b))\\
&+\lambda(\partial\otimes \id)((\id\otimes (L_\dashv+R_\vdash)(b))((\id-\tau)(\Delta(a)-\delta(a))-\tau(\Delta(a)-\delta(a))) \\
&+((L_\vdash-L_\dashv)(b)\otimes \id)(\delta(a)-\Delta(a))+(\id\otimes (L_\vdash+R_\dashv)(a))\Delta(b)+(L_\vdash+2R_\dashv\otimes \id)\Delta(b)    ) \\
&+\lambda(\id\otimes \partial)( (\id\otimes L_\dashv(b))\tau(\delta(a)-\Delta(a))+( L_\star(b)\otimes \id )(\id-\tau)(\delta(a)-\Delta(a)) \\
&+( L_\star(b)\otimes \id )(\delta(a)-\Delta(a)) -( \id\otimes (L_\vdash+2R_\dashv)a+(L_\vdash+R_\dashv)(a)\otimes \id)\tau \delta(b)  ) \\
&+\lambda^2((\id\otimes (L_\dashv+R_\vdash)(b))\tau(\delta(a)-\Delta(a)) +(L_\star(b)\otimes \id ) (\delta(a)-\Delta(a)) \\
& -(\id\otimes (L_\vdash+R_\dashv)(a))\tau\delta(b)+((L_\vdash+R_\dashv)\otimes \id)\Delta(b)    )-(\partial\otimes \partial)(\Delta(b\dashv a)-\tau\delta(b\dashv a)) \\
&-(\partial^2\otimes \id)\Delta(b\dashv a)+(\id\otimes \partial^2)\tau\delta(b\dashv a)
-\lambda((\partial\otimes \id)\Delta(b\dashv a+a\vdash b )\\
&+(\id\otimes \partial)\tau\delta(b\dashv a+a\vdash b ))=0.
\end{flalign*}}
Then by considering the coefficients of $\partial^2 A\otimes A$,
$A\otimes \partial^2 A$ and $\lambda^2 A\otimes A$ in the above
equality, we obtain that the three equalities are equivalent to
\eqref{ndba6}, \eqref{ndba4} and \eqref{ndba8} respectively. Furthermore by considering the
coefficients of $\partial A\otimes
\partial A $, $\lambda(\partial A\otimes A)$ and $\lambda(A\otimes
\partial A) $ in the above equality, we get that the three
equalities  are equivalent to \eqref{ndba9}, \eqref{ndba7} and
\eqref{ndba5} respectively. Hence the conclusion holds.
\end{proof}

\subsection{The  classical duplicate Novikov Yang-Baxter equation }\

Let $ (A,\dashv,\vdash)$ be a Novikov dialgebra and  $r=\sum_i a_i\otimes b_i\in A\otimes A $.
We define
\begin{equation*}
N(r)=\sum_{i,j} a_i\vdash a_j\otimes b_i\otimes b_j+a_i\otimes (b_i\vdash a_j-b_i\dashv a_j)\otimes b_j+a_i\otimes a_j\otimes (b_j\vdash b_i+b_i\dashv b_j).
\end{equation*}
The equation $N(r)=0$ is called the {\bf classical duplicate Novikov Yang-Baxter equation (CDNYBE)} in $A$.
In addition, we define
\begin{equation*}
M(r)=\sum_{i,j} a_i\vdash b_j\otimes b_i\otimes a_j+a_i\otimes (b_i\vdash b_j-b_i\dashv b_j)\otimes a_j+a_i\otimes b_j\otimes (a_j\vdash b_i+b_i\dashv a_j).
\end{equation*}
It is clear that if $r$ is symmetric, then $N(r)=M(r)$.

\begin{pro}
Let $ (A,\dashv,\vdash)$ be a Novikov dialgebra and $r=\sum_i a_i\otimes b_i\in A\otimes A $. Define
\begin{align}
&\Delta_r: A\rightarrow A\otimes A,\qquad  \Delta_r(a)=( (L_\dashv-L_\vdash)(a)\otimes \id-\id\otimes(L_\dashv+R_\vdash)(a))r,\label{ndadr1}\\
&\delta_r: A\rightarrow A\otimes A,\qquad  \delta_r(a)=(\id\otimes
{L_\star}(a)+L_\dashv(a)\otimes \id)\tau r \;\;\;{\rm
for}\;{\rm all}\; a\in A. \label{ndadr2}
\end{align}
Then $(A,\Delta_r,\delta_r)$ is a Novikov co-dialgebra if and only if the following conditions hold for all $a\in A:$

{\small
\begin{flalign}
&\label{1ncda}(\id-\tau_{23})(\id\otimes \id \otimes (L_\star(a) )(\tau_{13}N(r)-N(\tau r))\\
&+(\id-\tau_{23})\sum_i(\id\otimes \id\otimes L_\star(a_i))(b_i\otimes (\id\otimes L_\dashv(a))(r-\tau r))=0,\nonumber\\
&\label{2ncda}-(\id\otimes \id\otimes L_\star(a))\tau_{23}\tau_{12}N(r)
+\tau_{12}\sum_i b_i\otimes (\id\otimes (L_\dashv+R_\vdash)(a\dashv a_i-a\vdash a_i))(r-\tau r)& \\
&-( \id\otimes (L_\dashv+R_\vdash)(a)\otimes \id)\tau_{12}M(\tau r) +(\id\otimes (L_\vdash+R_\dashv)(b_i)\otimes \id)(a_i\otimes (L_\dashv(a)\otimes \id)(r-\tau r) ) \nonumber\\
&-\tau_{23}( (L_\dashv( a\dashv a_i-a\vdash a_i )\otimes \id )(r-\tau r)\otimes b_i  )=0,\nonumber\\
&\label{3ncda}\sum_i(L_\dashv(a\dashv a_i-a\vdash a_i)\otimes \id+\id\otimes (L_\dashv+R_\vdash)(a\vdash a_i-a\dashv a_i ))(r-\tau r)\otimes b_i & \\
&+(\id\otimes \id\otimes (L_\dashv+R_\vdash)(a))(\tau_{23}\tau_{12}M(\tau r)+\tau_{13}M(\tau r)-M(\tau r)\nonumber\\
&+\sum_i ( b_i\otimes (\id\otimes L_\dashv(a_i))r-a_i\otimes (\id \otimes L_\dashv (b_i)) \tau r)  )=0,\nonumber\\
&\label{4ncda}(L_\dashv(a)\otimes \id\otimes \id)(M(\tau r)+\tau_{23}M(\tau r)-\tau_{12}\tau_{23}M(\tau r)
+\sum_i((R_\vdash(b_i)\otimes \id)r\otimes a_i \\
&-(R_\vdash(a_i)\otimes \id)\tau r\otimes b_i))+\sum_i b_i\otimes ((\id\otimes (L_\dashv+R_\vdash)(a\star a_i)-L_\dashv(a\star a_i)\otimes \id)(r-\tau r))=0 \nonumber\\
&\label{5ncda}(L_\dashv(a)\otimes \id\otimes \id)(N(\tau r)-\tau_{12}N(\tau r)+\sum_i((R_\vdash-R_\dashv)(b_i)\otimes \id)(r-\tau r)\otimes a_i)\\
&+(\id\otimes (L_\dashv-L_\vdash)(a)\otimes \id)(\tau_{12}\tau_{23}N(r) -\tau_{23}N(\tau r)+\sum_i b_i\otimes (L_\dashv(a_i)\otimes \id)(r-\tau r))\nonumber \\
&+\sum_i(\id\otimes \id\otimes L_\star(a_i))(b_i\otimes (\id\otimes (L_\dashv+R_\vdash)(a))(r-\tau r))=0,\nonumber\\
&\label{6ncda}((L_\dashv-L_\vdash)(a)\otimes \id\otimes \id)N(r)-\sum_i(R_\vdash(a\dashv a_i-a\vdash a_i)\otimes \id)(r-\tau r)\otimes b_i\\
&+\sum_i ((R_\dashv-R_\vdash)(a_i)\otimes \id\otimes \id)( (R_\vdash(a)\otimes \id)(r-\tau r)\otimes b_i )-(\id\otimes (L_\dashv-L_\vdash)(a)\otimes \id)\tau_{12}N(r)\nonumber\\
&+\sum_i(\id\otimes R_\vdash(a\vdash a_i-a\dashv a_i) +  (\id \otimes (R_\dashv-R_\vdash)(a_i) )(\id\otimes R_\vdash(a) )  )(r-\tau r)\otimes b_i\nonumber\\
&+\sum_i (\id\otimes \id \otimes (L_\dashv+R_\vdash)(a) )(\id-\tau_{12})(a_i\otimes (\id\otimes L_\dashv(b_i))(r-\tau r)  )=0.  \nonumber&
\end{flalign}}
\end{pro}
\begin{proof}
We just need to prove that \eqref{1ncda}-\eqref{6ncda} are equivalent to \eqref{ncda1}-\eqref{ncda5}.
For all $a\in A $, we have
{\small
\begin{flalign*}
&(\delta_r\otimes \id)\delta_r(a)-\tau_{23}((\delta_r\otimes \id)\delta_r(a))\\
=&\sum_{i,j}(\id-\tau_{23})(  b_j\otimes (b_i\star a_j)\otimes (a\star a_i) +(b_i\dashv b_j)\otimes a_j\otimes (a\star a_i) + b_j\otimes ((a\dashv b_i)\star a_j  )\otimes a_i \\
&+(a\dashv b_i)\dashv b_j\otimes a_j\otimes a_i        )
=(\id-\tau_{23})(\id\otimes \id\otimes L_\star(a) )(\tau_{13}N(r)-N(\tau r))\\
&+(\id-\tau_{23})\sum_{i,j}(b_i\otimes b_j\otimes a\star (a_i\dashv a_j) +  b_j\otimes ((a\dashv b_i)\star a_j  )\otimes a_i)  )\\
=&(\id-\tau_{23}) \{    (\id\otimes \id \otimes L_\star(a) )(\tau_{13}N(r)-N(\tau r))+\sum_i(\id\otimes \id\otimes L_\star(a_i))(b_i\otimes (\id\otimes L_\dashv(a))(r-\tau r)) \},&
\end{flalign*}}
which means that \eqref{ncda1} is equivalent to \eqref{1ncda}.
Similarly, we have
\begin{eqnarray*}
\eqref{ncda2} \Longleftrightarrow \eqref{2ncda}\;{\rm
and}\;\eqref{3ncda},\;\; \eqref{ncda3} \Longleftrightarrow
\eqref{4ncda},\;\;  \eqref{ncda4} \Longleftrightarrow
\eqref{5ncda},\;\;  \eqref{ncda5} \Longleftrightarrow
\eqref{6ncda}.\;\;
\end{eqnarray*}
Thus this conclusion holds.
\end{proof}

\begin{thm}\label{thm-cob-dN}
Let $ (A,\dashv,\vdash)$ be a Novikov dialgebra and $r\in
A\otimes A $. The maps $\Delta_r$ and $\delta_r$ given by
\eqref{ndadr1} and \eqref{ndadr2} respectively are
co-multiplications of a Novikov bi-dialgebra  $A$ if and only if
they satisfy \eqref{1ncda}-\eqref{6ncda} and the following
conditions for all $a,b\in A:$
\begin{flalign}
&\label{1nbda}(L_\dashv(a)\otimes \id)(\id\otimes (L_\dashv+R_\vdash)(b))(r-\tau r)=0,\\
&\label{2nbda}(L_\dashv(a)\otimes \id)(\id\otimes L_\dashv(b))(r-\tau r)=0,\\
&\label{3nbda}(R_\vdash(b)\otimes \id)(\id\otimes (L_\dashv+R_\vdash)(a))(r-\tau r)=0,\\
&\label{4nbda}(\id\otimes L_\star(a))(\id\otimes R_\vdash(b) +(R_\vdash-R_\dashv)(b)\otimes \id)(r-\tau r)=0,\\
&\label{7nbda}((R_\vdash-R_\dashv)(b)\otimes \id)(L_\dashv(a)\otimes \id+\id\otimes L_\star(a))(r-\tau r)=0,\\
&\label{8nbda}(\id\otimes L_\star(b))(\id\otimes L_\dashv(a)+ L_\star(a)\otimes \id)(r-\tau r)\\
&+((L_\dashv+R_\vdash)(b)\otimes \id)(\id\otimes (L_\dashv+R_\vdash)(a))(r-\tau r)=0,\nonumber\\
&\label{9nbda}(R_\vdash(b)\otimes \id)(\id\otimes R_\vdash(a))(r-\tau r)+((R_\vdash-R_\dashv)(a)\otimes \id)(\id\otimes (R_\vdash-R_\dashv)(b))(r-\tau r)&\\
&+( R_\dashv(b)\otimes \id)(R_\vdash(a)\otimes \id)(r-\tau r)-(
R_\dashv(a)\otimes \id)(R_\vdash(b)\otimes \id)(r-\tau
r)=0.\nonumber
\end{flalign}
\end{thm}
\begin{proof} $r=\sum_i a_i\otimes b_i$.
For all $a,b\in A$, we have
\begin{align*}
&(L_\dashv(a)\otimes \id)\Delta_r(b)+(\id\otimes (L_\dashv+R_\vdash)(b))\delta_r(a)\\
=&\sum_i((a\dashv (b\dashv a_i)-a\dashv(b\vdash a_i))\otimes b_i -(a\dashv a_i)\otimes(b\dashv b_i+b_i\dashv b)\\
&+(a\dashv b_i)\otimes (b\dashv a_i+a_i\vdash b)+b_i\otimes (b\dashv (a\star a_i)+(a\star a_i)\vdash b)     )\\
=&-(L_\dashv(a)\otimes \id)(\id\otimes (L_\dashv+R_\vdash)(b))(r-\tau r),
\end{align*}
which means that \eqref{ndba1} is equivalent to \eqref{1nbda}.
Similarly, we have
\begin{align*}
&\eqref{ndba2}\Longleftrightarrow \eqref{2nbda}, \quad \eqref{ndba3}, \eqref{ndba5} \; {\rm and }\; \eqref{ndba6}  \Longleftrightarrow \eqref{3nbda},\quad
\eqref{ndba4}\Longleftrightarrow \eqref{4nbda},\\
&\eqref{ndba7}\Longleftrightarrow \eqref{7nbda},\quad \eqref{ndba8}\Longleftrightarrow \eqref{8nbda},\quad \eqref{ndba9}\Longleftrightarrow \eqref{9nbda}.
\end{align*}
\delete{
obtain that \eqref{ndba2}-\eqref{ndba9}  are
equivalent to \eqref{2nbda}-\eqref{9nbda}. \cm{Similarly as the
proof of Proposition 5.5, give the exact correspondence in  terms
of $\Longleftrightarrow$} }Thus this conclusion holds.
\end{proof}

\begin{cor}\label{condba}
Let $ (A,\dashv,\vdash)$ be a Novikov dialgebra. If $r\in A\otimes
A $ is a symmetric solution of the CDNYBE in  $A$, then $
(A,\dashv,\vdash,\delta_r,\Delta_r)$ is a Novikov bi-dialgebra,
where  $\Delta_r$ and $\delta_r$ are defined by \eqref{ndadr1} and
\eqref{ndadr2} respectively.
\end{cor}
\begin{proof}
It is straightforward by Theorem \ref{thm-cob-dN}.
\end{proof}
Next, we give the relationship between the CLCYBE and
the CDNYBE as follows.
\begin{pro}\label{relcdn}
Let $ (A,\dashv,\vdash)$ be a Novikov dialgebra, $(\bar{A}={\bf k}[\partial]A,[\cdot_\lambda\cdot])$ be the Leibniz conformal algebra corresponding to $(A,\dashv,\vdash)$ and $r\in A\otimes A$ be symmetric. 
Then $r$ is a solution of the CLCYBE in $\bar{A}$ if and only if $r$ is a solution of the CDNYBE in $A$.
\end{pro}
\begin{proof} Set $r=\sum_{i} a_i\otimes b_i$.
Note that
{\small
\begin{align*}
[[r,r]]=&\sum_{i,j}( \partial_1(a_i\dashv a_i \otimes b_i\otimes b_j-a_i\otimes(b_i\vdash a_j+a_j\dashv b_i)\otimes b_j )\\
&+\partial_2((a_i\vdash a_j+a_j\dashv a_i )\otimes b_i\otimes b_j+a_i\otimes a_j\otimes (b_j\vdash b_i+b_i\dashv b_j)\\
&-a_i\otimes a_j(\dashv b_i+b_i\dashv a_j)\otimes b_j )+\partial_3(a_i\otimes a_j\otimes b_i\dashv b_j-a_i \otimes  (a_j\vdash b_i+b_i\dashv a_j)\otimes b_j))\\
=&\sum_{i,j}(\partial_1(-a_i\vdash a_j\otimes b_i\otimes b_j -a_i\otimes (b_i\vdash a_j-b_i\dashv a_j)\otimes b_j -a_i\otimes a_j\otimes (b_j\vdash b_i+b_i\dashv b_j))   \\
&+\partial_3 ( -a_i\otimes a_j\otimes (b_j\vdash b_i)   -a_i\otimes (a_j\vdash b_i-a_j\dashv b_i )\otimes b_j\\
&-(a_i\vdash a_j+a_j\dashv a_i)\otimes b_i\otimes b_j    ) )\;\;\text{mod} \;\;(\partial^{\otimes ^3}).
\end{align*}}
Since $r$ is symmetric, we have $[[r,r]]= -\partial_1N(r)-\partial_3\tau_{13}N(r)$. Thus this conclusion holds.
\end{proof}


\begin{cor}
Let $ (A,\dashv,\vdash)$ be a Novikov dialgebra and $(\bar{A}={\bf
k}[\partial]A,[\cdot_\lambda\cdot])$ be the Leibniz conformal
algebra corresponding to $(A,\dashv,\vdash)$. Suppose that $r\in
A\otimes A$ is a symmetric solution of the CDNYBE in $A$. On the
one hand, there is a Novikov bi-dialgebra $
(A,\dashv,\vdash,\delta_r,\Delta_r)$ by Corollary \ref{condba}
where  $\Delta_r$ and $\delta_r$ are defined by \eqref{ndadr1} and
\eqref{ndadr2} respectively. By Theorem~\ref{thnblb}, let
$(\bar{A},[\cdot_\lambda \cdot],\alpha_r) $ be the corresponding
Leibniz conformal bialgebra, where  $\alpha_r$ is defined by
\eqref{cmnda}. On the other hand,  $r$ is also the symmetric
solution of the CLCYBE in $\bar{A}$ by Proposition \ref{relcdn}.
By Corollary \ref{condba}, there is a triangular Leibniz conformal
bialgebra $(\bar{A},[\cdot_\lambda\cdot],\alpha_r')$, where
$\alpha_r'$ is defined by {\rm (\ref{dr})}. Therefore both
the Leibniz conformal bialgebras
$(\bar{A},[\cdot_\lambda\cdot],\alpha_r)$ and
$(\bar{A},[\cdot_\lambda\cdot],\alpha_r')$ coincide. That is,
there is the following commutative diagram.
\begin{equation}\label{diag2}
{\small \xymatrix@C=1.5cm{
  & \txt{   symmetric solutions of \\the CDNYBE } \ar[d]^{\text{Prop.} \ref{relcdn}} \ar[r]^{\quad \text{Coro.} \ref{condba}}  &\txt{ Novikov\\ bi-dialgebras} \ar[d]^{\text{Thm.}\ref{thnblb}}         \\
  & \txt{ symmetric solutions of \\the CLCYBE} \ar[u]\ar[r]^{\quad \text{Coro.} \ref{colcba}} & \txt{Leibniz \\ conformal bialgebras}\ar[u] }               }
\end{equation}
\end{cor}

\begin{proof}
Let $r=\sum_ia_i\otimes b_i$ and $a\in A$. By Proposition
\ref{lcnd}, we have

\begin{align*}
\alpha_r'(a):=&( L(a)_\lambda\otimes \id+R(a)_\lambda\otimes
\id-\id\otimes
R(a)_\lambda)(r)|_{\lambda=-\partial^{\otimes^2}}\\
=&\sum_i ([a_{-\partial^{\otimes^2}} a_i]\otimes b_i+[{a_i}_{\id\otimes \partial} a]\otimes b_i-a_i\otimes [{b_i}_{\partial\otimes \id} a] )\\
=& (\partial\otimes \id)((-a\vdash a_i\otimes b_i)+(a\dashv a_i\otimes b_i) -(a_i\otimes b_i\vdash a)-(a_i\otimes a\dashv b_i)  )\\
&-(\id \otimes \partial)( (a_i\dashv a  +a\vdash a_i-a_i\vdash a-a\dashv a_i)\otimes b_i +(a_i\otimes a\dashv b_i ) )\\
=&(\partial\otimes \id)\Delta_r(a)-\tau (\partial\otimes \id
)\delta_r(a):=\alpha_r(a).
\end{align*}
Hence the conclusion holds.
\end{proof}


\delete{Let $A$ and $V$ be vector spaces. We identify $r\in A\otimes A$ as a linear map  $T_r: A^*\rightarrow A$ by
\begin{align*}
\langle T_r(a^*),b^*\rangle=\langle r,a^*\otimes b^*\rangle, \qquad a^*,b^*\in A^*.
\end{align*}
\delete{ for all $a,b\in A$, $x,y\in A^*$, $f\in \text{Hom}_{\bf k}(A^*,A)$. and then introduce the notion of $\mathcal{O}$-operators of Novikov dialgebras and use it to interpret the CDNCYBE.}
\begin{defi}
Let $(A,\dashv,\vdash) $ be a Novikov dialgebra and $(V,l_A,r_A,\phi_A,\psi_A )$ be a representation of $(A,\dashv,\vdash)$. A linear map $T: V\rightarrow A$ is called an {\bf $\mathcal{O}$-operator} on the Novikov dialgebra $(A,\dashv,\vdash) $ associated to   the representation $(V,l_A,r_A,\phi_A,\psi_A )$, if it satisfies
\begin{equation}\label{rneq}
T(x)\dashv T(y)=T(l_A(T(x))y +r_A(T(y))x),\qquad T(x)\vdash T(y)=T(\phi_A(T(x))y +\psi_A(T(y))x),\;\;x, y\in V.
\end{equation}
In particular, if $V=A$ is the regular representation of $(A,\dashv,\vdash) $, then $T$ is called a {\bf Rota-Baxter operator} on $(A,\dashv,\vdash) $.
\end{defi}
\begin{pro}\label{pprbeq}
Let $(A,\dashv,\vdash) $ be a Novikov dialgebra, $(V,l_A,r_A,\phi_A,\psi_A )$ be a representation of $(A,\dashv,\vdash)$, $(\bar{A}={\bf k}[\partial] A,[\cdot_\lambda\cdot ])$ be the Leibniz conformal algebra corresponding to $(A,\dashv,\vdash)$ and $(M={\bf k}[\partial]V,l_{r_A,\phi_A},r_{l_A,\psi_A})$ be  a representation of $\bar{A}$. Let $T$ be a {\bf k}$[\partial]$-module homomorphism from $ M$ to $\bar{A} $,  which satisfies $T(V) \subset A$.
 Then $T$ is an $\mathcal{O}$-operator on $(A,\dashv,\vdash)$ associated to   $(V,l_A,r_A,\phi_A,\psi_A )$ if and only if $ T$ is an $\mathcal{O}$-operator on $\bar{A}$ associated to  $(M,l_{r_A,\phi_A},r_{l_A,\psi_A})$.
\end{pro}
\begin{proof}
For all $x,y\in V$, we have
\begin{align*}
[T(x)_\lambda T(y)]=\partial(T(y)\dashv T(x) )+\lambda( T(x)\vdash T(y) +T(y)\dashv T(x) ),
\end{align*}
and
\begin{align*}
&T( l_{r_A,\phi_A}(T(x))_\lambda y+r_{l_A,\psi_A}( T(y))_{-\lambda-\partial}x)\\
=&T(\partial( l_A(T(y))x +r_A(T(x) )y)+\lambda( (\phi_A+r_A)(T (x))y+(\psi_A+l_A)(T(y) )x    )   ),
\end{align*}
which means that \eqref{oeq} holds if and only if \eqref{rneq} holds. Thus this conclusion holds.
\end{proof}
\begin{pro}
    Let $(A,\dashv,\vdash) $ be a Novikov dialgebra. 
    Define the co-multiplications $\delta_r$ and $\Delta_r$ by \eqref{ndadr1} and \eqref{ndadr2}. Then $A^*$ is endowed with a Novikov dialgebra structure $(A^*,\dashv^*,\vdash^*)$, where $\dashv^*$ and $\vdash^*$ are defined by
    \begin{align*}
        &x\dashv^* y={L_\star^*}(T_{\tau r}x)y-R_\dashv^*(T_r y)x,\\
        &x\vdash^* y=(R_\vdash^*-R_\dashv^*)(T_{\tau r}y)x+(R_\dashv^*+L_\vdash^*)(T_r x)y,
    \end{align*}
    for all $x,y\in A^*$.
\end{pro}
\begin{proof}
    It can be obtained by some straightforward calculation.
\delete{
    For $a\in A$ and $x,y\in A^*$, we have
    \begin{align*}
        \langle x\dashv^* y,a\rangle=&\langle x\otimes y,\delta_r (a)\rangle
        =-\langle (\id\otimes {L_\star^*}(a)+L_\dashv^*(a)\otimes \id)x\otimes y,\tau r \rangle\\
        =&-\langle {L_\star^*}(a)y,\xi(\tau r)x \rangle-\langle L_\dashv^*(a)x,\xi(r)y \rangle\\
        =&\langle {L_\star^*}(\xi(\tau r)x)y- R_\dashv^*(\xi(r)y)x, a\rangle,
    \end{align*}
    which means that $ x\dashv^* y={L_\star^*}(\xi(\tau r)x)y-R_\dashv^*(\xi(r)y)x$.
    Similarly, we can get the other equation by some calculation.}
\end{proof}
\begin{pro}
    Let $(A,\dashv,\vdash) $ be a Novikov dialgebra and $r\in A\otimes A$.
    Then $r$ is a solution of the CDNYBE in $(A,\dashv,\vdash)$ if and only if the following equality holds:
    \begin{align}
        &T_{\tau r} x\vdash T_{\tau r}y=T_{\tau r}((R_\vdash^*-R_\dashv^* )(T_{\tau r}y)x+((L_\vdash^*+R_\dashv^*)(T_{r}x)y  )    ).\label{ndop1}
    \end{align}
    In particular, if $r$ is symmetric, then $T_r$ is an $\mathcal{O}$-operator on $(A,\dashv,\vdash)$ associated to $(A^*, L_\star^\ast, -R_\dashv^*, L_\vdash^*+R_\dashv^*, R_\vdash^*-R_\dashv^*)$.
\end{pro}
\delete{
For $x,y,z\in A^*$, we have
\begin{align*}
&\sum_{i,j} \langle x\otimes y\otimes z, a_i\vdash a_j\otimes b_i\otimes b_j \rangle
=\sum_{i,j}\langle x,\langle y,b_i\rangle a_i\vdash\langle z,b_j\rangle a_j \rangle=\langle x,\xi(\tau r) y\vdash \xi(\tau r)z\rangle.
\end{align*}
Similarly, we can get
\begin{align*}
&\sum_{i,j} \langle x\otimes y\otimes z, a_i\otimes(b_i\vdash a_j-b_i\dashv a_j) \otimes b_j \rangle
=\langle x,\xi(\tau r )( (r_A^*-\psi_A^*)( \xi (\tau r) z)y      )\rangle,\\
&\sum_{i,j} \langle x\otimes y\otimes z, a_i\otimes a_j\otimes(b_j\vdash b_i+b_i\dashv b_j) \rangle
=-\langle x,\xi(\tau r)((r^*+\phi^*)(\xi (r) y)z   ) \rangle.
\end{align*}
Thus \eqref{x1} holds. We also can get
\begin{align*}
&\sum_{i,j} \langle x\otimes y\otimes z, a_i\vdash a_j\otimes b_i\otimes b_j \rangle\\
=&-\langle z,\xi(r)(\phi_A^*(\xi (\tau r)y)x)\rangle,\\
&\sum_{i,j} \langle x\otimes y\otimes z, a_i\otimes(b_i\vdash a_j-b_i\dashv a_j) \otimes b_j \rangle \\
=&\langle z, \xi(r)( (l_A^*-\phi_A^*)(\xi (r)x)y  )  \rangle,\\
&\sum_{i,j} \langle x\otimes y\otimes z, a_i\otimes a_j\otimes(b_j\vdash b_i+b_i\dashv b_j) \rangle \\
=&\langle z, \rangle,
\end{align*}
 which means that \eqref{x2} holds.
}
\begin{proof}
We set $r=\sum_i a_i\otimes b_i\in A\otimes A$. For all $x,y,z\in A^*$, we have
\begin{align*}
    &\sum_{i,j} \langle  a_i\vdash a_j\otimes b_i\otimes b_j, x\otimes y\otimes z\rangle
    =\sum_{i,j}\langle a_i \vdash a_j\rangle \langle b_i,y\rangle \langle b_j,z\rangle\\
    =&\sum_{i,j}\langle  \langle b_i,y\rangle  a_i\vdash \langle b_j,z\rangle a_j,x \rangle
    =\langle T_{\tau r}y\vdash T_{\tau r}z,x \rangle,\\
&\sum_{i,j} \langle  a_i\otimes(b_i\vdash a_j-b_i\dashv a_j) \otimes b_j,x\otimes y\otimes z \rangle
=\sum_{i,j}\langle a_i,x\rangle \langle b_i\vdash T_{\tau r}z-b_i\dashv T_{\tau r}z ,y  \rangle\\
=&-\sum_{i,j} \langle a_i,x\rangle \langle b_i,(R^*_\vdash -R^*_\dashv)(T_{\tau r}z)y \rangle
=-\langle T_{\tau r} (R^*_\vdash -R^*_\dashv )(T_{\tau r} z)y ,x     \rangle,\\
&\sum_{i,j} \langle  a_i\otimes a_j\otimes(b_j\vdash b_i+b_i\dashv b_j) ,x\otimes y\otimes z  \rangle
=\langle a_i,x\rangle \langle (T_r y)\vdash b_i+b_i\dashv (T_r y),z\rangle\\
=&-\langle  T_{\tau r}(  L^*_\vdash +L^*_\dashv)(T_r y)z,x\rangle.
\end{align*}
Thus we have
\begin{equation*}
\langle N(r), x\otimes y\otimes z\rangle =\langle  T_{\tau r}y\vdash T_{\tau r}z-T_{\tau r} (R^*_\vdash -R^*_\dashv )(T_{\tau r} z)y  -T_{\tau r}(  L^*_\vdash +L^*_\dashv)(T_r y)z,x\rangle,
\end{equation*}
which means that $r$ is a solution of the CDNYBE in $(A,\dashv,\vdash)$ if and only if \eqref{ndop1} holds.
Similarly, we get
\begin{equation*}
\langle N(r), x\otimes y\otimes z\rangle =\langle  -T_r(L_\vdash^*(T_{\tau r}y)x-(L_\vdash^*-L_\dashv^*)(T_r x)y +T_r y\vdash T_r x+T_r x\dashv T_r y ),z \rangle,
\end{equation*}
which means that $r$ is a solution of the CDNYBE in $(A,\dashv,\vdash)$ if and only if the following equality holds:
$$
T_r y\vdash T_r x+T_r x\dashv T_r y=T_r(L_\vdash^*(T_{\tau r}y)x -(L_\dashv^*-L_\vdash^*)(T_r x)y  ).
$$
Thus, if $r$ is symmetric, then $T_r$ is an $\mathcal{O}$-operator on $(A,\dashv,\vdash)$ associated to $(A^*, L_\star^\ast, -R_\dashv^*, L_\vdash^*+R_\dashv^*, R_\vdash^*-R_\dashv^*)$.

\end{proof}

By Theorem \ref{thamc}, Propositions \ref{relcdn} and \ref{pprbeq}, we obtain the following theorem.
\begin{thm}\label{lastth}
Let $(A,\dashv,\vdash) $ be a Novikov dialgebra and $(V,l_A,r_A,\phi_A,\psi_A )$ be its representation. Let $T:V\rightarrow A$ be a linear map which is identified with  
$r_T \in V^*\otimes A\subset (A\ltimes_{{l_\star^*}_A,-r_A^*,\phi_A^*+r_A^*,\psi_A^*-r_A^*} V^*)\otimes (A\ltimes_{{l_\star^*}_A,-r_A^*,\phi_A^*+r_A^*,\psi_A^*-r_A^*} V^*)$ through $Hom_{\bf k}(V,A)\cong A\otimes V^*$. Then $r=r_T+\tau r_T$ is a symmetric solution of the CDNYBE in the Novikov dialgebra $A\ltimes_{{l_\star^*}_A,-r_A^*,\phi_A^*+r_A^*,\psi_A^*-r_A^*} V^* $ if and only if $T$ is an $\mathcal{O}$-operator on $(A,\dashv,\vdash) $ associated to  $(V,l_A,r_A,\phi_A,\psi_A )$.
\end{thm}}
Finally, we present an example to construct Novikov bi-dialgebras by symmetric solutions of the CDNYBE, and hence Leibniz conformal bialgebras.

\delete{
\begin{ex}
Let $(A={\bf k}x\oplus {\bf k}y,\dashv,\vdash)$ be a Novikov dialgebra whose nonzero-products are given by
$$
x\dashv x=ay,\qquad x\vdash x=y,
$$
for some $a\in {\bf k}$.
Let $T:A\rightarrow A$ be a linear map defined by
$T(x)=y$ and $T(y)=y$.
It is obvious that $T$ is a Rota-Baxter operator on $(A,\dashv,\vdash) $. Let $\{x^* ,y^*\}$ be the basis of $A^*$ dual to $\{x, y\}$. Then we consider the semi-direct product of $A$ and its representation $(A^*,{L_\star^*},-R_\dashv^*,L_\vdash^*+R_\dashv^*,R_\vdash^*-R_\dashv^*)$, and denote this semi-direct product Novikov dialgebra by $(A\ltimes A^* ={\bf k}x\oplus {\bf k}y            \oplus{\bf k}x^*\oplus{\bf k}y^* ,\dashv,\vdash )$.  The nonzero-products of  $A\ltimes A^*$ are given by
\begin{align*}
x\dashv x=ay,\qquad x\vdash x=y,\qquad x\vdash y^*=-(a+1)x^*,\qquad y^*\dashv x=ax^*,\qquad y^*\vdash x=(a-1)x^*,
\end{align*}
where $1$ is the identity element in {\bf k}.
 By Theorem \ref{lastth}, $r=r_T+\tau r_T=(x^*+y^*)\otimes y+y\otimes (x^*+y^*)$ is a symmetric solution of the CDNYBE in $A\ltimes A^*$. Thus by Corollary \ref{condba}, there is a Novikov bi-dialgebra  $(A\ltimes A^*,\dashv,\vdash,\delta,\Delta   )$, where nonzero co-multiplications $\delta$ and $\Delta$ are given by
\begin{align*}
\delta(x)=-a y\otimes x^*,\qquad \Delta(x)=(a+1)x^*\otimes y-(a-1)y\otimes x^*.
\end{align*}
Thus we obtain a Leibniz conformal bialgebra $(\overline{A\ltimes A^*}={\bf k}[\partial] (A\ltimes A^*),[\cdot_\lambda\cdot],\alpha)$ where the nonzero $\lambda$-bracket $ [\cdot_\lambda\cdot]$ and co-multiplication $\alpha$ are given by
\begin{align*}
&[x_\lambda x]=(a\partial +(a+1)\lambda )y,\qquad [x_\lambda y^*]=(a\partial-\lambda )x^*,\qquad [y^*_\lambda x]=(a-1)\lambda x^*,\\
&\alpha(x)=(a+1)\partial x^*\otimes y-(a-1)\partial y\otimes x^*+a x^*\otimes \partial y.
\end{align*}
\end{ex}
}

\delete{
\begin{ex}
Let $(A={\bf k}x\oplus {\bf k}y,\cdot)$ be a  perm algebra whose nonzero-products are given by
$$
x\cdot x=x,\qquad x\cdot y=y,\qquad y\cdot x=ax,\qquad y\cdot y=ay,\qquad a\neq0\in{\bf k}.
$$
Let $D$ be a derivation on $(A,\cdot)$. One has $D(x)=-b(ax-y),~ D(y)=-c(ax-y)$ where $b,c\in {\bf k}$.
Then according to the construction presented in the proof of Corollary \ref{copermlcn}, it follows that $(A,\dashv,\vdash)$ is a Novikov dialgebra whose nonzero-products are given by
\begin{align*}
y\vdash x= a x\vdash x= -ab(ax-y),\qquad y\vdash y= ax\vdash y=-ac(ax-y).
\end{align*}
Let $T: A\rightarrow A$ be a linear map defined by $T(x)=\frac{d}{a}(ax-y),~T(y)=d(ax-y)$. It is not difficult to check that $T$ is a Rota--Baxter operator on $(A,\dashv,\vdash)$. Let $\{x^*,y^*\}$ be the basis of $A^*$ dual to $\{x,y\}$. Then one has a semi-direct product $(A\ltimes A^* ={\bf k}x\oplus {\bf k}y            \oplus{\bf k}x^*\oplus{\bf k}y^* ,\dashv,\vdash )$ of $(A,\dashv,\vdash)$ and its representation  $(A^*,{L_\star^*},-R_\dashv^*,L_\vdash^*+R_\dashv^*,R_\vdash^*-R_\dashv^*)$, where the nonzero-products are given by
\begin{align*}
    &y\vdash x= a x\vdash x= -ab(ax-y),\qquad\quad\,\,\, y\vdash y= ax\vdash y=-ac(ax-y),\\
    &x^*\vdash x=-ay^*\vdash x=ab(x^*+ay^*),\qquad x^*\vdash y=-ay^*\vdash y=ac(x^*+ay^*),\\
    &x\dashv x^*=-ax \dashv y^*=a(c-ab)y^*,\qquad \,\,\, y\dashv x^*=-ay\dashv y^*=a(ab-c)x^*,\\
    &y\vdash x= ax\vdash x^*=-a y\vdash y^*=-a^2x\vdash y^* =a^2(bx^*+cy^*).
\end{align*}
By Theorem \ref{lastth},
$$
r=r_T+\tau r_T=\frac{d}{a}\left( (ay^*+x^*)\otimes (ay^*+x^*)   +(ax-y)\otimes (ay^*+x^*)     \right)
$$
is a symmetric solution of the CDNYBE in $(A\ltimes A^*,\dashv,\vdash)$.
Thus by Corollary \ref{condba}, there is a Novikov bi-dialgebra  $(A\ltimes A^*,\dashv,\vdash,\delta,\Delta  )$, where nonzero co-multiplications $\delta$ and $\Delta$ are given by
\begin{align*}
 -\delta(y)=-a\delta(x)=\tau \Delta(y)=a\tau \Delta(x)=d(c-ab)  (ax-y) \otimes (ay^*+x^*)    .
\end{align*}
Thus one  obtains a Leibniz conformal bialgebra $(\overline{A\ltimes A^*}={\bf k}[\partial] (A\ltimes A^*),[\cdot_\lambda\cdot],\alpha)$ where the nonzero $\lambda$-bracket $ [\cdot_\lambda\cdot]$ and co-multiplication $\alpha$ are given by
\begin{align*}
&[y_\lambda x]=a[x_\lambda x]=-ab \lambda (ax-y), \qquad[y_\lambda y]=a[x_\lambda y]=-ac\lambda(ax-y),\\
&[x^*_\lambda x]=-a[y^*_\lambda x]=(ac-a^2b)\partial y^*+a\lambda (bx^*+cy^*), \\
&[x^*_\lambda y ]=-a[y^*_\lambda y]= (a^2b-ac)\partial x^*+a^2\lambda (bx^*+cy^*),\\
&[y_\lambda x^*]=a[x_\lambda x^*]=-a[y_\lambda y^*]=-a^2[x_\lambda y^*]=a^2\lambda (bx^*+cy^*)\\
&\alpha(y)=a\alpha(x)=d(c-ab) (\partial(ay^*+x^*)\otimes  (ax-y)+       (ay^*+x^*)\otimes \partial (ax-y)    ).
\end{align*}
\end{ex}}

\begin{ex}
Let $A={\bf k}x\oplus {\bf k}y$ be a vector space with a basis $\{x,y\}$ and  $\{x^*,y^*\}$ be the basis of $A^*$ dual to $\{x,y\}$. Then $(B=A\oplus A^\ast, \dashv,\vdash)$ is a Novikov dialgebra with the nonzero-products given by
\begin{align*}
    &y\vdash x= x\vdash x= -b(x-y),\,\,\, y\vdash y= x\vdash y=-c(x-y),\;\;x^*\vdash x=-y^*\vdash x=b(x^*+y^*),\\
    & x^*\vdash y=-y^*\vdash y=c(x^*+y^*),\;\;x\dashv x^*=-x \dashv y^*=(c-b)y^*,\,\,\, y\dashv x^*=-y\dashv y^*=(b-c)x^*,\\
    &y\vdash x= x\vdash x^*=-y\vdash y^*=-x\vdash y^* =bx^*+cy^*,
\end{align*}
where $b$, $c\in {\bf k}$. It is straightforward to check that $r= (y^*+x^*)\otimes (x-y)   +(x-y)\otimes (y^*+x^*)$
is a symmetric solution of the CDNYBE in $(B,\dashv,\vdash)$.
Thus by Corollary \ref{condba}, there is a Novikov bi-dialgebra  $(B,\dashv,\vdash,\delta,\Delta )$, where nonzero co-multiplications $\delta$ and $\Delta$ are given by
\begin{align*}
 -\delta(y)=-\delta(x)=\tau \Delta(y)=\tau \Delta(x)=(c-b)  (x-y) \otimes (y^*+x^*)    .
\end{align*}
Thus one  obtains a Leibniz conformal bialgebra $(\bar{B}={\bf k}[\partial]B,[\cdot_\lambda\cdot],\alpha)$ where the nonzero $\lambda$-bracket $ [\cdot_\lambda\cdot]$ and co-multiplication $\alpha$ are given by
\begin{align*}
&[y_\lambda x]=[x_\lambda x]=-b \lambda (x-y), \qquad[y_\lambda y]=[x_\lambda y]=-c\lambda(x-y),\\
&[x^*_\lambda x]=-[y^*_\lambda x]=(c-b)\partial y^*+\lambda (bx^*+cy^*), \;\;[x^*_\lambda y ]=-[y^*_\lambda y]= (b-c)\partial x^*+\lambda (bx^*+cy^*),\\
&[y_\lambda x^*]=[x_\lambda x^*]=-[y_\lambda y^*]=-[x_\lambda y^*]=\lambda (bx^*+cy^*),\\
&\alpha(y)=\alpha(x)=(c-b) (\partial(y^*+x^*)\otimes  (x-y)+       (y^*+x^*)\otimes \partial (x-y)    ).
\end{align*}
\end{ex}

\noindent {\bf Acknowledgments.} This research is supported by
NSFC (12171129, 12271265, 12261131498, W2412041),  the Zhejiang
Provincial Natural Science Foundation of China (No. Z25A010006)
 and the Fundamental Research Funds for the Central Universities and Nankai Zhide Foundation.

\smallskip

\noindent
{\bf Declaration of interests. } The authors have no conflicts of interest to disclose.

\smallskip

\noindent
{\bf Data availability. } No new data were created or analyzed in this study.


\end{document}